\documentclass[12pt]{amsart}

\usepackage{fullpage}

\usepackage{modulipaper} 

\title{Fundamental groups of moduli spaces of real weighted stable curves}
\author{Jake Levinson}
\address[J. Levinson]{D\'epartement de math\'ematiques et de statistique\\
Universit\'e de Montr\'eal\\
Montr\'eal, QC
H3T 1J4 \\
Canada}
\email[J. Levinson]{jake.levinson@umontreal.ca}

\author{Haggai Liu}
\address[H. Liu]{Department of Mathematics\\
Simon Fraser University\\
Burnaby, BC  V5A 1S6 \\
Canada}
\email[H. Liu]{haggail@sfu.ca}

\date{\today}

\thanks{First author partially supported by NSERC Discovery Grant RGPIN 2021-01469.\\ \indent Second author partially supported by a GDES at Simon Fraser University.}

\begin{document}

    \begin{abstract}
        The ordinary and $S_n$-equivariant fundamental groups of the moduli space $\overline{M_{0,n+1}}(\mathbb{R})$ of real $(n+1)$-marked stable curves of genus $0$ are known as \emph{cactus groups} $J_n$ and have applications both in geometry and the representation theory of Lie algebras.
        
        In this paper, we compute the ordinary and $S_n$-equivariant fundamental groups of the Hassett space of weighted real stable curves $\overline{M_{0,\mathcal{A}}}(\mathbb{R})$ with $S_n$-symmetric weight vector $\mathcal{A} = (1/a, \ldots, 1/a, 1)$, which we call \emph{weighted cactus groups} $J_n^a$. We show that $J_n^a$ is obtained from the usual cactus presentation by introducing braid relations, which successively simplify the group from $J_n$ to $S_n \rtimes \mathbb{Z}/2\mathbb{Z}$ as $a$ increases.
        
        Our proof is by decomposing $\overline{M_{0,\mathcal{A}}}(\mathbb{R})$ as a polytopal complex, generalizing a similar known decomposition for $\overline{M_{0,n+1}}(\mathbb{R})$. In the unweighted case, these cells are known to be cubes and are `dual' to the usual decomposition into associahedra (by the combinatorial type of the stable curve). For $\overline{M_{0,\mathcal{A}}}(\mathbb{R})$, our decomposition instead consists of products of permutahedra. The cells of the decomposition are indexed by weighted stable trees, but `dually' to the usual indexing.
    \end{abstract}

    \maketitle

    \section{Introduction}\label{sec:intro}

        Let $\DMumford{n+1}$ denote the Deligne-Mumford moduli space of stable $(n+1)$-pointed curves of genus $0$. The space $\DMumford{n+1}$ is a smooth projective variety \cite{deligneM1969, etingofHKR2010}, many of whose algebraic and geometric properties are well understood, as are many topological invariants of the complex manifold $\DMumford{n+1}(\C)$.
        
        The real locus $\DMumford{n+1}(\R)$, which parametrizes stable curves whose marked points and singularities are all real, is of interest in its own right and has rather different topological features than $\DMumford{n+1}(\mathbb{C})$. For example, Kapranov \cite{kapranov1993} and Devadoss \cite{devadoss1999} showed that $\DMumford{n+1}(\R)$ admits a cell decomposition that is a tiling by $\tfrac{1}{2} n!$ associahedra, while Davis-Januszkiewicz-Scott \cite{davisJS2003} (see also \cite{ilinKLPR2023}) gave a dual tiling by cubes. The cohomology of the real moduli space is also quite different from that of $\DMumford{n+1}(\C)$\cite{etingofHKR2010}, as are its homotopy groups. Notably, it has a very interesting fundamental group known as the \textbf{pure cactus group} $PJ_n := \pi_1(\overline{M_{0,n+1}}(\R))$ \cite{davisJS2003, henriquesK2004}. It is a subgroup of the \textbf{cactus group} $J_n$, which in turn can be thought of as the $S_n$-equivariant or orbifold fundamental group $\pi_1^{S_n}(\overline{M_{0,n+1}}(\R))$, where $S_n$ acts by permuting the first $n$ marked points. (See Section \ref{subsec:G-eq-pi1} for background on $G$-equivariant fundamental groups.) Specifically, $J_n$ has generators $s_{p,q}$ for $ 1 \le p < q \le n$, with the following \textbf{cactus relations}:
        \begin{enumerate}
            \item $s_{p,q}^2=1$ for each $p<q$;
            \item $s_{p,q}s_{m,r}=s_{m,r}s_{p,q}$ if $p<q<m<r$;
            \item $s_{p,q}s_{m,r}=s_{p+q-r,p+q-m}s_{p,q}$ if $p\le m<r\le q$.
        \end{enumerate}
        There is a natural map $J_n \to S_n$ sending $s_{p,q}$ to the permutation $w_{p,q}$ that reverses the interval $p, \ldots, q$; and the kernel of this map is $PJ_n$, yielding the short exact sequence
        \[1\to PJ_n \to J_n \to S_n \to 1.\]
        The presentation of $J_n$ above closely resembles that of braid group $B_n$, the fundamental group of the unordered configuration space of $n$ points in $\C$. The subgroup $PJ_n$ is then analogous to the pure braid group $PB_n$, the fundamental group of the ordered configuration space of $n$ points in $\C$. The analogy between cactus groups and braid groups is described further in \cite{devadoss1999, morava2001}.
        Beyond its role in real moduli of curves, the cactus group arises in the representation theory of Lie algebras and operads, in real Schubert calculus and in tableau combinatorics, wherein its generators $s_{p,q}$ act by involutions related to Sch\"utzenberger's \emph{tableau evacuation} \cite{halachevaKRW2020, halachevaLLY2023, limY2023, speyer14}.

        In this paper, we consider the real loci of the moduli spaces of weighted stable curves $\DMumford{\mcA}$ introduced by Hassett \cite{hassett2003}. For such a space, we assign a vector of weights $\mcA = (a_1, \ldots, a_{n+1}) \in (0,1]^{n+1}$ to the marked points. Roughly, the Hassett space $\DMumford{\mcA}$ parametrizes curves in which marked points are allowed to coincide as long as their total weight is $\leq 1$. All Hassett spaces $\DMumford{\mcA}$ are blowdowns of $\DMumford{n+1}$ and hence are birational to $\DMumford{n+1}$; in general they are used to relate $\DMumford{n+1}$ to simpler varieties such as projective space. When at least one $a_i$ equals $1$, $\DMumford{\mcA}$ is an iterated blowup of projective space and arises as an intermediate step in Kapranov's blowup construction of $\DMumford{n+1}$ from $\Pj^{n-2}$.

        We will specifically consider Hassett spaces parameterizing curves with $S_n$-symmetric weights: $n$ points of weight $\tfrac{1}{a}$ and one point of weight $1$, i.e. such that up to $a$ of the first $n$ points can coincide at a time. The $S_n$-action on $\DMumford{n+1}$ (fixing the $n+1$-st point) descends to such a Hassett space, and so we may define a weighted variant $J_n^a$ of the cactus group $J_n$.

        \subsection{Statement of Results}\label{subsec:stmt-of-results}
        
            We define weighted variants of $J_n$ and $PJ_n$, namely the ordinary and $S_n$-equivariant fundamental groups
            \[J_n^a := \pi_1^{S_n}(\overline{M_{0,\mathcal{A}}}(\R)),
            \quad
            PJ_n^a := \pi_1(\overline{M_{0,\mathcal{A}}}(\R))
            \]
            for $a=1,2,\hdots, n-1$ and $\mathcal{A} = \mcA(a) := (\frac{1}{a}, \ldots, \frac{1}{a}, 1)\in \Q^{n+1}$, which we call the \emph{weighted cactus group} and \emph{pure weighted cactus group}, respectively. (See Section~\ref{subsec:G-eq-pi1} for the definition of the $G$-equivariant fundamental group.)  Our main theorem generalizes the above presentation of the cactus group. We note that for $a=1$ or $2$, $\hassett{a}(\R)$ is isomorphic to $\DMumford{n+1}(\R)$, so $J_n^a$ is just the ordinary cactus group.

            \begin{thm}\label{thm:main-thm-pres-of-Jna}
                Let $a\ge 3$. The weighted cactus group $J_n^a$ has generators $s_{p,q}, 1\le p<q\le n$ satisfying either $q-p\in \{a,a+1, \ldots, n-1\}$ or $q-p=1$, 
                with the relations:
                \begin{enumerate}
                    \item $s_{p,q}^2=1$ for each $p<q$; \label{relation:involution}
                    \item $s_{p,q}s_{m,r}=s_{m,r}s_{p,q}$ if $p<q<m<r$;
                    \label{relation:commuting}
                    \item $s_{p,q}s_{m,r}=s_{p+q-r,p+q-m}s_{p,q}$ if $p\le m<r\le q$;
                    \label{relation:cactus}
                    \item (Braid relations) $(s_{i,i+1}s_{i+1,i+2})^3=1$ for $1\le i\le n-2$. 
                    \label{relation:braid}
                \end{enumerate}
            \end{thm}
            Let $w_{p,q}\in S_n$ be the permutation that reverses the interval $p,\ldots,q$. For each $a$, there is a short exact sequence
            \begin{equation}
            \label{eqn:SES-wted-Jna}
            1 \to PJ_n^a \to J_n^a \to S_n \to 1
            \end{equation}
            sending $s_{p,q} \mapsto w_{p,q}$ as in the unweighted case.

            The relations \eqref{relation:involution}, \eqref{relation:commuting}, and \eqref{relation:cactus} are the usual cactus relations on the restricted set of generators. 
            For $n\ge 3$ and $a = n-1$, $\hassett{a}(\R)$ is the real projective space $\R\Pj^{n-2}$ and Eq. \eqref{eqn:SES-wted-Jna} becomes
            \begin{equation} \label{eqn:RPn-semidirect}
                1 \to \Z/2\Z \to S_n \rtimes \Z/2\Z \to S_n \to 1,
            \end{equation}
            where the action of $\Z/2\Z$ is given by conjugating by $w_{1,n}$.

            For each $a$, the map $J_n^a \to S_n$ maps the subgroup generated by the elements 
            $$\{s_{i,i+1} : 1 \leq i < n\}$$
            surjectively to $S_n$. Thus, by the braid relations \eqref{relation:braid}, these elements generate a copy of $S_n \subseteq J_n^a$ as soon as $a \geq 3$. It follows that the natural surjection $J_n \onto J_n^a$ induced by the blowdown morphism $\overline{M_{0,n+1}} \to \hassett{a}$ can be described by
            \begin{equation}\label{eqn:surjection-of-cactus-groups}
            s_{p,q} \mapsto
            \begin{cases}
            s_{p,q} \in J_n^a & \text{if } q-p+1 > a, \\
            w_{p,q} \in S_n \subseteq J_n^a & \text{if } q-p+1 \leq a.
            \end{cases}
            \end{equation}
            This map is surjective on fundamental groups, so taking $a$ from $1$ to $n-1$ gives a sequence of quotients
            \begin{equation}\label{eqn:sequence-quotients}
                J_n = J_n^1 = J_n^2 \twoheadrightarrow J_n^3 \twoheadrightarrow \cdots \twoheadrightarrow J_n^{n-1} = S_n \rtimes \Z/2\Z.
            \end{equation}

    \begin{rmk}[Presentations of $PJ_n^a$]
        Unlike for the pure braid group, a presentation of the pure cactus group is unknown, even for the unweighted moduli space, although generators for $PJ_n$ as a normal subgroup of $J_n$ were given in \cite{etingofHKR2010}. It would be interesting to find presentations for $PJ_n^a$ for all $a$, particularly since the group simplifies as $a$ grows.
    \end{rmk}

        Our proof of Theorem~\ref{thm:main-thm-pres-of-Jna} relies on constructing a cellular decomposition of $\hassett{a}(\R)$ into polytopes, which is dual to the standard decomposition by the topological type of the curve. As mentioned above, $\DMumford{n+1}(\R)$ admits a dual (Coxeter)~\cite{davisJS2003} cell decomposition into (essentially) cubes, which is closely related to the above presentation of the cactus group. We construct a similar polytopal decomposition for $\hassett{a}(\R)$, by merging various cells along the blowdown maps from $\DMumford{n+1}(\R)$. We note that in $\DMumford{n+1}(\R)$, these unions of cells (cubes) are typically non-contractible. Hence, the main difficulty is in showing that the resulting unions are once again cells after blowing down (in particular, that they are contractible). The resulting decomposition has cells indexed by labeled trees, and the cells turn out to be combinatorially equivalent to products of permutahedra. 

        Both the standard and dual decompositions of $\hassett{a}(\R)$ are indexed by the set $\StRtree([n]; a)$ of \emph{$a$-stable trees} (see Section \ref{sec:a-stable-trees}), which are certain trees $\tau$ with leaves labeled by nonempty subsets $A_\ell \subseteq [n]$ of size $\leq a$, forming a set decomposition $\coprod_\ell A_\ell = [n]$.
        For each $(C; x_\bullet) \in \hassett{a}(\R)$, we run a \emph{distance algorithm} (see Definitions \ref{def:dist-alg-diff-to-rtree} and \ref{alg:a-weighted-diff-vec-to-rtree}) which yields an $a$-stable tree $\tau^{\mathrm{dual}}(C;x_\bullet) \in \StRtree([n]; a)$. For each $a$-stable tree $\tau$, we define the locally closed \emph{dual cell} $W_\tau$ to be the set of curves for which $\tau^{\mathrm{dual}}(C; x_\bullet) = \tau$.

        Let $\Pi_k\subseteq \R^{k}$ denote the $(k-1)$-dimensional permutahedron (see Equation \eqref{eqn:def-permutahedron}). The dual decomposition is built as follows.

        \begin{thm} \label{thm:intro-dual-cell-decomposition}
            There is a polytopal decomposition 
            $$\hassett{a}(\R) = \coprod_{\tau} W_\tau$$
            where $\tau$ ranges over the $a$-stable trees. If $\tau$ has $d$ internal edges and leaves labeled $A_1, \ldots, A_k$, where $\coprod_{i=1}^k A_i = [n]$, the closure $\overline{W_\tau}$ has the form
            \[
            \overline{W_\tau} \cong [-1, 1]^d \times \prod_{i=1}^k \Pi_{|A_i|}.
            \]
        \end{thm}

\begin{figure}
    \centering
    $$
    \vcenter{
    \hbox{
    \scalebox{0.8}{
\begin{tikzpicture}
    \def \radius {3.5}
    \begin{scope}[shift={(0, 0)}, scale=0.7]

    \def \smallradius {1}
     \fill[cyan!20] (0:\radius) -- (0:\smallradius) -- (60:\smallradius) -- (60:\radius) -- cycle;
     \fill[yellow!50] (60:\radius) -- (60:\smallradius) -- (120:\smallradius) -- (120:\radius) -- cycle;
     \fill[magenta!20] (120:\radius) -- (120:\smallradius) -- (180:\smallradius) -- (180:\radius) -- cycle;
     \fill[cyan!20] (180:\radius) -- (180:\smallradius) -- (240:\smallradius) -- (240:\radius) -- cycle;
     \fill[yellow!50] (240:\radius) -- (240:\smallradius) -- (300:\smallradius) -- (300:\radius) -- cycle;
     \fill[magenta!20] (300:\radius) -- (300:\smallradius) -- (360:\smallradius) -- (360:\radius) -- cycle;

    \draw (\radius, 0) node {$\bullet$}
    	-- (60:\radius) node {$\bullet$}
	-- (120:\radius) node {$\bullet$}
	-- (180:\radius) node {$\bullet$}
	-- (240:\radius) node {$\bullet$}
	-- (300:\radius) node {$\bullet$}
	-- (360:\radius) node {$\bullet$};

    \foreach \a in {0, 60, 120, 180, 240, 300}
        \draw (\a:\radius) -- (\a:\smallradius);
        
   \begin{scope}[shift={(271:\radius*0.79)}]
   \draw (0, 0) -- (-0.3, 0.2);
   \draw (-0.3, 0.2) -- (-0.6, 0.5) node[above] {$23$};
   \draw (-0.3, 0.2) -- (0, 0.5) node[above] {$1$};
   \draw (0, 0) -- (0.6, 0.5) node[above] {$4$};
   \end{scope}

   \begin{scope}[shift={(221:\radius*0.7)}]
   \draw (0, 0) -- (-0.3, 0.2);
   \draw (-0.3, 0.2) -- (-0.6, 0.5) node[above] {$2$};
   \draw (-0.3, 0.2) -- (0, 0.5) node[above] {$13$};
   \draw (0, 0) -- (0.6, 0.5) node[above] {$4$};
   \end{scope}

   \begin{scope}[shift={(319:\radius*0.7)}]
   \draw (0, 0) -- (-0.3, 0.2);
   \draw (-0.3, 0.2) -- (-0.6, 0.5) node[above] {$3$};
   \draw (-0.3, 0.2) -- (0, 0.5) node[above] {$12$};
   \draw (0, 0) -- (0.6, 0.5) node[above] {$4$};
   \end{scope}   
   \end{scope}

   \draw (3.5, 0) node {$\longrightarrow$};

    \begin{scope}[shift={(8, 0)}, scale=0.7]

     \fill[cyan!40!lime!20] (0:\radius) -- (60:\radius) -- (120:\radius) -- (180:\radius) -- (240:\radius) -- (300:\radius) -- cycle;
    
    \draw (\radius, 0) node {$\bullet$}
    	-- (60:\radius) node {$\bullet$}
	-- (120:\radius) node {$\bullet$}
	-- (180:\radius) node {$\bullet$}
	-- (240:\radius) node {$\bullet$}
	-- (300:\radius) node {$\bullet$}
	-- (360:\radius) node {$\bullet$};

    \draw (0, -0.5) -- (-0.5, 0) node[above] {$123$};
    \draw (0, -0.5) -- (0.5, 0) node[above] {$4$};

   \begin{scope}[shift={(135:\radius*1.1)}]
   \foreach \x/\i in {-0.6/1, -0.2/2, 0.2/3, 0.6/4}
    	\draw (0, 0) -- (\x, 0.5) node[above] {$\i$};
   \end{scope}
   \begin{scope}[shift={(45:\radius*1.1)}]
   \foreach \x/\i in {-0.6/1, -0.2/3, 0.2/2, 0.6/4}
    	\draw (0, 0) -- (\x, 0.5) node[above] {$\i$};
   \end{scope}
   \begin{scope}[shift={(180:\radius*1.25)}]
   \foreach \x/\i in {-0.6/2, -0.2/1, 0.2/3, 0.6/4}
    	\draw (0, 0) -- (\x, 0.5) node[above] {$\i$};
   \end{scope}
   \begin{scope}[shift={(0:\radius*1.25)}]
   \foreach \x/\i in {-0.6/3, -0.2/1, 0.2/2, 0.6/4}
    	\draw (0, 0) -- (\x, 0.5) node[above] {$\i$};
   \end{scope}
   \begin{scope}[shift={(230:\radius*1.4)}]
   \foreach \x/\i in {-0.6/2, -0.2/3, 0.2/1, 0.6/4}
    	\draw (0, 0) -- (\x, 0.5) node[above] {$\i$};
   \end{scope}
   \begin{scope}[shift={(310:\radius*1.4)}]
   \foreach \x/\i in {-0.6/3, -0.2/2, 0.2/1, 0.6/4}
    	\draw (0, 0) -- (\x, 0.5) node[above] {$\i$};
   \end{scope}
   \begin{scope}[shift={(270:\radius*1.25)}]
   \foreach \x/\i in {-0.6/23, 0/1, 0.6/4}
    	\draw (0, 0) -- (\x, 0.5) node[above] {$\i$};
   \end{scope}
   \end{scope}
\end{tikzpicture}
} 
}}
$$
\caption{
A local picture of the blowdown map $\DMumford{\mcA(2)}(\R) \to \DMumford{\mcA(3)}(\R)$, with $5$ marked points. Left: three square regions homeomorphic to $[-1, 1] \times \Pi_2$, each indexed by a $2$-stable tree. Note that on the inner hexagon, antipodal points are identified, giving a copy of $\mathbb{RP}^1$ (the exceptional divisor). 
Right: Contracting the inner boundary to a point and merging the three cells yields a copy of $\Pi_3 \subseteq \hassett{3}(\R)$ indexed by a $3$-stable tree. On the right, we also show vertex labels and one illustrative edge label, which do not change. The interior edges on the left are labeled by $2$-stable trees whose \emph{$a$-compression} is the $3$-stable tree labeling the entire $\Pi_3$: 
{\protect
\begin{tikzpicture}[scale=0.7]
    \protect\draw (0, 0) -- (-0.3, 0.2);
    \protect\draw (-0.3, 0.2) -- (-0.6, 0.5) node[above] {$1$};
    \protect\draw (-0.3, 0.2) -- (-0.3, 0.5) node[above] {$2$};
    \protect\draw (-0.3, 0.2) -- (0, 0.5) node[above] {$3$};
    \protect\draw (0, 0) -- (0.6, 0.5) node[above] {$4$};
\protect\end{tikzpicture}, 
\protect
\begin{tikzpicture}[scale=0.7]
    \protect\draw (0, 0) -- (-0.3, 0.2);
    \protect\draw (-0.3, 0.2) -- (-0.6, 0.5) node[above] {$1$};
    \protect\draw (-0.3, 0.2) -- (-0.3, 0.5) node[above] {$3$};
    \protect\draw (-0.3, 0.2) -- (0, 0.5) node[above] {$2$};
    \protect\draw (0, 0) -- (0.6, 0.5) node[above] {$4$};
\protect\end{tikzpicture}, 
    \protect\begin{tikzpicture}[scale=0.7]
    \protect\draw (0, 0) -- (-0.3, 0.2);
    \protect\draw (-0.3, 0.2) -- (-0.6, 0.5) node[above] {$2$};
    \protect\draw (-0.3, 0.2) -- (-0.3, 0.5) node[above] {$1$};
    \protect\draw (-0.3, 0.2) -- (0, 0.5) node[above] {$3$};
    \protect\draw (0, 0) -- (0.6, 0.5) node[above] {$4$};
\protect\end{tikzpicture}}.}
\label{fig:hexagon-merger}
\end{figure}
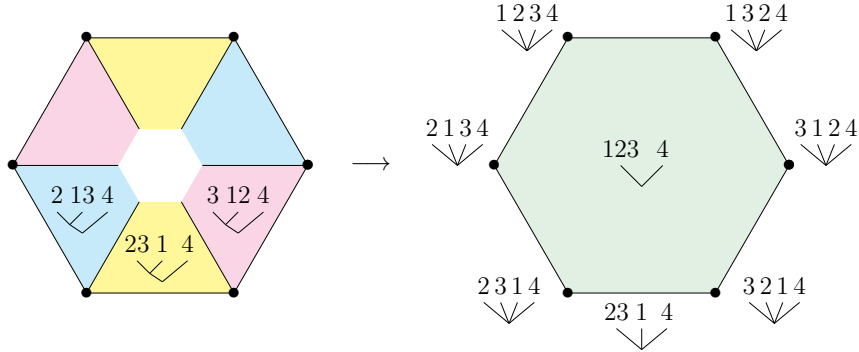

        Figure~\ref{fig:hexagon-merger} illustrates Theorem~\ref{thm:intro-dual-cell-decomposition} for $n=4, a=3$, where three cubes (each shaped like an hourglass) merge into a $\Pi_3$ cell. See Figure~\ref{fig:illust-of-3d-permutahedron} for a more complicated example in three dimensions, showing the decomposition of a three-dimensional permutahedron.

        Our results thus shed light both on the real geometry of the Hassett space and on the algebraic structure of the standard cactus group $J_n$ as an iterated extension of the symmetric group.
        
        The remainder of the paper is structured as follows. In Section \ref{sec:background}, we discuss the relevant classes of trees, stable curves and moduli spaces and give some background on $G$-equivariant fundamental groups. We construct the cell decomposition of Theorem~\ref{thm:intro-dual-cell-decomposition} in Section~\ref{sec:dual-cell}. Finally, we read off the presentation $J_n^a$ in Theorem \ref{thm:main-thm-pres-of-Jna} by examining the $2$-skeleton of the cell complex.

        \begin{rmk}[Double Covers]
            The factor of $\Z/2\Z$ in \eqref{eqn:sequence-quotients} appears because $\DMumford{n+1}(\R)$ allows orientation-reversing changes of coordinates on $\mathbb{P}^1$. These factors disappear when considering 
            the natural double cover $\hassettDC{a}(\R) \onto \hassett{a}(\R)$ given by orienting the curve at its $(n+1)$-st marked point. We discuss this in Section~\ref{subsec:cubes-in-2:1-cover-unweighted} for the unweighted case $a=1,2$.
        \end{rmk}

    \subsection{Acknowledgments}

        The authors thank Joel Kamnitzer for helpful conversations on $G$-equivariant fundamental groups.

    \section{Setup and Background}\label{sec:background}

    \subsection{Stable and $a$-stable trees}
    \label{subsec:rooted-trees-defs}
    
    Let $A\neq \varnothing$ be a finite set.
    A \textit{rooted ordered tree} (or simply a \textit{rooted tree}, also referred to as a \emph{planar tree}) with leaf set $A$ is a rooted tree $\tau$ with, for each vertex, an ordering on its set of children, such that $\tau$ has exactly $|A|$ leaves labeled by the elements in $A$, and each internal vertex has $\ge 2$ children. We use elements of $A$ to refer (by abuse of notation) to both the leaf vertex and its label. For an internal edge $e\in\tau$, we write $e=(u,w)$, where $u,w$ are the vertices adjacent to $e$, with $u$ as the parent of $e$ and $w$ as the child of $e$. We write $V(\tau), E(\tau)$ for the vertex and edge set of $\tau$, respectively, and $V_{\Int}(\tau), E_{\Int}(\tau)$ for the sets of internal vertices and internal edges. We write $c(v)$ for the number of children of the vertex $v$. We equip $V(\tau)$ with the partial order with the root as the greatest element and the leaves as the minima:
    \begin{equation}\label{eqn:rtree-partial-order}
        v\le w \iff \text{the unique path from $v$ to the root contains $w$.}
    \end{equation}
    We write $\Rtree(A)$ for the set of all such trees on $A$. If $A=[n]$ and $\sigma \in S_n$ is a permutation, we write
    $$\Rtree^\sigma([n]) \subseteq \Rtree([n])$$
    for the set of $\tau$ for which the leaf labels, in order from left to right, are $\sigma^{-1}(1), \ldots, \sigma^{-1}(n).$
    
    We next consider trees in which some or all the internal vertices are `flippable'. Given a tree $\tau\in \Rtree(A)$ and a vertex $v \in V_{\Int}(\tau)$, the \emph{flip of $\tau$ at $v$} is given by reversing the ordering of the children at every vertex in the subtree rooted at $v$, including $v$ itself. 
    Let $\rflip: \Rtree(A)\to \Rtree(A)$ be the involution that maps a tree to its flip at the root.
    
    A \emph{stable tree} (on $A$) is an equivalence class of trees $\tau\in \Rtree(A)$, considered up to flips at all internal vertices. We write $\StRtree(A)$ for the set of stable trees on $A$. 
    
    A \emph{refined stable tree} (on $A$) is the data of a rooted ordered tree $\tau$ on $A$, together with a subset $F\subseteq V_{\Int}(\tau)$ of its vertices that are called \emph{flippable}, which is required to include the root. We consider two such trees equivalent if they differ by flips at flippable vertices and have the same set of flippable vertices. We write $\RStRtree(A)$ for the set of refined stable trees on $A$ up to equivalence.
    
    We will see that stable trees index both the standard and dual stratifications of $\overline{M_{0,n+1}}(\R)$, while refined stable trees index the cells of the common refinement of the standard and dual stratifications.

    \begin{rmk}
        There is a map 
        \begin{equation} \label{eqn:map-r-rst-tree}
            \Rtree(A) \to \RStRtree(A)
        \end{equation}
        by marking the root vertex (only) as flippable. These refined stable trees will correspond to refined cells contained in $M_{0,n}$ (i.e. irreducible curves).
    \end{rmk}

    \begin{rmk}\label{rmk:rst-to-st-trees}
        There are two maps from refined stable trees to stable trees. First, we have
        \begin{equation} \label{eqn:map-std-rst-st-tree}
        \RStRtree(A) \to \StRtree(A)
        \end{equation}
        by contracting all edges $(v, w)$ for which the child $w$ is non-flippable. Second, we have
        \begin{equation} \label{eqn:map-dual-rst-st-tree}
        \RStRtree(A) \to \StRtree(A)
        \end{equation}
        by marking every vertex as flippable. We will see that, in terms of cell structures, these maps assign to each refined cell the standard and dual cell, respectively, in which it lies.
    \end{rmk}

    \subsubsection{Weighted stable trees}
    \label{sec:a-stable-trees}

    A \emph{composition} of the set $A$ is a sequence $A_\bullet = (A_1, \ldots, A_r)$ of nonempty, pairwise disjoint subsets of $A$ whose union is $A$. If each $|A_i| \leq k$, we say $A_\bullet$ is a \emph{$k$-composition}. A composition $B_\bullet$ \emph{refines} $A_\bullet$ if $A_\bullet$ can be obtained from $B_\bullet$ by merging consecutive parts of $B_\bullet$, and we call $B_\bullet$ a \emph{proper refinement} if $A_\bullet \ne B_\bullet$. We write $P(A_\bullet)$ for $A_\bullet$ considered as an (unordered) set partition of $A$. We say a composition $A_\bullet$ of $[n]$ is \emph{compatible} with a permutation $\sigma \in S_n$ if $(\{\sigma(1)\}, \ldots, \{\sigma(n)\})$ refines $A_\bullet$.

    Let $\tau$ be a rooted tree with leaves labeled by subsets of $[n]$, such that the leaf labels form a composition $A_\bullet$ of $[n]$. Let $a \in \{1, \ldots, n-1\}$. We say $\tau$ is \emph{$a$-stable} if
    \begin{enumerate}
        \item[(i)] $A_\bullet$ is an $a$-composition, and
        \item[(ii)] For all $v \in V_{\Int}(\tau)$, if the subtree of $\tau$ rooted at $v$ has leaves labeled $A_i, \ldots, A_j$, then $|A_i| + \cdots + |A_j| \geq a+1$.
    \end{enumerate}
    We define $\Rtree(a; [n])$ as the set of all such trees, and we define $\StRtree(a; [n])$ and $\RStRtree(a; [n])$ as above, respectively, considering $\tau$ up to flips at all vertices, or at a specified set $F \subseteq V_{\Int}(\tau)$ of flippable vertices including the root. For $a=1$, we recover the sets above.

    There are \emph{$a$-compression} maps
    \begin{align*}
        \varpi_a : \Rtree([n]) &\to \Rtree(a; [n]), \\
        \varpi_a : \RStRtree([n]) &\to \RStRtree(a; [n]), \\
        \varpi_a : \StRtree([n]) &\to \StRtree(a; [n]),
    \end{align*}
    defined by contracting the subtree rooted at $v \in V_{\Int}(\tau)$ whenever that subtree contains $\leq a$ leaves. In doing so, $v$ becomes a leaf, labeled by the union of the corresponding leaves. In the same way, we have compression maps $\varpi_{a, b}$ from $a$-stable to $b$-stable trees, whenever $a \leq b$, by contracting subtrees that do not satisfy condition (ii) of $b$-stability. Figure~\ref{fig:ex-of-tree-compressing-map} gives an example demonstrating the maps $\varpi_a$ and $\varpi_{a,b}$.

    \begin{figure}[h]
    
        \centering
        $$
        \vcenter{
        \hbox{
            \begin{tikzpicture}[scale=0.5]
            {
            \node [circle, draw, fill, inner sep =1pt, label = { 90: 1 }] (node_0_0) at (0,0) {};
            \node [circle, draw, fill, inner sep =1pt, label = { 90: 2 }] (node_1_0) at (1,0) {};
            \node [circle, draw, fill, inner sep =1pt, label = { 90: 3 }] (node_2_0) at (2,0) {};
            }

            {
            \node [circle, draw, fill, inner sep =1pt, label = { 90: 4 }] (node_3_0) at (3,0) {};
            \node [circle, draw, fill, inner sep =1pt, label = { 90: 5 }] (node_4_0) at (4,0) {};
            }
            
            {
            \node [circle, draw, fill, inner sep =1pt, label = { 90: 6 }] (node_5_0) at (5,0) {};
            \node [circle, draw, fill, inner sep =1pt, label = { 90: 7 }] (node_6_0) at (6,0) {};
            \node [circle, draw, fill, inner sep =1pt, label = { 90: 8 }] (node_7_0) at (7,0) {};
            \node [circle, draw, fill, inner sep =1pt, label = { 90: 9 }] (node_8_0) at (8,0) {};
            }
            
            
            \node [circle, draw, fill, inner sep =1pt, label = { 90:  }] (node_1_-1) at (1,-1) {};

            \node [circle, draw, fill, inner sep =1pt, label = { 90:  } ] (node_3'5_-1) at (3.5,-1) {};
            \node [circle, draw, fill, inner sep =1pt, label = { 90: }] (node_6'5_-1) at (6.5,-1) {};

            \node [circle, draw, fill, inner sep =1pt, label = { 90: }] (node_5'0_-2) at (5.0,-2) {};
            
            \node [circle, draw, fill, inner sep =1pt, label = { 90:  }] (node_3'0_-3) at (3.0,-3) {};
            
            \draw[] (node_3'0_-3) -- (node_1_-1);
            \draw[] (node_3'0_-3) -- (node_5'0_-2);
            
            {
            \draw[] (node_1_-1) -- (node_0_0);
            \draw[] (node_1_-1) -- (node_1_0);
            \draw[] (node_1_-1) -- (node_2_0);
            }
            
            {
            \draw[] (node_5'0_-2) -- (node_3'5_-1);
            \draw[] (node_5'0_-2) -- (node_6'5_-1);
            }
            
            {
            \draw[] (node_3'5_-1) -- (node_3_0);
            \draw[] (node_3'5_-1) -- (node_4_0);
            }
            
            {
            \draw[] (node_6'5_-1) -- (node_5_0);
            \draw[] (node_6'5_-1) -- (node_6_0);
            \draw[] (node_6'5_-1) -- (node_7_0);
            \draw[] (node_6'5_-1) -- (node_8_0);
            }
            
            \end{tikzpicture}
        }}\xrightarrow{\varpi_3}
        \vcenter{
        \hbox{
            \begin{tikzpicture}[scale=0.6]

            {
            \node [circle, draw, fill, inner sep =1pt, label = { 90: $\{6\}$ }] (node_5_0) at (5,0) {};
            \node [circle, draw, fill, inner sep =1pt, label = { 90: $\{7\}$ }] (node_6_0) at (6,0) {};
            \node [circle, draw, fill, inner sep =1pt, label = { 90: $\{8\}$ }] (node_7_0) at (7,0) {};
            \node [circle, draw, fill, inner sep =1pt, label = { 90: $\{9\}$ }] (node_8_0) at (8,0) {};
            }
            
            
            \node [circle, draw, fill, inner sep =1pt, label = { 90:  $\{1,2,3\}$}] (node_1_-1) at (1,-1) {};

            \node [circle, draw, fill, inner sep =1pt, label = { 90:  $\{4,5\}$} ] (node_3'5_-1) at (3.5,-1) {};
            \node [circle, draw, fill, inner sep =1pt, label = { 90: }] (node_6'5_-1) at (6.5,-1) {};

            \node [circle, draw, fill, inner sep =1pt, label = { 90: }] (node_5'0_-2) at (5.0,-2) {};
            
            \node [circle, draw, fill, inner sep =1pt, label = { 90:  }] (node_3'0_-3) at (3.0,-3) {};
            
            \draw[] (node_3'0_-3) -- (node_1_-1);
            \draw[] (node_3'0_-3) -- (node_5'0_-2);

            {
            \draw[] (node_5'0_-2) -- (node_3'5_-1);
            \draw[] (node_5'0_-2) -- (node_6'5_-1);
            }

            {
            \draw[] (node_6'5_-1) -- (node_5_0);
            \draw[] (node_6'5_-1) -- (node_6_0);
            \draw[] (node_6'5_-1) -- (node_7_0);
            \draw[] (node_6'5_-1) -- (node_8_0);
            }
            
            \end{tikzpicture}
        }}\xrightarrow{\varpi_{7,3}}
        \vcenter{
        \hbox{
            \begin{tikzpicture}[scale=0.6]

            
            \node [circle, draw, fill, inner sep =1pt, label = { 90:  $\{1,2,3\}$}] (node_1_-1) at (1,-1) {};

            \node [circle, draw, fill, inner sep =1pt, label = { 90: $\{4,5,6,7,8,9\}$}] (node_5'0_-2) at (4.7,-2) {};
            
            \node [circle, draw, fill, inner sep =1pt, label = { 90:  }] (node_3'0_-3) at (3.0,-3) {};
            
            \draw[] (node_3'0_-3) -- (node_1_-1);
            \draw[] (node_3'0_-3) -- (node_5'0_-2);
            
            \end{tikzpicture}
        }}
        $$
        \caption{Example of tree compression. Observe that $\varpi_7 = \varpi_{7,3}\circ \varpi_3$.}
        \label{fig:ex-of-tree-compressing-map}
    \end{figure}
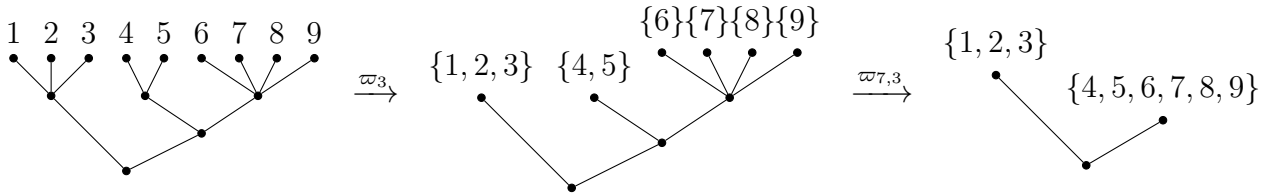

    \begin{rmk}
        If $e \in E_{\Int}(\tau)$ is an internal edge whose child vertex $v$ is flippable, there are in general two ways two contract $e$ to obtain a smaller tree $\tau'$; depending on whether or not we reverse the children of $v$ prior to performing the contraction. We will say that both possible trees $\tau'$ are obtained ``by contracting an edge of $\tau$''.
    \end{rmk}

    \subsection{Curves and moduli spaces} \label{subsec:curves-moduli-spaces}

    An \emph{$(n+1)$-pointed prestable curve of genus zero} is a reduced, connected algebraic curve $C$ of arithmetic genus $0$, together with an $n$-tuple $x_{\bullet} = (x_1, \ldots, x_{n+1})$ of smooth \textit{marked points} of $C$. The genus condition is equivalent to every irreducible component of $C$ being isomorphic to $\mathbb{P}^1$, with the components attached along simple nodes to form a tree structure. We say a point on $C$ is \textit{special} if it is a marked point or a node. Such a curve is \emph{stable} if all the marked points are distinct and every irreducible component contains at least $3$ special points. We write $\DMumford{n+1}$ for the moduli space of stable $(n+1)$-pointed curves of genus $0$.

    A stable curve is \emph{real}, $(C; x_\bullet) \in \DMumford{n+1}(\R)$, if all its marked points and nodes are at real points of $\Pj^1$. Its \emph{standard tree}\footnote{In the literature, this tree is called the `dual tree', because its edges correspond to the special points of $(C; x_\bullet)$. We instead reserve the term `dual' for the dual cell structure.} $\tau^\mathrm{std}(C; x_\bullet) \in \StRtree([n])$ is defined as follows. It has an internal vertex $v$ for each irreducible component $C_v \subseteq C$ and an internal edge $(v, w)$ when $C_v$ and $C_w$ are attached at a node; and it has leaves corresponding to the marked points $x_1, \ldots, x_n$ and edges $(v, i)$ when $x_i \in C_v$. We root the tree at the component containing $x_{n+1}$. Finally, we order the children of each $v \in V_{\Int}(\tau)$, up to a flip at each $v$, as follows.
    \begin{defn}\label{def:rooted-coordinates}
        We define \emph{rooted coordinates} on $(C; x_\bullet)$ as follows. If $v$ is not the root, we define $q(C_v) \in C_v$ to be the node corresponding to the unique edge from $v$ back towards the root. If $v$ is the root, we define $q(C_v) := x_{n+1}$. Then, on each $C_v$, we choose any coordinates such that $q(C_v) = \infty$.\footnote{It is important to note that this choice of coordinates applies on $C_v$, and not on the other component containing $q(C_v)$.}
    \end{defn}
    
    The special points on $C_v$ other than $x_{n+1}$ or $q(C_v)$ thus correspond to the children of $v$. They inherit a well-defined ordering along $\R\Pj^1$, up to a flip, by taking any choice of rooted coordinates.
    
    For each $\tau\in \StRtree([n])$, the \emph{standard cell} $X_{\tau}\subseteq \overline{M_{0,n+1}}(\R)$ consists of the stable curves with $\tau$ as the standard tree. See Figure~\ref{fig:stable-curve-with-dual-tree-ex} for an example. We note that
    \begin{equation}\label{eqn:dimn-formula-for-stdcell-Xtau}
    \dim X_\tau = \sum_{v \in V_{\Int}(\tau)} (c(v) - 2).
    \end{equation}
    It is well-known that $X_\tau$ is homeomorphic to a product of associahedra; see e.g. \cite{devadoss1999}.

    Hassett~\cite{hassett2003} constructed a weighted variant of this space, denoted $\overline{M_{0,\mcA}}$, where 
    $$\mcA=(a_1, \hdots, a_{n+1}) \in (0,1]^{n+1},$$ is a weight vector and satisfies $\sum a_i>2$. The $i$-th marked point is assigned the weight $a_i$ and, roughly, marked points are allowed to coincide precisely when their weights sum to $\leq 1$. Formally, a prestable $(n+1)$-pointed curve $(C;x_{{\bullet}})$ is \emph{$\mcA$-stable} if for each smooth point $p\in C$, we have $\sum_{x_i=p}a_i\le 1$, and for each irreducible component $C' \subseteq C$,
    \[
    \#\text{nodes of } C' + \sum_{x_i \in C'} a_i > 2.
    \]
    The \textit{Hassett space} $\overline{M_{0, \mcA}}$ is the moduli space of $\mcA$-stable curves up to equivalence. If $\mcB\le \mcA$, there is a reduction (blowdown) map $\rho_{\mcA, \mcB}: \overline{M_{0,\mcA}}\to \overline{M_{0,\mcB}}$, where $\rho_{\mcA, \mcB}(C; x_\bullet)$ is formed by successively contracting components of $C$ that are $\mcB$-unstable.

    \begin{figure}[h]
        \centering
        $\vcenter{\hbox{
        \begin{tikzpicture}[xscale=0.75, yscale= 1]
            \node [circle, draw, fill, inner sep =1pt, label = { 90: 1 }] (node_0_0) at (0,0) {};
            \node [circle, draw, fill, inner sep =1pt, label = { 90: 2 }] (node_1_0) at (1,0) {};
            \node [circle, draw, fill, inner sep =1pt, label = { 90: 3 }] (node_2_0) at (2,0) {};
            \node [circle, draw, fill, inner sep =1pt, label = { 90: 4 }] (node_3_0) at (3,0) {};
            \node [circle, draw, fill, inner sep =1pt, label = { 90: 5 }] (node_4_0) at (4,0) {};
            \node [circle, draw, fill, inner sep =1pt, label = { 90: 6 }] (node_5_0) at (5,0) {};
            \node [circle, draw, fill, inner sep =1pt, label = { 90: 7 }] (node_6_0) at (6,0) {};
            \node [circle, draw, fill, inner sep =1pt, label = { 90: 8 }] (node_7_0) at (7,0) {};
            \node [circle, draw, fill, inner sep =1pt, label = { 90: 9 }] (node_8_0) at (8,0) {};
            \node [circle, draw, fill, inner sep =1pt, label = { 90: 10 }] (node_9_0) at (9,0) {};
            \node [circle, draw, fill, inner sep =1pt, ] (node_1_-1) at (1,-1) {};
            \node [circle, draw, fill, inner sep =1pt, ] (node_6_-1) at (6,-1) {};
            \node [circle, draw, fill, inner sep =1pt, ] (node_6_-2) at (5.5,-2) {};
            \node [circle, draw, red, fill, inner sep =1pt, minimum size=1.5mm] (node_4'75_-3) at (4.75,-3) {};
            \draw[blue,thick] (node_4'75_-3) -- (node_1_-1);
            \draw[] (node_4'75_-3) -- (node_3_0);
            \draw[blue,thick] (node_4'75_-3) -- (node_6_-2);
            \draw[] (node_4'75_-3) -- (node_9_0);
            \draw[] (node_1_-1) -- (node_0_0);
            \draw[] (node_1_-1) -- (node_1_0);
            \draw[] (node_1_-1) -- (node_2_0);
            \draw[] (node_6_-2) -- (node_4_0);
            \draw[blue, thick] (node_6_-2) -- (node_6_-1);
            \draw[] (node_6_-2) -- (node_8_0);
            \draw[] (node_6_-1) -- (node_5_0);
            \draw[] (node_6_-1) -- (node_6_0);
            \draw[] (node_6_-1) -- (node_7_0);
        \end{tikzpicture}
        }}
        \vcenter{\hbox{
            \begin{tikzpicture}[scale=1]
            
                \def\cx{0}
                \def\cy{0}
                \node [circle, draw, minimum size=2cm] (c1) at (\cx,\cy) {};
    
                \def\labelpos{90};
                \markedpoint{c1}{\labelpos}{}{4};
                \draw[thick] (\cx, \cy) -- ({\cx+cos(\labelpos)},{\cy+sin(\labelpos)}); 
                \def\labelpos{0};
                \markedpoint{c1}{\labelpos}{}{10};
                \draw[thick] (\cx, \cy) -- ({\cx+cos(\labelpos)},{\cy+sin(\labelpos)});
                \def\labelpos{225};
                \markedpoint{c1}{\labelpos}{}{$n+1$};
                \draw[dotted] (\cx, \cy) -- ({\cx+cos(\labelpos)},{\cy+sin(\labelpos)});
            
                \def\cxprev{0}
                \def\cyprev{0}

                \def\cx{sqrt(2)}
                \def\cy{sqrt(2)}
                \node [circle, draw, minimum size=2cm] (c2) at ({\cx}, {\cy}) {};
                \def\labelpos{90};
                \markedpoint{c2}{\labelpos}{}{5};
                \draw[thick] ({\cx}, {\cy}) -- ({\cx+cos(\labelpos)},{\cy+sin(\labelpos)});
                \def\labelpos{330};
                \markedpoint{c2}{\labelpos}{}{9};
                \draw[thick] ({\cx}, {\cy}) -- ({\cx+cos(\labelpos)},{\cy+sin(\labelpos)});

                \draw[blue, thick] ({\cx},{\cy}) -- ({\cxprev}, {\cyprev});
            
                \def\cxprev{sqrt(2)}
                \def\cyprev{sqrt(2)}
                      
                \def\cx{2*sqrt(2)}
                \def\cy{2*sqrt(2)}
                \node [circle, draw, minimum size=2cm] (c3) at ({\cx}, {\cy}) {};
                \def\labelpos{90};
                \markedpoint{c3}{\labelpos}{}{6};
                \draw[thick] ({\cx}, {\cy}) -- ({\cx+cos(\labelpos)},{\cy+sin(\labelpos)});
                \def\labelpos{330};
                \markedpoint{c3}{\labelpos}{}{7};
                \draw[thick] ({\cx}, {\cy}) -- ({\cx+cos(\labelpos)},{\cy+sin(\labelpos)});
                \def\labelpos{300};
                \markedpoint{c3}{\labelpos}{}{8};
                \draw[thick] ({\cx}, {\cy}) -- ({\cx+cos(\labelpos)},{\cy+sin(\labelpos)});

                \draw[blue, thick] ({\cx},{\cy}) -- ({\cxprev}, {\cyprev});

                \def\cxprev{0}
                \def\cyprev{0}
                
                \def\cx{-sqrt(2)}
                \def\cy{sqrt(2)}
                \node [circle, draw, minimum size=2cm] (c4) at ({\cx}, {\cy}) {};
                \def\labelpos{180};
                \markedpoint{c4}{\labelpos}{}{1};
                \draw[thick] ({\cx}, {\cy}) -- ({\cx+cos(\labelpos)},{\cy+sin(\labelpos)});
                \def\labelpos{150};
                \markedpoint{c4}{\labelpos}{}{2};
                \draw[thick] ({\cx}, {\cy}) -- ({\cx+cos(\labelpos)},{\cy+sin(\labelpos)});
                \def\labelpos{90};
                \markedpoint{c4}{\labelpos}{}{3};
                \draw[thick] ({\cx}, {\cy}) -- ({\cx+cos(\labelpos)},{\cy+sin(\labelpos)});

                \draw[blue, thick] ({\cx},{\cy}) -- ({\cxprev}, {\cyprev});

                \node[red, 
                     draw, thick, 
                     fill, 
                     inner sep = 0pt, 
                     minimum size=1.5mm, 
                     circle, 
                     ] at (0, 0) 
                      {};
                 
            \end{tikzpicture}
        }}$
        \caption{($n=10$) Left: A stable tree $\tau\in \StRtree([10])$ with the root colored red and internal edges coloured blue. Right: An element of the standard cell $X_{\tau}\subseteq \overline{M_{0,11}}(\R)$ with $\tau$ overlayed. Here, $\dim X_{\tau}=(4-3)+(5-3)+(4-3)+(4-3)=5$.}
        \label{fig:stable-curve-with-dual-tree-ex}
    \end{figure}
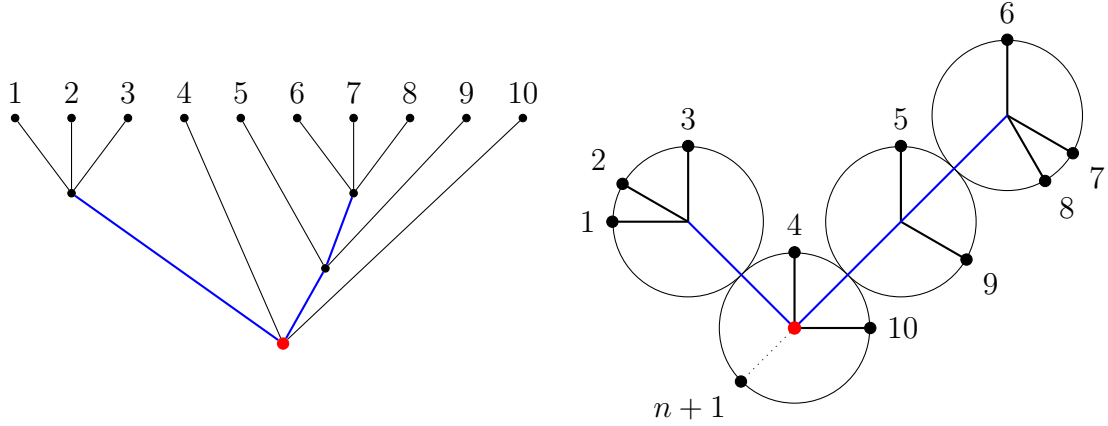

    We will also make use of the double cover $\overline{\DC{M}_{0,n+1}}(\R) \to \DMumford{n+1}(\R)$ that parametrizes real stable curves in which the component of the $(n+1)$-st marked point is given an orientation. (See e.g. Kapranov~\cite{kapranov1993}, where it is denoted as $\widetilde{S}^{n-3}$.) It can be thought of as the space of unit vectors in the real line bundle $\mathbb{L}_{n+1} \to \DMumford{n+1}(\R)$ parametrizing the cotangent space to $x_{n+1}$ on $C$. If $a_{n+1}=1$, a similar description gives a double cover $\overline{\DC{M}_{0,\mcA}}(\R)$ of $\overline{M_{0,\mcA}}(\R)$. These covers are compatible with the reduction maps $\rho_{\mcA, \mcB}$; these lift to maps $\tilde{\rho}_{\mcA, \mcB}: \overline{\DC{M}_{0,\mcA}}(\R) \to \overline{\DC{M}_{0,\mcB}}(\R)$.

    \subsubsection{Symmetric weights}

    We will focus on $(n+1)$-pointed Hassett spaces with $S_n$-symmetric weight vector $\mcA=(a_1,\hdots, a_n, a_{n+1} )\in (0,1]^{n+1}$, that is,
    \begin{equation}\label{eqn:assumptions-on-the-weight-data}
        a_1=a_2=\hdots=a_n \text{ and } a_{n+1}=1.
    \end{equation}
    The subsets $S\subseteq [n+1]$ that satisfy $\sum_{i\in S}a_i\le 1$ form an abstract simplicial complex $\mcK$ on $[n+1]$. Up to isomorphism, $\overline{M_{0, \mcA}}$ is completely determined by $\mcK$. Consequently, we may assume our weight vector is of the form
    \[\mcA=\mcA(a):=\mcA_{n+1}(a):=(\underbrace{\tfrac{1}{a}, \hdots, \tfrac{1}{a}}_{n}, 1)\]
    where $a\in \{ 1,2,\hdots, n-1 \}$. We say a curve is $a$-stable if it is $\mcA(a)$-stable. If $a=1, 2$, the moduli space is isomorphic to $\DMumford{n+1}$. If $a=n-1$, it is isomorphic to projective space. More explicitly, an $(n-1)$-stable curve is automatically irreducible, and we have an isomorphism $\DMumford{\mcA(n-1)} \xrightarrow{\ \sim\ } \Pj^{n-2}$ given by 
    $$(\Pj^1; x_{\bullet})\mapsto (x_2:x_3:\ldots: x_n)$$
    where the coordinates are chosen such that $x_1=0$ and $x_{n+1}=\infty$. For the remainder of this paper, we write $\overline{M_{0,\mcA_{n+1}(n-1)}} = \Pj^{n-2}$ under this isomorphism.

    Finally, for $(C; x_\bullet) \in \DMumford{\mcA(a)}(\R)$, we define its standard tree $\tau^{\mathrm{std}}(C; x_\bullet)$ analogously to the unweighted case, but where, for each $p \in C$ with $\{i : x_i = p\}$ nonempty, we have a leaf labeled by that subset of $[n]$. We then have $\tau^{\mathrm{std}}(C; x_\bullet) \in \StRtree(a; [n])$. The reduction map $\DMumford{n+1}(\R) \to \DMumford{\mcA(a)}(\R)$ then corresponds to the compression map $\varpi_a : \StRtree([n]) \to \StRtree(a; [n])$, and similarly for $\rho_{\mcA(a), \mcA(b)}$ and $\varpi_{a,b}$.

    \subsection{$G$-equivariant Fundamental Group}\label{subsec:G-eq-pi1}

    We review the notion of $G$-equivariant fundamental group and its relation to the ordinary fundamental group, for a space with a finite group action.

    \begin{defn}(\cite[Section 11.1]{ilinKLPR2023})\label{def:equivariant-fundamental-group}
        Let $G$ be a finite group acting on a path connected and locally simply connected space $X$, and fix a basepoint $x\in X$. We define the {\bf $G$-equivariant fundamental group} (relative to $x$) to be the group $\pi_1^G(X,x)$ of pairs $(g, \gamma)$, where $g\in G$ and $\gamma$ is a homotopy class of paths from $x$ to $gx$, equipped with the operation
        $$(g_1, \gamma_1)\cdot (g_2, \gamma_2) = (g_1g_2, \gamma_1 * (g_1\circ \gamma_2))$$
        where $*$ denotes concatenation of paths.
    \end{defn}
    Thus $\pi_1^G$ defines a functor from the category of pointed, locally simply connected $G$-spaces to groups. Given a pointed map $\phi : (X, x) \to (Y, y)$ of $G$-spaces, we denote the induced map of $G$-equivariant fundamental groups by $\phi_*^G$. As with the ordinary fundamental group, a path from $x$ to $x' \in X$ induces an isomorphism $\pi_1^G(X, x) \cong \pi_1^G(X, x')$, which depends only on the homotopy class of the path. We will therefore often suppress the basepoint in our notation and denote the group as $\pi_1^G(X)$. 
    
    The relationship between $\pi_1^G$ and $\pi_1$ is captured by the following \SES:
    \begin{equation}\label{eqn:ses_G_X}
        1\to \pi_1(X)\to \pi_1^G(X)\to G\to 1,
    \end{equation}
    where $\gamma\in \pi_1(X)$ maps to $(1,\gamma)\in \pi_1^G(X)$ and $(g,\gamma)\in \pi_1^G(X)$ maps to $g\in G$.

    Assume that $X$ is a connected CW-complex with an action of $G$ for which the $2$-skeleton $X^{(2)}$ is $G$-stable, that is, $gX^{(2)} \subseteq X^{(2)}$ for all $g \in G$. Then, as is the case for ordinary fundamental groups, we may essentially replace $X$ by its $2$-skeleton:

    \begin{thm}
        Let $X$ be a connected CW-complex. Let $G$ be a finite group acting on $X$. 
        Let $X^{(k)}$ denote the $k$-skeleton, and suppose that $X^{(2)}$ is $G$-stable. Then  $\pi_1^G(X)\cong \pi_1^G(X^{(2)})$.
    \end{thm}

    \begin{proof}
        Consider the inclusion $\iota:X^{(2)}\into X$. This yields the following commutative diagram
            \begin{equation}
                \begin{tikzcd}
                1 \arrow[r] & \pi_1(X^{(2)}) \arrow[r] \arrow[d, "\iota^*"'] & \pi_1^G(X^{(2)}) \arrow[r] \arrow[d, "\iota^*_G"'] & G \arrow[r] \arrow[d, Rightarrow, no head] & 1 \\
                1 \arrow[r] & \pi_1(X) \arrow[r]                             & \pi_1^G(X) \arrow[r]                               & G \arrow[r]                                & 1
                \end{tikzcd}
            \end{equation}
        where the rows are exact. Then $\iota^*$ is an isomorphism by the cellular approximation theorem. We conclude, by the short five lemma, that $\iota^*_G$ is an isomorphism.
    \end{proof}

    We state an analog, for equivariant fundamental groups, of the fact that a CW complex with one $0$-cell has a fundamental group with one generator for each $1$-cell and one generator for each $2$-cell.

    We replace the assumption ``$X$ has only one $0$-cell'' by ``the action of $G$ on $X$ is simply transitive on the $0$-cells'', i.e. for $0$-cells $x$ and $y$, there is a unique element $g\in G$ such that $y=gx$. We assume moreover that $G$ acts on the $1$-cells and $2$-cells, and that each $1$-cell connects two 0-cells, $x$ and $gx$, via an element $g \in G$ such that $g^2=1$. We may then assign, without ambiguity, the label $g$ to that 1-cell. 
    
    We obtain the following presentation of the $G$-equivariant fundamental group.

    \begin{thm}\label{thm:pres-pi1GX-simply-transitive}
        Let $X$ be a connected CW-complex with an action of a finite group $G$. Suppose in addition that $G$ acts on the $k$-cells of $X$ for $k \leq 2$, and that the action on the $0$-cells is simply transitive. Furthermore, suppose that each 1-cell connects 0-cells $x$ and $gx$ with $g^2=1$.
        
        Fix a base point $x_0\in X^{(0)}$. Then $\pi_1^G(X, x_0)$ has the presentation
        \[
        \pi_1^G(X, x_0) \cong \langle S | R \cup \{s^2: s\in S\}\rangle
        \]
        where $S$ is the set of $1$-cells incident to $x_0$ (ie. the set of distinct labels of the 1-cells) and $R$ is the set of words obtained from reading the boundaries of the $2$-cells containing $x_0$. 
    \end{thm}

    We follow the argument given in the proof of \cite[Thm 4.7.2]{davisJS2003}, which makes use of the following alternate description of $\pi_1^G$ in terms of lifts.

    Each element $g\in G$ defines a map $\theta_g: X\to X$ sending $x\mapsto gx$. Let $p: \tilde{X}\to X$ be the universal cover. A lift of $\theta_g$ to $\tilde{X}$ is a (continuous) map $\tilde{\theta_g}:\tilde{X}\to \tilde{X}$ such that the diagram below commutes.
            \begin{equation}
                \begin{tikzcd}
                    \tilde{X} \arrow[d, "p"'] \arrow[r, "\tilde{\theta_g}"] & \tilde{X} \arrow[d, "p"] \\
                    X \arrow[r, "{\theta_g}"]                               & X                     
                \end{tikzcd}
            \end{equation}
    We refer to such a lift as a lift of the $G$-action on $X$ to $\tilde{X}$. The set of all lifts of the $G$-action on $X$ form a group~\cite[Section 9]{bredon1972} under composition. There is a natural surjective homomophism from this group of lifts to $G$, whose kernel is $\pi_1(X)$.

            \begin{lem}\label{lem:ses_G_X}
                Let $X$ be as in Definition \ref{def:equivariant-fundamental-group}.
                There is an isomorphism between $\pi_1^G(X)$ and group of lifts of the $G$-action on $X$ to its universal cover.
            \end{lem}

    We now prove Theorem \ref{thm:pres-pi1GX-simply-transitive}.
    \begin{proof}
        \WLOG,\ assume $\dim X\le 2$. Let $p: \tilde{X}\to X$ be the universal cover. Then $\tilde{X}$ is a CW-complex whose $k$-cells are the connected components of the preimages by $p$ of the $k$-cells of $X$. Let $A$ be group of lifts of the $G$-action on $X$ to $\tilde{X}$. By Lemma~\ref{lem:ses_G_X}, $A\cong \pi_1^G(X)$.
        
        Fix 0-cells $\tilde{x}_0\in \tilde{X}$ and $x_0= p(\tilde{x}_0)\in X$. Identifying $S$ as a subset of $G$, we know that the 1-cells with $x_0$ as one endpoint connect to 0-cells of the form $sx_0$ for some $s\in S$. The set $S$ consists of order two elements in $G$ and each $s\in S$ can be viewed as a path from $x_0$ to $s x_0$: a homeomorphism from $[0,1]$ to that 1-cell. Let $\tilde{s}$ denote the unique lift (starting at $\tilde{x}_0$) of the path $s$ to $\tilde{X}$. Since the 1-skeleton $\tilde{X}^{(1)}$ is connected, the set $\{\tilde{s}: s\in S\}$ forms a set of generators for $A$. Furthermore, we see that $\tilde{s}^2$ is a contractible loop, so $\tilde{s}^2=1\in A$. By examining the 2-cells of $\tilde{X}$ containing $\tilde{x}_0$, we obtain precisely the relations in $R$ (where each $s\in S$ is replaced by $\tilde{s}$). 
        
        Let $Y$ be the Cayley 2-complex of $H:=\gen{S | R \cup \{s^2: s\in S\}}$. By the above discussion, $A$ is a quotient of $H$ and there is a covering map $\tilde{X} \to Y$. Since $\tilde{X}$ is simply connected, $H\cong\pi_1(Y)$ is the group of deck transformations~\cite[Proposition 1.39]{hatcher2002} of this covering map. Interpreting both $H$ and $A$ as a group of continuous maps $\tilde{X}\to \tilde{X}$ under composition, it can be checked that the identity map $H\to A$ is an isomorphism.
    \end{proof}

    In our application to $\DMumford{\mcA}(\R)$, the action of $G = S_n$ is \emph{not} simply transitive on the $0$-cells; rather, each $0$-cell has an order-two stabilizer (because reversing the order of the marked points can be done by an orientation-reversing automorphism of $\R\Pj^1$). The action of $S_n$ on the double cover $\overline{\DC{M}_{0,\mcA}}(\R)$, however, \emph{is} transitive and so Theorem \ref{thm:pres-pi1GX-simply-transitive} applies to it. To account for this step, we briefly state how equivariant fundamental groups behave under certain covering maps. 
    
    We say a covering map $p : (X', x_0') \to (X, x_0)$ of connected, pointed $G$-spaces is \emph{$G$-regular} if $p_*(\pi_1^G(X'))$ is a normal subgroup of $\pi_1^G(X)$. This is stronger than the usual notion of regular covering map.

    \begin{lem} \label{lem:stabilizer-semidirect}
        Let $p : (X', x_0') \to (X, x_0)$ be a $G$-regular covering map. Suppose the stabilizer of $x_0'$ is trivial. 
        Then
        \[
        \pi_1^G(X) \cong \pi_1^G(X') \rtimes G_{x_0},
        \]
        where $G_{x_0} \subseteq G$ is the stabilizer of $x_0$.
    \end{lem}
    \begin{proof}
        We have
        \[G_{x_0}\cong \{(g, \id_{x_0}): g\in G_{x_0}\}\subseteq \pi_1^G(X),\]
        and
        \[G_{x_0} \cap p_*(\pi_1^G(X')) = G_{x_0'} = 1,\]
        since the stabilizer of $x_0'$ is trivial. Finally, these two subgroups generate $\pi_1^G(X)$: if $(g, \gamma) \in \pi_1^G(X)$, we may lift $\gamma$ to a path from $x_0'$ to some $y' = g' x_0'$. Then $g x_0 = g' x_0$, so
        \[(g, \gamma) = (g', \gamma) * ((g')^{-1} g, \id_{x_0})\]
        and $(g', \gamma) \in p_*(\pi_1^G(X')).$
    \end{proof}

    Returning to our case of interest, let $\DC{X}$ be a connected double cover for $X$, satisfying the assumptions of Lemma \ref{lem:stabilizer-semidirect}. Then $G_{x_0} \cong \Z/2\Z$. We then have the following commutative diagram, with exact rows:
    \begin{equation}\label{comdiag:X2:1-to-X}
        \begin{tikzcd}
        1 \arrow[r] & \pi_1(\DC{X}) \arrow[r] \arrow[d, "\text{index}\ 2"', hook] & \pi_1^G(\DC{X}) \arrow[r] \arrow[d, "\text{index}\ 2"', hook] & G \arrow[r] \arrow[d, Rightarrow, no head] & 1 \\
        1 \arrow[r] & \pi_1(X) \arrow[r]                                           & \pi_1^G(X) \arrow[r]                                           & G \arrow[r]                                & 1
        \end{tikzcd}
    \end{equation}
    In this case, a presentation for $\pi_1^G(X)$ can be stated directly in terms of one for $\pi_1^G(X')$.

    \begin{prop}\label{prop:group-pres-extend-by-Zmod2}
        Let $G$ be a group with subgroups $H$ and $K$. Assume $G=HK$, $H\cap K =1$ and $|K| = 2$ with generator $\alpha \in K$. Let $\iota: H\to H$ be the automorphism $\iota(h)=\alpha h \alpha$.
        
        Let $H=\gen{S|R}$ be a presentation for $H$ such that $\iota(S)\subseteq S$. Then a presentation of $G$ is given by
        \[
        G= \langle S \sqcup \{\alpha\} \mid R\cup \{\alpha^2\} \cup \{\alpha s\alpha \iota(s)^{-1}: s\in S\} \rangle.
        \]
        Equivalently, $G\cong (H*K)/\gen{\alpha s\alpha \iota(s)^{-1}: s\in S}=H\rtimes K$.
    \end{prop}

    \subsection{Cactus groups as $S_n$-equivariant fundamental groups}\label{subsec:rel-to-cactus-group}

    We briefly describe how the cactus groups arise as $S_n$-equivariant fundamental groups. As discussed above, we consider Hassett spaces $\DMumford{\mcA(a)}$ with the symmetric weight vector 
    \[\mcA(a) = \big( \tfrac{1}{a}, \ldots, \tfrac{1}{a}, 1).\]
    The symmetric group $S_n$ acts by permuting the labels $1, \ldots, n$; this action is compatible with the reduction maps and double covers discussed above.
    
    We always use as our base point for $\pi_1^{S_n}$ the irreducible stable curve
    \[(\Pj^1; 1, \ldots, n, \infty).\]
    These points are all evenly-spaced on $\Pj^1$, apart from $x_{n+1} = \infty$. For each $\sigma \in S_n$, we denote by $C_\sigma$ the curve with permuted marked points
    \[(\Pj^1; \sigma^{-1}(1), \ldots, \sigma^{-1}(n), \infty),\]
    which we call a \emph{permutation point} and denote $C_\sigma$. We note that on $\DMumford{\mcA(a)}(\R)$, the choice of $\sigma$ is unique up to right composition with $w_{1, n}$. Hence, there are exactly $\tfrac{1}{2} n!$ permutation points. In contrast, on the double cover $\overline{\DC{M}_{0,\mcA(a)}}(\R)$, the points $C_\sigma$ are all distinct, because $w_{1, n}$ effectively reverses the orientation of the $\R\Pj^1$. As such, there are $n!$ permutation points on the double cover.

    In the presentation of the unweighted cactus group $J_n$ above, the generator $s_{p,q}$ corresponds to the pair $(w_{p, q}, \hat{s}_{p, q})$, where $w_{p,q}\in S_n$ is the permutation that reverses the interval $p, \ldots, q$, and $\hat{s}_{p,q}$ is the path from $C_{\mathrm{id}}$ to $C_{w_{p,q}}$ defined by having the $p$-th through $q$-th marked points approach each other and collide. Explicitly, for $t \in [0, 1] \setminus \{\tfrac{1}{2}\}$, we put
    \[x_i(t) =
    \begin{cases}
    i & \text{ if } i \in \{1, \ldots, p-1, q+1, \ldots, n\}, \\
    t i + (1-t) (p+q-i) & \text{ if } i \in \{p, \ldots, q\},
    \end{cases}\]
    and $x_{n+1}(t) = \infty$ for all $t$. For $t=\tfrac{1}{2}$, the marked points $x_p, \ldots, x_q$ bubble off to a second irreducible component, where they are again evenly spaced relative to the attaching node (which is at $\infty$). See Figure \ref{fig:spq-hat-path}.
    
\begin{figure}
    \centering
    $\vcenter{\hbox{
    \begin{tikzpicture}
        \node [circle, draw, minimum size=1.5cm] (c1) at (0,0) {};
        \markedpoint{c1}{120}{}{1};
        \markedpoint{c1}{90}{}{2};
        \path (0, 0.75) arc (90:30:0.75) node [midway, above, sloped, align=center] {$\hdots$};
        \draw (0, 0.75) arc (90:0:0.75);
        \markedpoint{c1}{30}{}{$p$};
        \markedpoint{c1}{330}{}{$q$};
        \path ({0.75*cos(30)}, {0.75*sin(30)}) arc (30:-30:0.75) node [midway, above, sloped, align=center] {$\hdots$};
        \path ({0.75*cos(-30)}, {0.75*sin(-30)}) arc (-30:-90:0.75) node [midway, below, sloped, align=center] {$\hdots$};
        \markedpoint{c1}{240}{}{$n{+}1$};
        \markedpoint{c1}{270}{}{$n$};
    \end{tikzpicture}
    }}$
    $\leadsto$
    $\vcenter{\hbox{
    \begin{tikzpicture}
        \node [circle, draw, minimum size=1.5cm] (c1) at (0,0) {};
        \markedpoint{c1}{120}{}{1};
        \markedpoint{c1}{90}{}{2};
        \path (0, 0.75) arc (90:30:0.75) node [midway, above, sloped, align=center] {$\hdots$};
        \path ({0.75*cos(-30)}, {0.75*sin(-30)}) arc (-30:-90:0.75) node [midway, below, sloped, align=center] {$\hdots$};
        \markedpoint{c1}{240}{}{$n{+}1$};
        \markedpoint{c1}{270}{}{$n$};
    
        \node [circle, draw, minimum size=1.5cm] (c2) at (1.5,0) {};
        \markedpoint{c2}{60}{}{$p$};
        \markedpoint{c2}{300}{}{$q$};
        \path ({1.5+0.75*cos(30)}, {0.75*sin(30)}) arc (30:-30:0.75) node [midway, below, sloped, align=center] {$\hdots$};
    \end{tikzpicture}
    }}$
    $=$
    $\vcenter{\hbox{
    \begin{tikzpicture}
        \node [circle, draw, minimum size=1.5cm] (c1) at (0,0) {};
        \markedpoint{c1}{120}{}{1};
        \markedpoint{c1}{90}{}{2};
        \path (0, 0.75) arc (90:30:0.75) node [midway, above, sloped, align=center] {$\hdots$};
        \path ({0.75*cos(-30)}, {0.75*sin(-30)}) arc (-30:-90:0.75) node [midway, below, sloped, align=center] {$\hdots$};
        \markedpoint{c1}{240}{}{$n{+}1$};
        \markedpoint{c1}{270}{}{$n$};
    
        \node [circle, draw, minimum size=1.5cm] (c2) at (1.5,0) {};
        \markedpoint{c2}{60}{}{$q$};
        \markedpoint{c2}{300}{}{$p$};
        \path ({1.5+0.75*cos(30)}, {0.75*sin(30)}) arc (30:-30:0.75) node [midway, below, sloped, align=center] {$\hdots$};
    \end{tikzpicture}
    }}$
    $\leadsto$
    $\vcenter{\hbox{
    \begin{tikzpicture}
        \node [circle, draw, minimum size=1.5cm] (c1) at (0,0) {};
        \markedpoint{c1}{120}{}{1};
        \markedpoint{c1}{90}{}{2};
        \path (0, 0.75) arc (90:30:0.75) node [midway, above, sloped, align=center] {$\hdots$};
        \draw (0, 0.75) arc (90:0:0.75);
        \markedpoint{c1}{30}{}{$q$};
        \markedpoint{c1}{330}{}{$p$};
        \path ({0.75*cos(30)}, {0.75*sin(30)}) arc (30:-30:0.75) node [midway, above, sloped, align=center] {$\hdots$};
        \path ({0.75*cos(-30)}, {0.75*sin(-30)}) arc (-30:-90:0.75) node [midway, below, sloped, align=center] {$\hdots$};
        \markedpoint{c1}{240}{}{$n{+}1$};
        \markedpoint{c1}{270}{}{$n$};
    \end{tikzpicture}
    }}$
    \caption{The path $\hat{s}_{p,q}$ in $\DMumford{n+1}(\R)$ corresponding to $s_{p,q}\in J_n$: the $p$-th through $q$-th marked points approach one another and collide, reversing their order. (The positions of the marked points are not drawn to scale.)}
    \label{fig:spq-hat-path}
\end{figure}
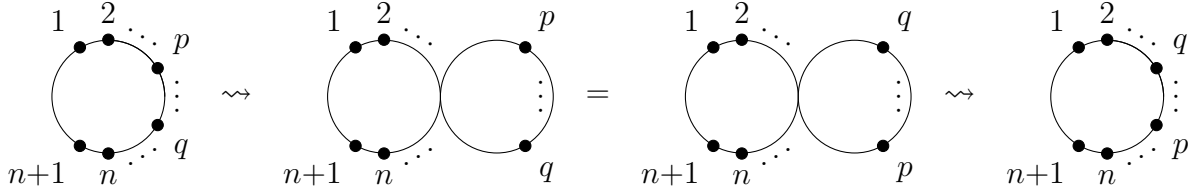

    We define the \emph{weighted cactus group} $J_n^a$ and \emph{pure weighted cactus group} $PJ_n^a$ as in the introduction,
    \[
    J_n^a := \pi_1(\overline{M_{0,\mcA(a)}}(\R)), \quad
    PJ_n^a := \pi_1(\overline{M_{0,\mcA(a)}}(\R)).
    \]
    We likewise define \emph{oriented weighted (pure) cactus groups}
    \[
    \tilde{J}_n^a := \pi_1(\overline{\DC{M}_{0,\mcA(a)}}(\R))), \quad
    P\tilde{J}_n^a := \pi_1(\overline{\DC{M}_{0,\mcA(a)}}(\R))).
    \]

    The various cactus groups are related to one another as follows. The reduction map $\overline{M_{0,n+1}}\onto \overline{M_{0, \mcA(a)}}$ is $S_n$-equivariant, hence induces a commutative diagram with exact rows
    \begin{equation}\label{comdiag:proj-Jn-to-Jna}
        \begin{tikzcd}
        1 \arrow[r] & PJ_n \arrow[r] \arrow[d, ->>] & J_n \arrow[r] \arrow[d, ->>] & S_n \arrow[r] \arrow[d, Rightarrow, no head] & 1 \\
        1 \arrow[r] & PJ_n^a \arrow[r]         & J_n^a \arrow[r]         & S_n \arrow[r]                                & 1
        \end{tikzcd}
    \end{equation}
    We note that $PJ_n\to PJ_n^a$ and $J_n \to J_n^a$ are surjective since they are induced by a blowdown map (this follows e.g. by the transversality homotopy theorem).

    By Proposition \ref{prop:group-pres-extend-by-Zmod2}, $\tilde{J}_n$ is the index 2 subgroup of $J_n$ generated by $s_{p,q}$ with $(p,q)\neq (1,n)$, and
    \[
    J_n \cong \tilde J_n \rtimes \Z/2\Z,
    \]
    where the $\Z/2\Z$ is generated by $s_{1, n}$. Likewise, we have
    \[
    J_n^a \cong \tilde{J}_n^a \rtimes \Z/2\Z.
    \]
    In particular, from \eqref{comdiag:X2:1-to-X}, we obtain the commutative diagram with exact rows:
        \begin{equation}
            \begin{tikzcd}
            1 \arrow[r] & P\tilde{J}_n^a \arrow[r] \arrow[d, "\text{index}\ 2"', hook] & \tilde{J}_n^a \arrow[r] \arrow[d, "\text{index}\ 2"', hook] & S_n \arrow[r] \arrow[d, Rightarrow, no head] & 1 \\
            1 \arrow[r] & PJ_n^a \arrow[r]                                           & J_n^a \arrow[r]                                           & S_n \arrow[r]                                & 1
            \end{tikzcd}
        \end{equation}
    We remark that for $a=1,2$, the reduction map $\overline{M_{0,n+1}}\to \overline{M_{0, \mcA(a)}}$ is an isomorphism and the above diagrams yield isomorphisms $J_n\cong J_n^a$, $PJ_n\cong PJ_n^a$, $\tilde{J}_n\cong \tilde{J}_n^a$, $P\tilde{J}_n\cong P\tilde{J}_n^a$.

    \section{Dual CW-Structure on the Hassett Space and its Double Cover}\label{sec:dual-cell}

As discussed in Section \ref{subsec:curves-moduli-spaces}, $\overline{M_{0,n+1}}(\R)$ has a standard stratification into cells indexed by stable trees; the \emph{standard tree} $\tau^\mathrm{std}(C; x_\bullet)$ describes the arrangement of the marked points $x_\bullet$ and nodes on the irreducible components of $C$. The cells
\[
X_\tau = \{(C; x_\bullet) : \tau^\mathrm{std}(C; x_\bullet) = \tau\}
\]
are well-known to be products of associahedra. There is a dual decomposition based on assigning to $(C; x_\bullet)$ a \emph{dual tree}, and for which the resulting cells turn out to be cubes. We recall the details of this construction below, following the presentation of \cite{ilinKLPR2023}.

The goal of this section is to construct a cell decomposition of the Hassett spaces $\overline{M_{0, \mcA}}(\R)$, for weights $\mcA=\mcA(a)$ where $a\in [n-1]$. In the case $a = 1$, this stratification is dual to the standard one. These cells are indexed by trees similar to those defined in Section \ref{subsec:rooted-trees-defs}, except that the leaves are labelled by subsets of $[n]$ and form a partition of $[n]$. The parameter $a$ upper-bounds the sizes of the parts of the partition. 
In the case where $a=1$, the description is especially simple and recovers the dual decomposition of $\overline{M_{0, n+1}}(\mathbb{R})$ into cubes. 

\begin{rmk}\label{rmk:ilin-LB-decomp}
    Ilin et. al~\cite{ilinKLPR2023} also describe a cell decomposition of the real locus of the line bundle $\mathbb{L}_{n+1} \to \DMumford{n+1}(\R)$ (which they denote by by $\widetilde{M_{0,n+1}}(\R)$) and a ``cactus flower space". The decomposition given in Section~\ref{subsec:dual-Mod-space-unweighted-strat} reduces to the one given by Ilin et al., restricted to its zero section $\overline{M_{0,n+1}}(\R)$.)
\end{rmk}

    \subsection{Dual Stratification of $\overline{M_{0,n+1}}(\R)$}\label{subsec:dual-Mod-space-unweighted-strat}

    We first describe the dual stratification of the unweighted moduli space. This material can be found in \cite{davisJS2003, ilinKLPR2023}. The description uses ordered rooted trees\footnote{rooted trees are drawn with the root on the bottom, which is the convention in ~\cite{ilinKLPR2023}} and distances (gaps) between consecutively placed marked points. We will describe, in Section~\ref{subsec:dual-refined-cubes}, two different kinds of cells that are combinatorially equivalent to cubes: \textit{dual cubes} which form the dual stratification; and \textit{refined cubes} which form the intersection of both the standard and dual stratifications. Ilin et al. refer to these as \textit{big cubes} and \textit{little cubes} respectively.

    \subsubsection{The distance algorithm}
    
    We first give an algorithm that associates to a vector $\vec{d} = (d_1, \ldots, d_{n-1}) \in \mathbb{R}^{n-1}_{\geq 0}$ and a permutation $\sigma \in S_n$ a rooted tree in 
    $$\Rtree^{\sigma}([n]):=\Rtree(\sigma(1), \ldots, \sigma(n)).$$ 
    In particular, we give a surjective map 
    $$\R_{\ge0}^{n-1}\onto \Rtree^{\sigma}([n]).$$
    We call $\vec{d}$ a \emph{vector of differences} and the $d_i$ \emph{distances}.

\begin{defn}[Distance algorithm] \label{def:dist-alg-diff-to-rtree}
    Let $\vec{d}$ and $\sigma$ be as above. We construct $\tau$ as follows. We begin with $n$ isolated vertices labeled $\sigma(1), \ldots, \sigma(n)$ from left to right, each considered as a rooted tree with one vertex. We join these trees together by repeating the following steps until $\vec{d}$ is empty.
    \begin{enumerate}
        \item Let $d := \min d_i$ be the minimum distance remaining in $\vec{d}$.
        \item For each sequence of consecutive copies of $d$ in $\vec{d}$, say $d_i = d_{i+1} = \cdots = d_j$, we attach the roots of the $i$-th through $(j+1)$-st subtrees as ordered children of a new, unlabeled vertex. We then delete $d_i, \ldots, d_j$ from $\vec{d}$.
    \end{enumerate}
    We write $\tau^{\mathrm{dist}}(\sigma, \vec{d})$ for the resulting rooted tree. See Figure \ref{fig:tree-algorithm-on-stable-curve} (Left).
\end{defn}

In fact, this map is uniquely determined by the following two properties: if $v_i$ denotes the nearest common ancestor of leaves $\sigma(i)$ and $\sigma(i+1)$ in $\tau^{\mathrm{dist}}(\sigma, \vec{d})$, then under the partial order~\eqref{eqn:rtree-partial-order} on $\tau$,
\begin{itemize}
    \item if $v_i=v_j$, then $d_i=d_j$; and
    \item if $v_i<v_j$, then $ d_i<d_j$.  
\end{itemize}
That is, vertices close to (far from) the root represent large (small) gaps between two consecutively placed marked points. 
    
We observe that $\tau^{\mathrm{dist}}(\sigma, \vec{d})$ is binary if and only if every two consecutive coordinates of $\vec{d}$ are distinct (i.e. $d_i\neq d_{i+1}$). At the opposite extreme, if all the distances in $\vec{d}$ are equal, $\tau^{\mathrm{dist}}(\sigma, \vec{d})$ has one internal vertex and its leaves spell out $\sigma$ from left to right.
    
    Finally, given a tree in $\Rtree^{\sigma}([n])$, it is straightforward to find a distance vector $\vec{d}$ compatible with it, so surjectivity readily follows.

    \begin{rmk}\ \label{rmk:distances-to-trees}
    We will run the distance algorithm on successive distances of an ordered list of values, that is, with $\vec{d}$ defined by
    \[d_i := x_i - x_{i-1} \text{ for some } x_1 < \cdots < x_n \in \R.\]
    In this case,
    \begin{enumerate}
    \item\label{rmk-item:tau-little-inv-under-pos-affine} The resulting tree is invariant under acting on the $x_i$ by an affine linear transformation $x \mapsto ax+b$ with $a > 0$. If $a < 0$, the action flips the resulting tree at its root.
    \item For general $\vec{x}$, Definition~\ref{def:dist-alg-diff-to-rtree} outputs a full binary tree.
    \end{enumerate}
    \end{rmk}

    \subsubsection{From marked curves to trees}\label{subsec:marked-curves-to-trees}

    We now associate to each element of $\overline{M_{0,n+1}}(\R)$ a tree in $\Rtree([n])$.

    We first define maps on the interior of the moduli space,
    \begin{align*}
    \littlecell &: M_{0,n+1}(\R) \to \RStRtree([n]), \\
    \bigcell &: M_{0,n+1}(\R) \to \StRtree([n]),
    \end{align*}
    as follows. Given $(C; x_\bullet)$, we choose any coordinates on $C \cong \R\Pj^1$ for which the marked point $x_{n+1} = \infty$. The remaining marked points are ordered as
    \[
    x_{\sigma(1)} < \cdots < x_{\sigma(n)} \text{ for some } \sigma \in S_n.
    \]
    Let $\vec{d}$ be the vector of successive differences
    \[
    \vec{d} = (x_{\sigma(2)} - x_{\sigma(1)}, \ldots, x_{\sigma(n)} - x_{\sigma(n-1)})
    \]
    and let $\tau = \tau^\mathrm{dist}(\sigma; \vec{d}) \in \Rtree^\sigma([n])$ be the rooted tree obtained from the distance algorithm. By Remark \ref{rmk:distances-to-trees}(\ref{rmk-item:tau-little-inv-under-pos-affine}), $\tau$ is independent of the choice of coordinates on $C$, up to flipping at the root, as long as $x_{n+1} = \infty$. Therefore, following the maps in Equations \eqref{eqn:map-r-rst-tree} and \eqref{eqn:map-dual-rst-st-tree}, we define $\littlecell(C; x_\bullet) \in \RStRtree([n])$ to be $\tau$ with (only) the root vertex marked as flippable, and we define $\bigcell(C; x_\bullet) \in \StRtree([n])$ to be $\tau$ with every internal vertex marked as flippable.

    Next, we extend the maps $\littlecell,\bigcell$ to the boundary of the moduli space:
    \begin{align*}
        \littlecell &: \overline{M_{0,n+1}}(\R)\to \RStRtree([n]), \\
        \bigcell &: \overline{M_{0,n+1}}(\R)\to \StRtree([n]),
    \end{align*}
    as follows.
    Let $(C; x_{\bullet})$ be an $(n+1)$-pointed stable curve. We choose rooted coordinates on $(C; x_\bullet)$ as in Definition \ref{def:rooted-coordinates}. 
    We then compute $\littlecell(C')$ for each component $C'$, together with its marked points and nodes, by the algorithm above. Finally, we let $\tau \in \Rtree([n])$ be obtained by joining the resulting trees according to how the components of $C$ are attached: if $C''$ is the other component attached to $C'$ at $q(C')$, we replace the corresponding leaf of $\littlecell(C'')$ by the subtree $\littlecell(C')$. The resulting tree $\tau$ is well-defined up to flips at the root and at the vertices corresponding to the roots of the trees $\littlecell(C')$. Accordingly, we define $\littlecell(C; x_\bullet)$ to be $\tau$ with precisely these vertices marked flippable. We define $\bigcell(C; x_\bullet)$ to be $\tau$ with all its vertices marked flippable. See Figure \ref{fig:tree-algorithm-on-stable-curve}~(Right). We remark that these trees can equivalently be obtained by replacing each internal vertex of $\tau^\mathrm{std}(C; x_\bullet)$ by the appropriate subtree $\littlecell(C')$. \\

\begin{figure}
    \centering
    \scalebox{0.65}{
    \begin{tikzpicture}[scale=1.5]
        \draw[red] (-0.5,0) -- (7,0) node[anchor = west] {$\color{black}\R$};
        \node[circle, draw, fill, inner sep =1pt, label = { 90: $x_1=0$ }] (x1) at (0,0) {};
        \node[circle, draw, fill, inner sep =1pt, label = { 90: $x_2=1$ }] (x2) at (1,0) {};
        \node[circle, draw, fill, inner sep =1pt, label = { 90: $x_3=2$ }] (x3) at (2,0) {};
        \node[circle, draw, fill, inner sep =1pt, label = { 90: $x_4=4$ }] (x4) at (4,0) {};
        \node[circle, draw, fill, inner sep =1pt, label = { 90: $x_5=5$ }] (x5) at (5,0) {};
        \node[circle, draw, fill, inner sep =1pt, label = { 90: $x_6=6.5$ }] (x6) at (6.5,0) {};
        \node[circle, draw, fill, inner sep =1pt] (x123) at (1,-1) {};
        \node[circle, draw, fill, inner sep =1pt] (x45) at (4.5,-1) {};
        \node[circle, draw, fill, inner sep =1pt] (x456) at (5.5,-2) {};
        \node[circle, draw, fill, inner sep =1pt] (x123456) at (3.25,-4) {};

        \coordinate[label={270: $x_{n+1}$}] (xnplus1) at (3.25,-4.5) {};
    
        \draw[] (x1) -- (x123);
        \draw[] (x2) -- (x123);
        \draw[] (x3) -- (x123);
        \draw[] (x4) -- (x45);
        \draw[] (x5) -- (x45);
        \draw[] (x45) -- (x456);
        \draw[] (x6) -- (x456);
        \draw[] (x123) -- (x123456);
        \draw[] (x456) -- (x123456);
        \draw[dotted, thick] (x123456) -- (xnplus1);
    \end{tikzpicture}
    }
    \scalebox{0.62}{
        \begin{tikzpicture}[scale=2]
        
            \def\cx{0}
            \def\cy{0}
            \node [circle, draw, minimum size=4cm] (c1) at (\cx,\cy) {};
            
            \def\labelpos{100};
            \markedpoint{c1}{\labelpos}{}{4};
            \def\temp{135} 
            \draw[ultra thick] ({\cx+cos(\temp)/2}, {\cy+sin(\temp)/2}) -- ({\cx+cos(\labelpos)},{\cy+sin(\labelpos)});

            \def\labelpos{5};
            \markedpoint{c1}{\labelpos}{}{10};
            \def\temp{45} 
            \draw[ultra thick] ({\cx+cos(\temp)/2}, {\cy+sin(\temp)/2}) -- ({\cx+cos(\labelpos)},{\cy+sin(\labelpos)});

            \def\labelpos{225};
            \markedpoint{c1}{\labelpos}{}{$n+1$};
            \draw[dotted] (\cx, \cy) -- ({\cx+cos(\labelpos)},{\cy+sin(\labelpos)});

            \def\cxprev{0}
            \def\cyprev{0}

            \def\cx{sqrt(2)}
            \def\cy{sqrt(2)}
            \node [circle, draw, minimum size=4cm] (c2) at ({\cx}, {\cy}) {};

            \def\labelpos{90};
            \markedpoint{c2}{\labelpos}{}{5};
            \def\temp{45} 
            \draw[ultra thick] ({\cx+cos(\temp)/2}, {\cy+sin(\temp)/2}) -- ({\cx+cos(\labelpos)},{\cy+sin(\labelpos)});
            
            \def\labelpos{330};
            \markedpoint{c2}{\labelpos}{}{9};
            \draw[ultra thick] ({\cx}, {\cy}) -- ({\cx+cos(\labelpos)},{\cy+sin(\labelpos)});
        
            \def\labelpos{45};
            \node[blue, 
                 inner sep = 0pt, 
                 minimum size=1.5mm, 
                 rectangle, 
                 label = {\labelpos : $\infty$ }  
                 ] at (c1.\labelpos) 
                  {};
            \draw[blue, ultra thick] ({\cx},{\cy}) -- ({\cxprev+cos(\labelpos)/2}, {\cyprev+sin(\labelpos)/2});
            \draw[ultra thick] ({\cxprev},{\cyprev}) -- ({\cxprev+cos(\labelpos)/2}, {\cyprev+sin(\labelpos)/2});
        
            \def\cxprev{sqrt(2)}
            \def\cyprev{sqrt(2)}
                  
            \def\cx{2*sqrt(2)}
            \def\cy{2*sqrt(2)}
            \node [circle, draw, minimum size=4cm] (c3) at ({\cx}, {\cy}) {};
            \def\labelpos{90};
            \markedpoint{c3}{\labelpos}{}{6};
            \draw[ultra thick] ({\cx}, {\cy}) -- ({\cx+cos(\labelpos)},{\cy+sin(\labelpos)});

            \def\lpa{330}
            \def\lpb{300}
            \def\avglp{(\lpa+\lpb)/2}
            \def\labelpos{\lpa};
            \markedpoint{c3}{\labelpos}{}{7};
            \draw[ultra thick] ({\cx+cos(\avglp)/2},{\cy+sin(\avglp)/2}) -- ({\cx+cos(\labelpos)},{\cy+sin(\labelpos)});
            \def\labelpos{\lpb};
            \markedpoint{c3}{\labelpos}{}{8};
            \draw[ultra thick] ({\cx+cos(\avglp)/2},{\cy+sin(\avglp)/2}) -- ({\cx+cos(\labelpos)},{\cy+sin(\labelpos)});

            \draw[ultra thick] ({\cx},{\cy}) -- ({\cx+cos(\avglp)/2},{\cy+sin(\avglp)/2});

            \def\labelpos{45};
            \node[blue, 
                 inner sep = 0pt, 
                 minimum size=1.5mm, 
                 label = {\labelpos : $\infty$ }  
                 ] at (c2.\labelpos) 
                  {};
            \draw[blue, ultra thick] ({\cx},{\cy}) -- ({\cxprev+cos(\labelpos)/2}, {\cyprev+sin(\labelpos)/2});
            \draw[ultra thick] ({\cxprev},{\cyprev}) -- ({\cxprev+cos(\labelpos)/2}, {\cyprev+sin(\labelpos)/2});

            \def\cxprev{0}
            \def\cyprev{0}
            
            \def\cx{-sqrt(2)}
            \def\cy{sqrt(2)}

            \def\lpa{180}
            \def\lpb{150}
            \def\avglp{(\lpa+\lpb)/2}
            \node [circle, draw, minimum size=4cm] (c4) at ({\cx}, {\cy}) {};
            \def\labelpos{\lpa};
            \markedpoint{c4}{\labelpos}{}{1};
            \draw[ultra thick] ({\cx+cos(\avglp)/2},{\cy+sin(\avglp)/2}) -- ({\cx+cos(\labelpos)},{\cy+sin(\labelpos)});
            \def\labelpos{\lpb};
            \markedpoint{c4}{\labelpos}{}{2};
            \draw[ultra thick] ({\cx+cos(\avglp)/2},{\cy+sin(\avglp)/2}) -- ({\cx+cos(\labelpos)},{\cy+sin(\labelpos)});
            \def\labelpos{90};

            \draw[ultra thick] ({\cx},{\cy}) -- ({\cx+cos(\avglp)/2},{\cy+sin(\avglp)/2});

            \markedpoint{c4}{\labelpos}{}{3};
            \draw[ultra thick] ({\cx}, {\cy}) -- ({\cx+cos(\labelpos)},{\cy+sin(\labelpos)});
        
            \def\labelpos{135};
            \node[blue, 
                 inner sep = 0pt, 
                 minimum size=1.5mm, 
                 label = {\labelpos : $\infty$ }  
                 ] at (c1.\labelpos) 
                  {};
            \draw[blue, ultra thick] ({\cx},{\cy}) -- ({\cxprev+cos(\labelpos)/2}, {\cyprev+sin(\labelpos)/2});
            \draw[ultra thick] ({\cxprev},{\cyprev}) -- ({\cxprev+cos(\labelpos)/2}, {\cyprev+sin(\labelpos)/2});

        \end{tikzpicture}
    }
    \caption{Left: The tree $\littlecell(0, 1, 2, 5, 6, 6.5)$, obtained by running the distance algorithm on the successive differences (Def. \ref{def:dist-alg-diff-to-rtree}). Right: A stable curve $(C; x_{\bullet})$ shown approximately to scale, with the tree $\littlecell(C; x_{\bullet})$ overlaid. Each component is coordinatized by setting the point nearest to $x_{n+1}$ to $\infty$ and the tree on each component is calculated as in the distance algorithm. These trees are attached along the three highlighted edges shown in blue. The resulting tree is only well-defined up to flipping each irreducible component, i.e. flipping at the child of each highlighted edge.}
    \label{fig:tree-algorithm-on-stable-curve}
\end{figure}
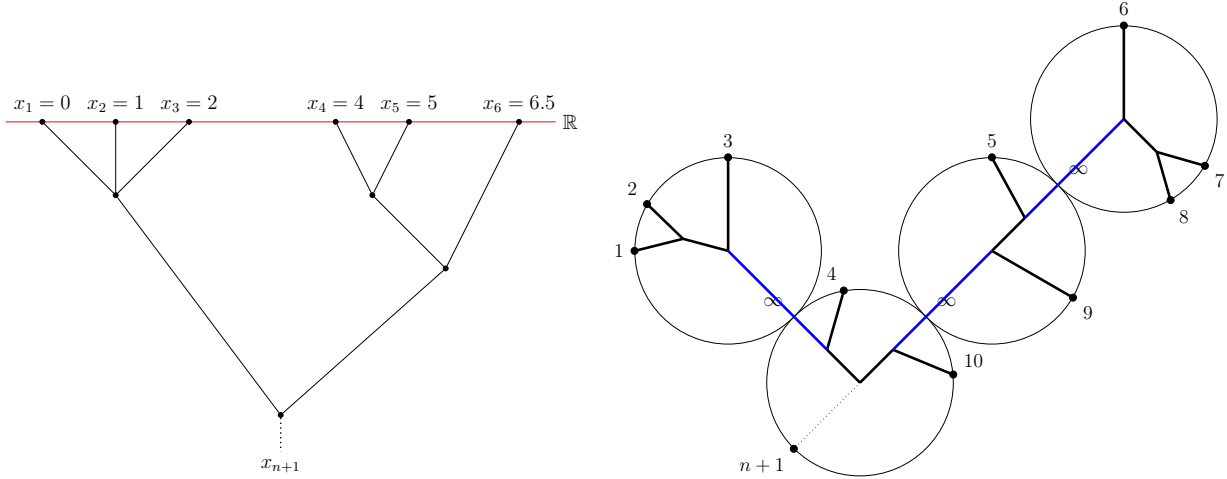

    \subsubsection{Dual cubes and refined cubes}\label{subsec:dual-refined-cubes} We now describe the cells in our dual stratification of both $\overline{M_{0, n+1}}(\R)$.
    For a stable tree $\tau \in \StRtree([n])$, we define the \textbf{dual cube}
    \[
    W_\tau = \{(C; x_\bullet) \in \DMumford{n+1}(\R) : \bigcell(C; x_\bullet) = \tau\}.
    \]
    For refined stable $\tau\in \RStRtree([n])$, we similarly define the \textbf{refined cube}
    \[
    R_\tau = \{(C; x_\bullet) \in \DMumford{n+1}(\R) : \littlecell(C; x_\bullet) = \tau\}.
    \]
 
    Following the methods in \cite{ilinKLPR2023}, we show that $\overline{W_{\tau}}\cong [-1,1]^{\dim W_{\tau}}$. More precisely, we have the following. 
    
    \begin{prop}\label{prop:big-cube-hoemo-to-prod-intervals}
        Let $\tau\in \StRtree([n])$. Then $\overline{W_{\tau}}$ is homeomorphic to a product of copies of the interval $[-1,1]$; we have
        \begin{equation}\label{eqn:big-cube-codim-formula}
            \dim W_\tau = |E_{\Int}(\tau)|, \text{ or equivalently }
            \codim W_{\tau} = \sum_{v\in V_{\Int}(\tau)} (c(v)-2).
        \end{equation}
        In particular, $\overline{W_{\tau}}$ is top-dimensional if and only if $\tau$ is a binary tree, and is $0$-dimensional if and only if $\tau$ consists only of a root vertex connected to $n$ leaves.

        Let $\tau \in \RStRtree([n])$. Then $\overline{R_\tau}$ is homeomorphic to a product of copies of the interval $[0, 1]$. Let $E \subseteq E_{\Int}(\tau)$ be the set of edges whose child vertex is not flippable. Let $V \subseteq V_{\Int}(\tau)$ be the set of flippable non-root vertices. We have
        \begin{equation}
            \dim R_\tau = |E|, \text{ or equivalently }
            \codim R_{\tau} = |V| + \sum_{v\in V_{\Int}(\tau)} (c(v)-2).
        \end{equation}
    \end{prop}

    We note that the dimensions and codimensions of the dual cells are reversed compared to the standard cells; see Equation \eqref{eqn:dimn-formula-for-stdcell-Xtau}. We sketch the proof of Proposition~\ref{prop:big-cube-hoemo-to-prod-intervals} below.

    \begin{rmk}(Some topological facts about the dual and refined cubes) \label{rmk:big-little-cube-topo-facts} \ 
        \begin{enumerate}
            \item The dual and refined cubes $W_{\tau}, R_{\tau}$ are locally closed. We thus sometimes (by abuse of notation) refer to them as \textit{open} cubes. We call their closures \textit{closed} cubes. \\
            
            \item  Let $\tau \in \RStRtree([n])$. Let $\tau^{\mathrm{dual}} \in \StRtree([n])$ be its image under the map \eqref{eqn:map-dual-rst-st-tree}, making all vertices flippable, and let $\tau^{\mathrm{std}}$ be its image under the map \eqref{eqn:map-std-rst-st-tree}, contracting all edges with non-flippable children. It is immediate that $R_{\tau} = W_{\tau^{\mathrm{dual}}} \cap X_{\tau^{\mathrm{std}}}$. \\

            \item The interior of the moduli space decomposes as $M_{0,n+1}(\R)= \coprod_{\tau} R_{\tau}$, where $\tau$ ranges over the refined stable trees on $[n]$ with only the root marked as flippable.
            
        \end{enumerate}

    \end{rmk}

    If $\tau \in \StRtree([n])$ with $|V_{\Int}(\tau)|=1$, then $W_{\tau}$ is a single point. Specifically, 
    this point is the permutation point 
    $$(\R\Pj^1; \inv{\sigma}(1), \ldots, \inv{\sigma}(n), \infty)$$
    for some $\sigma \in S_n$.
    This choice of $\sigma$ is unique precisely up to right composition with the involution $w_{1,n}$. Hence, there are exactly $\frac{1}{2}n!$ permutation points in $\overline{M_{0,n+1}}(\R)$. Similarly, for $\tau \in \RStRtree([n])$ with every vertex flippable, $R_{\tau}$ is a single point. It corresponds to the unique stable curve $(C; x_\bullet)$ with standard tree $\tau$, such that, in rooted coordinates, every component is isomorphic to a permutation point. We call these \emph{stable permutation points}.

    We now discuss the homeomorphisms of Proposition \ref{prop:big-cube-hoemo-to-prod-intervals}. We begin with the open refined cells $R_\tau$. 
    
    Let $\tau \in \RStRtree([n])$. We first consider the case where where only the root of $\tau$ is flippable. By abuse of notation, we let $\tau$ denote either one of the two possible representatives in $\Rtree([n])$. For each $v\in V_{\Int}(\tau)$, let $i_v, j_v$ be any two successive leaves such that $v$ is the nearest common ancestor of $i_v$ and $j_v$. (That is, for some successive child vertices $w, w'$ of $v$, $i_v$ is the largest leaf of $w$ and $j_v$ is the smallest leaf of $w'$). We define the map 
    \begin{equation}\label{eqn:refined-cell-to-I-Eint}
        \theta_{\tau}: R_{\tau}\to (0,1)^{E_{\Int}(\tau)}\text{ by }
        (\R\Pj^1;x_{\bullet})\mapsto \left(\frac{x_{j_v}-x_{i_v}}{x_{j_u}-x_{i_u}}\right)_{(u,v)\in E_{\Int}(\tau)}\in (0,1)^{E_{\Int}(\tau)}
    \end{equation}
    where the coordinates are chosen such that $x_{n+1}=\infty$ and such that the distance algorithm produces $\tau$ (rather than $\rflip \tau$).
    
    See Figure~\ref{fig:eg-of-open-cube-homeo-theta-tau} for an example. 
    Notice that when a vertex has more than two children, any two consecutively placed children can be used in the calculation. 
    
    We define an inverse map  
    \begin{equation}
    \varphi_{\tau} :(0,1)^{E_{\Int}(\tau)}\to R_{\tau}
    \end{equation}
    as follows: Given $r_e\in (0,1)$ for each $e\in E_{\Int}(\tau)$, we define $(\R\Pj^1; x_{\bullet})\in {M_{0,n+1}}(\R)$ by setting $x_1=0,\ x_{n+1}=\infty$, and for $2\le i\le n$, 
    \begin{equation} \label{eqn:def-varphi-tau}
    x_i=x_{i-1}+\prod_{e} r_e,
    \end{equation}
    where $e$ runs over all the edges on the path from the root to the nearest common ancestor of the $i$-th and $(i{-}1)$-st leaf. By convention, the empty product is $1$, so we obtain $x_i - x_{i{-}1} = 1$ whenever the nearest common ancestor is the root.

    \begin{lem}\label{lem:little-cube-interior-only-homeo-to-cube}
        Let $\tau \in \RStRtree([n])$ where only the root is flippable. Then the map 
        $$\theta_{\tau}: R_\tau \overset{\cong}{\longrightarrow} (0,1)^{E_{\Int}(\tau)}$$ 
        given by equation~\eqref{eqn:refined-cell-to-I-Eint} is a homeomorphism.
    \end{lem}

    \begin{proof}
        We first check that $\theta_{\tau}$ is well defined. The ratios given in the map are $> 0$ because the $x_i$ are distinct; and are $< 1$ because, under the distance algorithm, vertices closer to the root correspond to strictly larger distances. Furthermore, these ratios are independent of the choices of $(i_v: v\in V_{\Int}(\tau))$, and invariant under coordinate changes preserving $x_{n+1} = \infty$ and the orientation. They are also invariant under flipping $\tau$ at the root: this reverses the order of the $x_i$ and so negates both the numerator and denominator. Continuity is immediate. 
        
        By running the distance algorithm (Algorithm~\ref{def:dist-alg-diff-to-rtree}), we see that $\varphi_{\tau}$ takes values in $ R_{\tau}$ and so is well-defined.
        
        Finally, it is straightforward to check that $\theta_{\tau},\varphi_{\tau}$ are bijective inverses.
    \end{proof}

    \begin{figure}
        \centering
        \begin{tikzpicture}[xscale=1.4, yscale=1.4]
            \draw[red] (-0.5,0) -- (7,0) node[anchor = west] {$\color{black}\R$};
            \node[circle, draw, fill, inner sep =1pt, label = { 90: $x_1=0$ }] (x1) at (0,0) {};
            \node[circle, draw, fill, inner sep =1pt, label = { 90: $x_2=1$ }] (x2) at (1,0) {};
            \node[circle, draw, fill, inner sep =1pt, label = { 90: $x_3=2$ }] (x3) at (2,0) {};
            \node[circle, draw, fill, inner sep =1pt, label = { 90: $x_4=4$ }] (x4) at (4,0) {};
            \node[circle, draw, fill, inner sep =1pt, label = { 90: $x_5=5$ }] (x5) at (5,0) {};
            \node[circle, draw, fill, inner sep =1pt, label = { 90: $x_6=6.5$ }] (x6) at (6.5,0) {};
            \node[circle, draw, fill, inner sep =1pt] (x123) at (1,-1) {};
            \node[circle, draw, fill, inner sep =1pt] (x45) at (4.5,-1) {};
            \node[circle, draw, fill, inner sep =1pt] (x456) at (5.5,-2) {};
            \node[circle, draw, fill, inner sep =1pt] (x123456) at (3.25,-4) {};
        
            \draw[] (x1) -- (x123);
            \draw[] (x2) -- (x123);
            \draw[] (x3) -- (x123);
            \draw[] (x4) -- (x45);
            \draw[] (x5) -- (x45);
            \draw[] (x45) -- node[midway, above right]{$e_3$} (x456);
            \draw[] (x6) --  (x456);
            \draw[] (x123) -- node[midway, above right]{$e_1$} (x123456);
            \draw[] (x456) -- node[midway, above left]{$e_2$} (x123456);
            
    \end{tikzpicture}
        \caption{A tree $\tau$ (where only the root is flippable) with its three internal edges labelled as $e_1, e_2, e_3$; together with a stable curve $(\R\Pj^1; x_{\bullet})\in R_{\tau}$. It follows that $\theta_{\tau}(\R\Pj^1; x_{\bullet})= (r_e)$, where $r_{e_1}=\frac{x_2-x_1}{x_4-x_3}=\frac{x_3-x_2}{x_4-x_3}=\frac{1}{2}$; $r_{e_2}=\frac{x_6-x_5}{x_4-x_3}=\frac{3}{4}$; and $r_{e_3}=\frac{x_5-x_4}{x_6-x_5}=\frac{2}{3}$.
        }
        \label{fig:eg-of-open-cube-homeo-theta-tau}
    \end{figure}
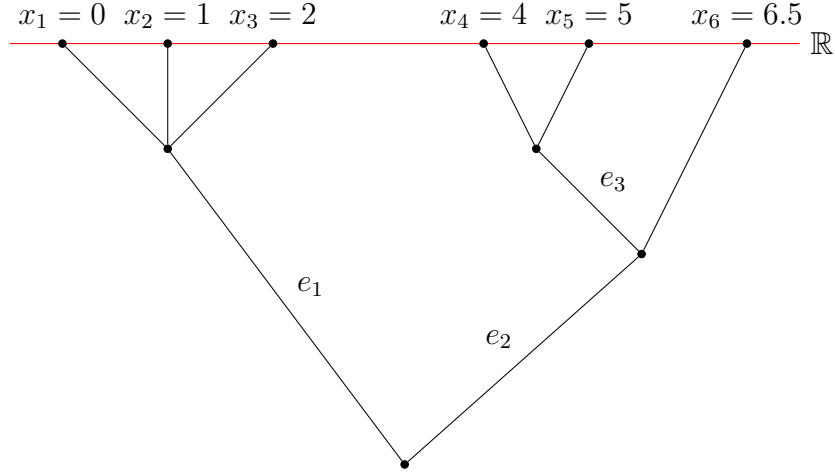

    We next consider an arbitrary refined tree $\tau\in \RStRtree([n])$. Let $v_0, v_1,\ldots, v_k$ be the distinct flippable vertices of $\tau$ and $v_0$ as the root. We associate to $\tau$ a sequence of refined stable trees $\tau_i \in \RStRtree(A_i)$, where
    \begin{itemize}
        \item the $A_i$ partition the set $[n]\coprod \{v_1, \ldots, v_k\}$;
        \item for each $i$, $\tau_i$ is formed by taking the subtree of $\tau$ rooted at $v_i$ and successively replacing each $v_j$ present in that subtree (along with the subtree rooted at $v_j$) by a leaf labelled $v_j$;
        \item each $\tau_i$ has only its root marked as flippable.
    \end{itemize}

    \begin{lem}\label{lem:little-cube-incl-boundary-homeo-to-cube}
        Let $\tau \in \RStRtree([n])$. Let $v_0, \ldots, v_k$ and $\tau_0, \ldots, \tau_k$ be defined as above. Then we have a homeomorphism $R_\tau \cong \prod_i R_{\tau_i}.$

        In particular, let $E \subseteq E_{\Int}(\tau)$ be the set of internal edges whose child vertex is not flippable. Let $E_i = E \cap E_{\Int}(\tau_i)$. Then we have homeomorphisms
        $$\theta_{\tau}= \prod_{i=0}^k \theta_{\tau_i}: R_{\tau}\to (0,1)^E\text{ and } \varphi_{\tau}=\prod_{i=0}^k \varphi_{\tau_i}:(0,1)^E \to R_{\tau},$$
        which are inverses of each other.
    \end{lem}

    \begin{proof}
        This follows from Lemma~\ref{lem:little-cube-interior-only-homeo-to-cube} and the definition of $\littlecell$ on the boundary of the moduli space discussed in Section~\ref{subsec:marked-curves-to-trees}. 
    \end{proof}

    We next consider closures of the refined cells $R_\tau$. We will use the partial order $\succeq$ on $\RStRtree([n])$ defined as the transitive closure of the following covering relation.
    
    Let $\tau, \tau'\in \RStRtree([n])$. We say $\tau$ \emph{covers} or \emph{borders} $\tau'$
     
    if $\tau'$ is formed from $\tau$ by either contracting an internal edge $(u,v)$ where $v$ is not flippable; or by marking exactly one additional internal vertex as flippable. By Lemma~\ref{lem:little-cube-incl-boundary-homeo-to-cube}, 
    $$\dim R_{\tau} - \dim R_{\tau'} =1.$$
    
    If $\dim R_{\tau_1} = \dim R_{\tau''}$, we say that $\tau$ and $\tau''$ are \textit{adjacent} if there is some $\tau'\in \RStRtree([n])$ covered by both $\tau$ and $\tau''$. Such a $\tau'$ is unique if it exists and we call it the \textit{border between $\tau$ and $\tau''$}.
    
    We now consider closed refined cubes. 

    \begin{prop}
        The map $\varphi_\tau$ of Lemma~\ref{lem:little-cube-interior-only-homeo-to-cube} extends to a homeomorphism $[0, 1]^E \to \overline{R_\tau}$.
    \end{prop}
    \begin{proof}
    We first observe that setting a coordinate to 1 corresponds to an edge contraction of the tree produced by the distance algorithm, but does not change the tree of components, so continuity is immediate (working in a chart of $\DMumford{n+1}(\R)$). Thus $\varphi_\tau$ extends to $(0,1]^E \to \overline{R_\tau}$, which remains injective.
    
    In contrast, from the definition of $\littlecell$, setting a coordinate to $0$ causes certain special points on the stable curve to collide and form a new component. Combinatorially, this corresponds to marking a vertex as flippable. Continuity here is not so straightforward, but can be seen essentially by examining the speed of which these marked points collide, i.e. the structure of $\DMumford{n+1}$ as an iterated blowup. Below, we sketch the proof of continuity at the origin $\vec{0} \in [0, 1]^{\dim R_\tau}$, in the case where $\tau$ is a binary tree with only the root marked as flippable. In this case we are approaching the unique stable curve $C$ corresponding to the standard $0$-cell whose tree is $\tau$ (with all internal vertices marked as flippable). The proof of the general case is similar, and the proof of bijectivity is straightforward from the recursive structure of the distance algorithm. Thus, once continuity is established, the map $[0, 1]^E \to \overline{R_\tau}$ is a continuous bijection from a compact space to a Hausdorff space, hence a homeomorphism.
    
    For continuity, with the assumptions above, the point $[C] \in \DMumford{n+1}$ is the transverse intersection of $n-2$ smooth divisors corresponding to the vertices $v \in V_{\Int}(\tau)$ other than the root. For each such $v$, there is a local equation $x_v$ of the corresponding divisor. That is, in a neighborhood of $[C]$, $x_v = 0$ for a stable curve $C'$ if and only if $C'$ has a node separating the leaves below $v \in \tau$ from the remaining leaves. 
    The local ring of $\DMumford{n+1}$ at $[C]$
    is then isomorphic to the localized polynomial ring
    \[
    k[x_v : v \in V_{\Int}(\tau) \text{ is not the root}]_{(x_v \:\ v \text{ is not the root})}.
    \]
    By abuse of notation, for the root vertex $v$ we put $x_v := 1$. For $C' \in M_{0, n+1}$, the relationship between the $x_v$'s and the coordinates $(y_\bullet)$ of the marked points $x_i$ is given as follows. For each $2 \leq i \leq n$, let $v_i \in \tau$ denote the nearest common ancestor of the leaves $i-1$ and $i$. Let $\ell(i)$ be the leftmost leaf of the subtree rooted at $v_i$; note that $\ell(i) \leq i-1$. Then consider the $(n{+}1)$-tuple $(y_\bullet)$ with $y_1 = 0$, $y_{n+1} = \infty$ and, for $2 \leq i \leq n$,
    \[
    y_i = y_{\ell(i)} + \prod_{v \geq v_i} x_v,
    \]
    where $\tau$ is equipped with the partial order~\eqref{eqn:rtree-partial-order}. This product corresponds to the path from the root to $v_i$. See Example~\ref{ex:xv-yv'-bintree-tau}. If all the $x_v$ are positive and sufficiently close to $0$, this gives $(\R\Pj^1, y_\bullet) \in R_\tau$. 

    We now compare with the distance algorithm. For each $w \in V_{\Int}(\tau)$ that is not the root, write $w = v_{i+1}$ and let $v_{j+1}$ be the parent of $w$. Let
    \begin{equation}\label{eqn:yw'-as-rational-fn-of-xv}
        y'_w=\frac{y_{i+1}-y_i}{y_{j+1}-y_j}\in \R(x_v).
    \end{equation}
    For the root $w$, we again put $y'_w := 1$. It immediately follows that $y_{i+1}=y_i+\prod_{v\ge v_{i+1}}y'_v$. Thus, identifying each $w$ with its parent edge, the mapping $\varphi_\tau$ of Equation \eqref{eqn:def-varphi-tau} is identified with $(y'_w)_w \mapsto (y_\bullet)$.
    
    To see that $\varphi_\tau$ extends continuously to $(y'_w)_w = \vec{0}$, we solve for the $x_w$'s in terms of the $y'_w$'s and show that $(y'_w)_w \mapsto (x_v)_v$ is continuous as $(y'_w)_w \to \vec{0}$.
    For $v \in V_{\Int}(\tau)$, let $\rank(v)$ be the number of edges in a longest path from $v$ to any leaf. We extend the partial order \eqref{eqn:rtree-partial-order} to a total order, \emph{reverse level order} $\le_{RL}$ on $V_{\Int}(\tau)$, by letting $v_i\le_{RL} v_j$ if and only if either $\rank(v_i)<\rank(v_j)$ or both $\rank(v_i)=\rank(v_j)$ and $i\le j$. One shows that for each $w\in V_{\Int}(\tau)$ that is not the root, there exist $f,g\in \R[x_v: v<_{RL}w]$ with constant term $1$, such that 
    \begin{equation}\label{eqn:yw-to-f-g-xw}
        y_w' = \begin{cases}
            \frac{x_wf}{1-x_wg},& w\text{ is a left child}\\
            \frac{x_wf}{g},& w\text{ is a right child}.
        \end{cases}
    \end{equation}
    Solving for $x_w$ in Equation~\eqref{eqn:yw-to-f-g-xw} yields: 
    \begin{equation}\label{eqn:xw-to-f-g-yw}
        x_w = \begin{cases}
            \frac{y_w'}{f+y_w'g},& w\text{ is a left child}\\
            \frac{y_w'g}{f},& w\text{ is a right child}.
        \end{cases}
    \end{equation}
    Inductively, $\tfrac{x_w}{y'_w}$ is a regular function of the $y'_v$ with $v \le_{RL} w$, with value $1$, as $(y'_w) \to \vec{0}$.
    \end{proof}

    \begin{exmp}\label{ex:xv-yv'-bintree-tau}
        $$\tau=\vcenter{
        \hbox{
        \begin{tikzpicture}
            \node [circle, draw, fill, inner sep =1pt, label = { 90: 1 }] (node_0_0) at (0,0) {};
            \node [circle, draw, fill, inner sep =1pt, label = { 90: 2 }] (node_1_0) at (1,0) {};
            \node [circle, draw, fill, inner sep =1pt, label = { 90: 3 }] (node_2_0) at (2,0) {};
            \node [circle, draw, fill, inner sep =1pt, label = { 90: 4 }] (node_3_0) at (3,0) {};
            \node [circle, draw, fill, inner sep =1pt, label = { 90: 5 }] (node_4_0) at (4,0) {};
            \node [circle, draw, fill, inner sep =1pt, label = { 90: 6 }] (node_5_0) at (5,0) {};
            \node [circle, draw, fill, inner sep =1pt, label = { 90: 7 }] (node_6_0) at (6,0) {};
            
            \node [circle, draw, fill, inner sep =1pt, label = {0 : $v_3$} ] (node_1'5_-1) at (1.5,-1) {};
            \node [circle, draw, fill, inner sep =1pt, label = {0 : $v_2$} ] (node_1_-2) at (1,-2) {};
            \node [circle, draw, fill, inner sep =1pt, label = {0 : $v_5$} ] (node_3'5_-1) at (3.5,-1) {};
            \node [circle, draw, fill, inner sep =1pt, label = {0 : $v_7$}] (node_5'5_-1) at (5.5,-1) {};
            \node [circle, draw, fill, inner sep =1pt, label = {0 : $v_6$} ] (node_4'5_-2) at (4.5,-2) {};
            \node [circle, draw, fill, inner sep =1pt, label = {90 : $v_4$} ] (node_3_-3) at (3,-3) {};
            \draw[] (node_3_-3) -- (node_1_-2);
            \draw[] (node_3_-3) -- (node_4'5_-2);
            \draw[] (node_1_-2) -- (node_0_0);
            \draw[] (node_1_-2) -- (node_1'5_-1);
            \draw[] (node_4'5_-2) -- (node_3'5_-1);
            \draw[] (node_4'5_-2) -- (node_5'5_-1);
            \draw[] (node_1'5_-1) -- (node_1_0);
            \draw[] (node_1'5_-1) -- (node_2_0);
            \draw[] (node_3'5_-1) -- (node_3_0);
            \draw[] (node_3'5_-1) -- (node_4_0);
            \draw[] (node_5'5_-1) -- (node_5_0);
            \draw[] (node_5'5_-1) -- (node_6_0);
        \end{tikzpicture}
        }
        }\in \Rtree(1,2,3,4,5,6,7)
        $$
        \par The internal vertices, in reverse level order, are $v_3,v_5,v_7, v_2, v_6, v_4$. Table~\ref{tab:calc-xv-yv'-tau-from-ex} expresses the positions of the marked points $y_i$ and edge coordinates $y'_v$ (identifying each $v$ with its parent edge, and setting $y'_v = 1$ for the root), in terms of the local coordinates $x_v$ of $\DMumford{8}(\R)$ near the stable curve $C$ with standard tree $\tau$. 
        Note that, in particular, each $x_w$ is expressed as a function of $y_w'$ and 
        $(x_v: v<_{RL}w).$
        
        \begin{center}
           
            \renewcommand*{\arraystretch}{2}
            \begin{table}[h]
                \centering
                \begin{tabular}{|c|c|c|c|c|} \hline
                     $i$ & $l(i)$ & $y_i$ & $y'_{v_{i}}$& $x_{v_i}$ \\
                     \hline
                     1 & --- & 0 & --- & ---  \\
                     \hline
                     2 & 1 & $x_{v_2}$ & $\dfrac{x_{v_2}}{1-x_{v_2}(1+x_{v_3})}$ & $\dfrac{y'_{v_2}}{1+y'_{v_2}(1+x_{v_3})}$\\[1ex]
                     \hline
                     3 & 2 & $x_{v_2}(1+x_{v_3})$& $x_{v_3}$& $y_{v_3}'$\\
                     \hline
                     4 & 1 & 1& 1&1\\
                     \hline
                     5 & 4 & $1+x_{v_5}x_{v_6}$& $\dfrac{x_{v_5}}{1-x_{v_5}}$ & $\dfrac{y'_{v_5}}{1+y'_{v_5}}$\\[1ex]
                     \hline
                     6 & 4 & $1+x_{v_6}$& $\dfrac{x_{v_6}(1-x_{v_5})}{1-x_{v_2}(1+x_{v_3})}$ & $y'_{v_6}\dfrac{1-x_{v_2}(1+x_{v_3})}{x_{v_6}(1-x_{v_5})} $\\[1ex]
                     \hline
                     7 & 6 & $1+x_{v_6}(1+x_{v_7})$& $\dfrac{x_{v_7}}{1-x_{v_5}}$&$y'_{v_7}(1-x_{v_5})$\\[1ex] \hline
                \end{tabular}
                \caption{Computations for the tree $\tau$ in Example~\ref{ex:xv-yv'-bintree-tau}}
                \label{tab:calc-xv-yv'-tau-from-ex}
            \end{table}
        \end{center}
    \end{exmp}

    Summarizing, we have the following description of the closed refined cubes.

    \begin{thm}\label{thm:WLtau-extend-along-border}
        Let $\tau\in \RStRtree([n])$ and $k=\dim R_{\tau}$. Let $E \subseteq E_{\Int}(\tau)$ be the set of edges $e$ whose child vertex is not flippable. Then the homeomorphism $\theta_{\tau}$ extends to a homoeomorphism 
        $$\overline{\theta_{\tau}}: \overline{R_\tau} = \coprod_{\tau'\preceq \tau} R_{\tau'}\to [0,1]^E.$$
        Explicitly, if $\tau' \preceq \tau \in \RStRtree([n])$, let
        $$ t_e=
        \begin{cases}
            1, & \text{if $e$ is contracted in the process of forming $\tau'$ from $\tau$,}\\
            0, & 
            \text{if the child vertex of $e$ is marked flippable in $\tau'$.}
        \end{cases}
        $$
        Then we have
        \[
        \overline{\theta_\tau}(R_{\tau'}) = (0, 1)^{E'} \times \prod_{e \in E \setminus E'} \{t_e\},
        \]
        where $E' \subseteq E$ is the set of edges remain in $\tau'$, with non-flippable child vertex.

    \end{thm}

    Finally, we describe how the refined cubes $R_\tau$ merge into dual cubes $W_\tau$. Let $\tau, \tau'\in \RStRtree([n])$ be adjacent with border $\tau''$. Then we have $W_{\tau} = W_{\tau'}$ if and only if $\tau''$ is formed from $\tau$ by marking a vertex as flippable and $\tau'$ is obtained by flipping $\tau$ at that vertex. It follows that the restrictions of $\overline{\theta_\tau}$ and $\overline{\theta_{\tau'}}$ agree on $\overline{R_{\tau''}}$, so we may glue the homeomorphisms $\overline{\theta_\tau}$ and $\overline{\theta_{\tau'}}$. It is convenient to do so by formally negating the corresponding edge coordinate.

    \begin{thm}\label{thm:cl-dual-cell-homeo-to-pm1-to-E}
        Let $\tau\in \StRtree([n])$. There are compatible homeomorphisms
        \[
        \theta_\tau : W_\tau \to (-1, 1)^{E_{\Int}(\tau)}, \qquad
        \overline{\theta_{\tau}}: \overline{W_{\tau}} \to[-1,1]^{E_{\Int}(\tau)}.
        \]
        More explicitly, let $\mcT \subseteq \RStRtree([n])$ be the preimage of $\tau$ under the map \eqref{eqn:map-dual-rst-st-tree} marking all vertices flippable, and fix $\tau' \in \mcT$ with all internal edges non-flippable. Then $\overline{\theta_\tau}$ can be chosen to map $\overline{R_{\tau'}}$ to the first quadrant, extending the map $\overline{\theta_{\tau'}}$ of Theorem~\ref{thm:WLtau-extend-along-border}. Flipping a vertex of $\tau'$ negates the corresponding edge coordinate, and setting a coordinate to $0$ corresponds to marking a vertex as flippable; in particular
        \[
        W_\tau = \coprod_{\tau' \in \mcT} R_{\tau'}.
        \]
    \end{thm}

    \subsubsection{Cubes in the double cover $\overline{\DC{M}_{0,n+1}}(\R)$ and the cactus presentation of $J_n$}\label{subsec:cubes-in-2:1-cover-unweighted}

    The cell decomposition of $\DMumford{n+1}(\R)$ lifts straightforwardly to the double cover $\overline{\DC{M}_{0,n+1}}(\R)$.
    We define $\DC{\RStRtree}([n])$ and $\DC{\StRtree}([n])$ analogously to $\RStRtree([n])$ and $\StRtree([n])$, respectively, where the root is \textit{never} marked as flippable. The remainder of the discussion proceeds analogously to the above. 

    As discussed in Section \ref{subsec:G-eq-pi1}, $S_n$ acts simply transitively on the $0$-cells of the dual decomposition of the double cover $\overline{\DC{M}_{0,n+1}}(\R)$. As such, a presentation of the oriented cactus group $\tilde J_n$ (and hence of the ordinary cactus group) can be obtained by examining the $2$-skeleton as in Theorem~\ref{thm:pres-pi1GX-simply-transitive}. We briefly discuss this calculation.

    Let $x_0\in \DC{M}_{0,n+1}(\R)$ be the oriented permutation point in which the $i$-th marked point is $i$ 
    
    for $1\le i\le n$, and $x_{n+1} = \infty$. The generators in the presentation are given by the $1$-cells with one end attached to $x_0$. These are attached to $w_{p,q} x_0$ on the other end, for some $1 \leq p < q \leq n$ with $(p,q)\neq (1,n)$. Such a 1-cell corresponds to a tree with a single internal edge, whose child subtree contains the $p$-th to $q$-th leaves. We write $s_{p, q}$ for the corresponding generator. See Figure~\ref{fig:1-cell-tree-of-moduli}.

    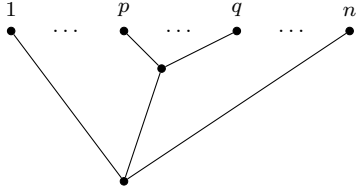
\begin{figure}
        \centering

    \begin{tikzpicture}
            \node [circle, draw, fill, inner sep =1pt, label = { 90: \tiny$1$ }] (node_0'5_-1) at (2.5,0) {};
            \node [circle, draw, fill, inner sep =1pt, label = { 90: \tiny{$p$} }] (node_11_0) at (4,0) {};
            \node [circle, draw, fill, inner sep =1pt, label = { 90: \tiny{$q$} } ] (node_12'5_-1) at (5.5,0) {};
            \node [circle, draw, fill, inner sep =1pt, label = { 90: \tiny$n$ } ] (node_14'5_-1) at (7,0) {};

            \node [circle, draw, fill, inner sep =1pt, ] (node_mrint) at (4.5,-0.5) {};
            \node [circle, draw, fill, inner sep =1pt, ] (node_7'5_-2) at (4,-2) {};

            \draw[draw=none] (node_0'5_-1) -- node[midway, sloped]{{\tiny\ldots}} (node_11_0);
            \draw[draw=none] (node_11_0) -- node[midway, sloped]{{\tiny\ldots}} (node_12'5_-1);
            \draw[draw=none] (node_12'5_-1) -- node[midway, sloped]{{\tiny\ldots}} (node_14'5_-1);

    
            \draw[] (node_7'5_-2) -- (node_0'5_-1);
            \draw[] (node_7'5_-2) -- (node_14'5_-1);
            \draw[] (node_7'5_-2) -- (node_mrint);
            \draw[] (node_mrint) -- (node_11_0);
            \draw[] (node_mrint) -- (node_12'5_-1);
    \end{tikzpicture}
        \caption{Tree that indexes the the 1-cell connecting $x_0$ and $w_{p,q} x_0$. The internal vertex is flippable.}
        \label{fig:1-cell-tree-of-moduli}
    \end{figure}

    The relations are given by the $2$-cells with closures containing $x_0$, corresponding to trees with exactly two internal edges. There are two kinds of $2$-cell, as illustrated in Figure~\ref{fig:2-cell-trees-of-moduli}. In the first kind, both internal edges are connected to the root. The leaves on the two subtrees are given by two disjoint intervals $m, \ldots, r$ and $p, \ldots, q$ where $1 \leq m < r < p < q \leq n$. The second kind has a pair of nested internal edges. The leaves on the subtrees are therefore described as a pair of nested intervals $p, \ldots, q$ and $m, \ldots, r$ where $1 \leq p \leq m < r \leq q \leq n$. By contracting edges, we see that the first kind gives the commutativity relations $s_{p,q} s_{m,r} = s_{m,r} s_{p,q}$, while the second gives the cactus relation $s_{p,q}s_{m,r}=s_{p+q-r,p+q-m}s_{p,q}$.

    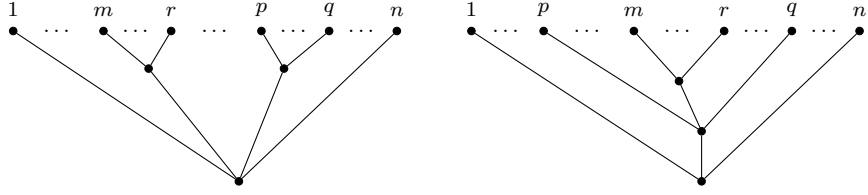
\begin{figure}
        \centering

        \begin{tikzpicture}[xscale=0.6]
        \preambleDisjointUncontractedCell
            \centerTree{$m$}{$r$}{$p$}{$q$}
        \end{tikzpicture}
        \quad 
        \begin{tikzpicture}[xscale=0.6]
        \preambleNestedUncontractedCell
            \centerTree{$p$}{$m$}{$r$}{$q$}
        \end{tikzpicture}

        \caption{Left: Tree corresponding to disjoint intervals $[m,r]$ and $[p,q]$. Right: Tree corresponding to nested intervals $[m,r]\subseteq[p,q]$.}
        \label{fig:2-cell-trees-of-moduli}
    \end{figure}

    To obtain the presentation for the (unoriented) cactus group, we note that the image of $x_0$ in $\overline{M_{0, n+1}}(\mathbb{R})$ has stabilizer $\{\id, w_{1,n}\}$. We accordingly define, as in the setup section,
    $$s_{1, n} := (w_{1,n}, \id)\in \pi_1^{S_n}(\DMumford{n+1}(\mathbb{R}), x_0),$$
    which gives the relation $s_{1, n} s_{p, q} = s_{n+1-q, n+1-p}s_{1,n}$.

    \subsection{Dual Stratification of $\overline{M_{0,\mcA(a)}}(\R)$}\label{subsec:dual-Mod-space-weighted-strat}

    We now generalize the above discussion to the weighted moduli space. The resulting dual cells are indexed by the set of $a$-stable trees, $\StRtree([n]; a)$. Each cell can be described as a union of images of cells of $\DMumford{n+1}(\R)$ under the reduction map 
    $$\rho: \DMumford{n+1}(\R)\onto \hassett{a}(\R).$$
    In general, unions and images of cells need not again be cells (in particular, they need not be contractible). Thus, rather than simply listing the cells to be merged (which we give below as Corollary ... ), we proceed by modifying the distance algorithm itself and using it to directly parametrize the dual cells. Our description yields explicit homeomorphisms from these cells to products of permutahedra.

    \subsubsection{Weighted distance algorithm}

We first describe the weighted variant of the distance algorithm. We let $\vec{d} \in \mathbb{R}^{n-1}_{\geq 0}$ and $\sigma \in S_n$ as in the unweighted case.

\begin{defn}[Distance algorithm, $a$-weighted version]\label{alg:a-weighted-diff-vec-to-rtree}
    Let $\vec{d}$ and $\sigma \in S_n$ be as above. We construct $\tau$ as follows. We begin with $n$ isolated vertices labeled by the singleton sets $\{\sigma(1)\}, \ldots, \{\sigma(n)\}$, from left to right, each considered as a rooted tree with one vertex. We join these trees together by repeating the following steps until $\vec{d}$ is empty.
    \begin{enumerate}
        \item[(1)] Let $d:= \min d_i$ be the minimum distance remaining in $\vec{d}$.
        \item[(2)] For each sequence of consecutive copies of $d$ in $\vec{d}$, say $d_i = d_{i+1} = \cdots  = d_j$, attach the roots of the $i$-th through $(j+1)$-st subtrees as ordered children of a new, unlabeled vertex $v$. Then delete $d_i, \ldots, d_j$ from $\vec{d}$.
        \item[(3)] (Compression step) Let $A_1, \ldots, A_k \subseteq [n]$ be the labels of the leaves of the newly created subtree $\tau_v$. If $\sum_{i=1}^k |A_i| \leq a$, replace $\tau_v$ by a single leaf labeled $\bigcup_{i=1}^k A_i$.
    \end{enumerate}
    We denote the resulting tree by $\tau^{\mathrm{dist}}_a(\sigma, \vec{d})$; it is an element of $\Rtree([n]; a)$ compatible with $\sigma$.
\end{defn}
The $a$-weighted distance algorithm is equivalent to applying the $a$-compression map $\varpi_a$ to the output of the unweighted algorithm (at least if all $d_i > 0$). We note that, inductively, the compression step can only occur if the subtrees being merged are, themselves, leaves.

        \subsubsection{Subsets of $\overline{M_{0, \mcA(a)}}(\R)$ indexed by $a$-stable trees} 
        
        Recall that for the unweighted moduli space $\overline{M_{0,n+1}}(\R)$, we have described both refined cubes, indexed by refined stable trees, and dual cubes, indexed by stable trees (on $[n]$). In the weighted case, for simplicity, we focus only on the dual cells and omit the refined cells. We define a map
        \[
        \bigcell_a : \DMumford{\mcA(a)}(\R) \to \StRtree([n]; a)
        \]
        that serves a similar purpose to $\bigcell$ defined in Section~\ref{subsec:dual-Mod-space-unweighted-strat}.

        If $(C; x_\bullet)\in \overline{M_{0,\mcA(a)}}(\R)$ is smooth, we choose any coordinates on $C \cong \R\Pj^1$ for which the marked point $x_{n+1} = \infty$. Since $x_{n+1}$ has weight $1$, all the other marked points then have finite coordinate values; they are therefore ordered as
        \[
        x_{\sigma(1)} \le \cdots \le x_{\sigma(n)} \text{ for some } \sigma \in S_n,
        \]
        where at least one of the inequalities is strict. In addition, if $x_{\sigma(i)}=x_{\sigma(i+1)}=\ldots = x_{\sigma(j)}$, then $j-i+1\le a$. We let $\vec{d}$ be the vector of successive differences
        \[
        \vec{d} = (x_{\sigma(2)} - x_{\sigma(1)}, \ldots, x_{\sigma(n)} - x_{\sigma(n-1)}).
        \]
        We then define $\bigcell_a(C;x_{\bullet})\in \StRtree([n]; a)$ to be the equivalence class of $\tau^{\mathrm{dist}}(\sigma, \vec{d})$ under marking all internal vertices as flippable. If $C$ is not smooth, then similarly to the unweighted case, we compute the tree for each component using rooted coordinates, then attach the trees according to how the components of $C$ are attached. It can be checked that $\bigcell_a = \varpi_a\circ \bigcell$ and that $\bigcell_a(C;x_{\bullet})$ is independent of the choice of $\sigma$.

        For an $a$-stable tree $\tau \in \StRtree([n]; a)$, we define the \textbf{dual cell}
        \[
        W_\tau := \{(C; x_\bullet) \in \DMumford{\mcA(a)}(\R) : \bigcell_a(C; x_\bullet) = \tau\}.
        \]

        To describe the cell closure, we describe a partial order $\succeq$ on $\Rtree([n]; a)$ and one on $\StRtree([n]; a)$ as follows. Let $\tau, \tau' \in \Rtree([n]; a)$. Then $\tau \succeq \tau'$ if $\tau'$ can be obtained from $\tau$ by any combination of contracting internal edges of $\tau$ and splitting leaves of $\tau$ into sequences of consecutive leaves (and accordingly decomposing the leaf label into disjoint subsets). For $a$-stable trees $\tau, \tau' \in \StRtree([n]; a)$, we say $\tau \succeq \tau'$ if there exist representatives in $\Rtree([n]; a)$ for which $\tau \succeq \tau'$.

        \begin{exmp}
        We have 
        $$\vcenter{
            \hbox{
                \begin{tikzpicture}[xscale=2]
                    \node [circle, draw, fill, inner sep =1pt, label = { 90: $\{1, 2, 3\}$ }] (node_0_0) at (0,0) {};
                    \node [circle, draw, fill, inner sep =1pt, label = { 90: $\{4, 5, 6, 7\}$ }] (node_1_0) at (1,0) {};
                    \node [circle, draw, fill, inner sep =1pt, label = { 90: $\{8, 9\}$ }] (node_2_0) at (2,0) {};
                    \node [circle, draw, fill, inner sep =1pt, ] (node_1'5_-1) at (1.5,-1) {};
                    \node [circle, draw, fill, inner sep =1pt, ] (node_0'75_-2) at (0.75,-2) {};
                    \draw[blue] (node_0'75_-2) -- (node_0_0);
                    \draw[] (node_0'75_-2) -- (node_1'5_-1);
                    \draw[green] (node_1'5_-1) -- (node_1_0);
                    \draw[red] (node_1'5_-1) -- (node_2_0);
                \end{tikzpicture}
            }
        } \succeq
        \vcenter{
            \hbox{
                \begin{tikzpicture}[xscale=1.5]
                    
                    \node [circle, draw, fill, inner sep =1pt, label = { 90: $\{1, 2, 3\}$ }] (node_0_0) at (0,0) {};
                    \node [circle, draw, fill, inner sep =1pt, label = { 90: $\{4\}$ }] (node_1_0) at (1,0) {};
                    \node [circle, draw, fill, inner sep =1pt, label = { 90: $\{5, 6\}$ }] (node_2_0) at (2,0) {};
                    \node [circle, draw, fill, inner sep =1pt, label = { 90: $\{7\}$ }] (node_3_0) at (3,0) {};
                    \node [circle, draw, fill, inner sep =1pt, label = { 90: $\{8\}$ }] (node_4_0) at (4,0) {};
                    \node [circle, draw, fill, inner sep =1pt, label = { 90: $\{9\}$ }] (node_5_0) at (5,0) {};
                    \node [circle, draw, fill, inner sep =1pt, ] (node_1'5_-2) at (2.5,-2) {};
                    \draw[blue] (node_1'5_-2) -- (node_0_0);
                    \draw[green] (node_1'5_-2) -- (node_1_0);
                    \draw[green] (node_1'5_-2) -- (node_2_0);
                    \draw[green] (node_1'5_-2) -- (node_3_0);
                    \draw[red] (node_1'5_-2) -- (node_4_0);
                    \draw[red] (node_1'5_-2) -- (node_5_0);
                    
                \end{tikzpicture}
            }
        },$$
        because the tree to the right is obtained by splitting certain leaf labels (illustrated by the edge colours) and contracting an internal edge.
        \end{exmp}

        \begin{rmk}\label{rmk:properties-of-weighted-algorithm}
            Let $\rho:\overline{M_{0,n+1}}(\R)\onto \overline{M_{0, \mcA(a)}}(\R)$ be the reduction map and let $\tau\in \StRtree([n]; a)$. The following observations are immediate from the definition of the $a$-weighted distance algorithm and the results of Section \ref{subsec:dual-Mod-space-unweighted-strat}.
            \begin{enumerate}
                \item $\overline{W_{\tau}}=\coprod_{\tau'\preceq \tau} W_{\tau'}$;
                \item ${W_{\tau}}= \bigcup_{\tau'\in \varpi_a^{-1}(\tau)} \rho(W_{\tau'})$; 
                \item $\overline{W_{\tau}}=\bigcup_{\tau'} \rho(R_{\tau'})$, where $\tau'$ ranges over the trees in $\RStRtree([n])$ such that, upon marking all vertices as flippable, we have $\varpi_a(\tau')\preceq \tau$;
                \item $W_{\tau}$ is locally closed in the Hassett space $\overline{M_{0,\mcA(a)}}(\R)$.
            \end{enumerate}
        \end{rmk}
        Thus, each dual cell in our Hassett space is the union of the images of several dual cubes in $\overline{M_{0, n+1}}(\R)$ via the reduction map $\overline{M_{0,n+1}}(\R)\onto \overline{M_{0, \mcA(a)}}(\R)$. We now analyze the structure of the cell.
        
        We first give a product decomposition of $W_{\tau}$, similar in spirit to Lemma \ref{lem:little-cube-incl-boundary-homeo-to-cube}.

        \begin{prop}\label{prop:a-proper-dual-cell-product-decmop}
            Let $\tau \in \StRtree([n]; a)$ and let $A_1, \ldots, A_r \subseteq [n]$ be its leaf labels listed in left to right order. 
            Let $\tau_0 \in \StRtree([r])$ be the (unweighted) stable tree obtained from $\tau$ by replacing the leaf label $A_i$ by the single label $i$, for all $i$. Let $W_{\tau_0} \subseteq \DMumford{n+1}(\R)$ denote the corresponding unweighted dual cube.

            For $1 \leq i \leq r$, let
            $\tau_i = \vcenter{\hbox{
            \begin{tikzpicture}[xscale=0.2, yscale=0.3]
                \trivialTree{$A_i$}{$\{\whitecircle\}$}
            \end{tikzpicture}
        }}$
            be the unique $|A_i|$-stable tree with exactly two leaves, one labeled $A_i$ and one formally labeled $\{\whitecircle\}$, both attached to the root. Let $W_{\tau_i} \subseteq \hassett{|A_i|}(\R)$ denote the corresponding $|A_i|$-stable dual cell in an $(|A_i|+2)$-marked moduli space.
            There are compatible homeomorphisms
            
            \begin{equation}\label{eqn:a-proper-cell-homeo-to-product}
            W_\tau \cong W_{\tau_0} \times \prod_{i=1}^r W_{\tau_i} \quad \text{and} \quad
            \overline{W_\tau} \cong \overline{W_{\tau_0}} \times \prod_{i=1}^r \overline{W_{\tau_i}}.
            \end{equation}
        \end{prop}

        \begin{proof}

            We construct a map
            $$\psi:\overline{W_{\tau}}\longrightarrow \overline{W_{\tau_0}} \times \prod_{i=1}^r \overline{W_{\tau_i}}$$
            and give an explicit description only for smooth curves. We discuss singular curves below.
            
            We denote by * the index of the marked point fixed at $\infty$. For each $j\in [n-1]$ with $j=|A_1|+\ldots+|A_i|$ where $i\in [r-1]$, we let $v_j\in V_{\Int}(\tau)$ denote the nearest common ancestor of the leaves labelled by $A_i$ and $A_{i+1}$. Note that the $v_j$  are distinct if and only if $\tau$ is binary.

            Choose $\sigma$ compatible with $A_{\bullet}$ and coordinates $x_{\bullet}$ on $C$ with $x_{\sigma(1)}\le \ldots \le x_{\sigma(n)}$ and $x_{*}=\infty$. For each $j \in [n-1]$, let $d_j=x_{\sigma(j+1)}-x_{\sigma(j)}$.  Note that, by the definition of the distance algorithm, if $v_j = v_k$ then $d_j = d_k$.

            Let $C=(C;x_{\bullet})\in \overline{W_{\tau}}$ be an arbitrary smooth ($C=\R\Pj^1$) curve. We define
            $$\psi(C)=\bigg( C_i=(C_i, x^{(i)}_{\bullet})\bigg)_{i=0}^r$$
            as follows, with $C_i=\R\Pj^1$ for all $i$.  

            We first describe the marked points on $C_i$ for $1\le i\le r$. By changing representatives of $\tau$ if necessary, assume the leaf labeled $A_i$ has a sibling to its right. Write $A_i=\{\sigma(k), \sigma(k+1), \ldots, \sigma(\ell)\}$ and let $j = |A_1| + \cdots + |A_i| + |A_{i+1}|$, so $v_j$ is the parent of the leaf labeled $A_i$.
            Define 
            $x^{(i)}_{\bullet}:=(x^{(i)}_{\sigma(k)} = 0,
            \ldots, x^{(i)}_{\sigma(\ell)}, x^{(i)}_{\whitecircle}, x^{(i)}_{*}=\infty)$
            by setting
            \begin{align*}
            x^{(i)}_{\sigma(p+1)} &= x^{(i)}_{\sigma(p)} + d_p, \text{ for } p = k, \ldots, \ell-1, \\
            x^{(i)}_{\whitecircle} &= x^{(i)}_{\sigma(\ell)} + d_j.
            \end{align*}
            (If the leaf labeled $A_i$ is a rightmost child, equivalently we may take $j = |A_1| + \cdots + |A_{i-1}|$ and place $x^{(i)}_{\whitecircle}$ left of $x_{\sigma(k)}^{(i)}$ at $-d_j$; this corresponds to flipping $\tau_i$.)

            For the curve $C_0$, we define $x^{(0)}_{\bullet}=(x^{(0)}_1=0, \ldots, x^{(0)}_r, x^{(0)}_*=\infty)$ where $x^{(0)}_{i+1}-x^{(0)}_i=d_j$ with $j=|A_1|+\ldots+|A_i|$.

            For singular curves $C$, the approach is similar to that discussed in Section~\ref{subsec:dual-Mod-space-unweighted-strat}. For $C_0$, we recursively construct $\psi$ on each component $C' \subseteq C$, and attach the $C'_0$ according to how the components $C'$ are attached in $C$. In contrast, for each $i$, $C_i$ depends only on the component $C'$ containing the leaf labeled $|A_i|$.

            It is straightforward to verify that $\psi$ is well-defined and bijective. Given each smooth pointed curve $C_i$, we can reconstruct the smooth curve $C$ with the vector of consecutive distances $(d_1, \ldots, d_{n-1})$ as follows: for each index $j$ for which $v_j$ is an internal vertex, the value of $d_j$ is determined by examining $C_0$. An appropriate choice of $\sigma \in S_n$ is determined by the curves $C_1, \ldots, C_r$. For $1\le i\le r$, normalize  the marked points on $C_i$ to be consistent with the values of $d_j$ that are already computed from $C_0$. The values of $d_j$ for the remaining indices $j\in [n-1]$ are now determined by examining $C_1, \hdots, C_r$. Continuity of $\psi$ holds by a similar argument to that discussed in Section~\ref{subsec:dual-Mod-space-unweighted-strat}. Since $\psi$ is a continuous bijection between compact Hausdorff spaces, it is a homeomorphism. The restriction to the interior claim is easy to verify.
        \end{proof}

\begin{exmp}[$n = 7, a = 3$]
\label{ex:a-proper-dual-cell-product-decmop-n7a3}
We demonstrate the identification $\overline{W_\tau} \cong \overline{W_{\tau_0}} \times \prod_{i=1}^3 \overline{W_{\tau_i}}$, where
\[\tau=
\vcenter{
    \hbox{
        \begin{tikzpicture}[xscale=1.5,yscale=1]
            \node [circle, draw, fill, inner sep =1pt, label = { 90: $\{1, 2, 3\}$ }] (node_0_0) at (0,0) {};
            \node [circle, draw, fill, inner sep =1pt, label = { 90: $\{4, 5\}$ }] (node_1_0) at (1,0) {};
            \node [circle, draw, fill, inner sep =1pt, label = { 90: $\{6, 7\}$ }] (node_2_0) at (2,0) {};
            \node [circle, draw, fill, inner sep =1pt, label = { 180: $v_5$ }] (node_1'5_-1) at (1.5,-1) {};
            \node [circle, draw, fill, inner sep =1pt, label = { 180: $v_3$ }] (node_1_-2) at (1,-2) {};
            \draw[] (node_1_-2) -- (node_0_0);
            \draw[] (node_1_-2) -- (node_1'5_-1);
            \draw[] (node_1'5_-1) -- (node_1_0);
            \draw[] (node_1'5_-1) -- (node_2_0);
        \end{tikzpicture}
    }
}\in \Rtree([7]; 3)\]

and $\tau_0, \tau_1, \tau_2, \tau_3$ are the trees
$$\tau_0=
\vcenter{
    \hbox{
        \begin{tikzpicture}[yscale=0.8]
            \node [circle, draw, fill, inner sep =1pt, label = { 90: $1$ }] (node_0_0) at (0,0) {};
            \node [circle, draw, fill, inner sep =1pt, label = { 90: $2$ }] (node_1_0) at (1,0) {};
            \node [circle, draw, fill, inner sep =1pt, label = { 90: $3$ }] (node_2_0) at (2,0) {};
            \node [circle, draw, fill, inner sep =1pt, ] (node_1'5_-1) at (1.5,-1) {};
            \node [circle, draw, fill, inner sep =1pt, ] (node_1_-2) at (1,-2) {};
            \draw[] (node_1_-2) -- (node_0_0);
            \draw[] (node_1_-2) -- (node_1'5_-1);
            \draw[] (node_1'5_-1) -- (node_1_0);
            \draw[] (node_1'5_-1) -- (node_2_0);
        \end{tikzpicture}
    }
},\ 
\tau_1=\vcenter{
    \hbox{
        \begin{tikzpicture}[xscale=0.4, yscale=0.8]

            \trivialTree{$\{1,2,3\}$}{$\{\whitecircle\}$}
            
        \end{tikzpicture}
    }
},\ 
\tau_2=\vcenter{
    \hbox{
        \begin{tikzpicture}[xscale=0.4, yscale=0.8]

            \trivialTree{$\{4,5\}$}{$\{\whitecircle\}$}
            
        \end{tikzpicture}
    }
},\ 
\tau_3=\vcenter{
    \hbox{
        \begin{tikzpicture}[xscale=0.4, yscale=0.8]

            \trivialTree{$\{\whitecircle\}$}{$\{6,7\}$}
            
        \end{tikzpicture}
    }
}.
$$
In what follows, we denote by * the index of the marked point fixed at $\infty$.
Consider an arbitrary stable curve $(C;x_{\bullet})\in \overline{W_{\tau}}$. For $i=0,1,2,3$, let 
$$(C_i; x^{(i)}_{\bullet})$$ 
denote its projection onto $\overline{W_{\tau_i}}$ under the homeomorphism given in Proposition~\ref{prop:a-proper-dual-cell-product-decmop}.
The curve $(C;x_{\bullet})$ (with $x_*=\infty$ on the component containing it) admits one of two possible tree structures:

\begin{enumerate}
    \item $C=\R\Pj^1$ is smooth. Then for some $\sigma\in S_7$ fixing the sets $\{1,2,3\},\{4,5\},\{6,7\}$, we have $x_{\sigma(1)}\le \ldots \le x_{\sigma(7)}$. Setting $d_i=x_{\sigma(i+1)}-x_{\sigma(i)}$, we have $\max\{d_1, d_2\}\le d_3$, and $\max\{d_4,d_6\}\le d_5\le d_3$. (See Figure~\ref{fig:curve-drawing-a-proper-dual-cell-product-decmop-n7a3} left.) Note that $d_3>0$ due to stability and that $d_5>0$ since the the marked points $x_4,x_5,x_6,x_7$ cannot collide without bubbling off. For $i=0,1,2,3$, we have $C_i=\R\Pj^1$ with 
    
    \begin{alignat*}{2}
    x^{(0)}_{\bullet} & = (x^{(0)}_1, x^{(0)}_2, x^{(0)}_3,  x^{(0)}_{*}) && = (0, d_3, d_3+d_5, \infty)\\
    x^{(1)}_{\bullet} & = (x^{(1)}_{\sigma(1)}, x^{(1)}_{\sigma(2)}, x^{(1)}_{\sigma(3)}, x^{(1)}_{\whitecircle}, x^{(1)}_{*})\ && = (0, d_1, d_1+d_2, d_1+d_2+d_3, \infty)\\
    x^{(2)}_{\bullet} & = (x^{(2)}_{\sigma(4)}, x^{(2)}_{\sigma(5)}, x^{(2)}_{\whitecircle}, x^{(2)}_{*}) && = (0, d_4, d_4+d_5, \infty)\\
    x^{(3)}_{\bullet} & = (x^{(3)}_{\whitecircle}, x^{(3)}_{\sigma(6)}, x^{(3)}_{\sigma(7)}, x^{(3)}_{*}) && = (-d_5, 0, d_6, \infty).
    \end{alignat*}
   
    \item $C$ consists of two components,  $C^*\cong\R\Pj^1$ containing $x_1,x_2,x_3,x_*$; and $C'\cong\R\Pj^1$ containing $x_4,x_5,x_6,x_7$. Let $y= C^*\cap C'$ be the node of $C$, and set $y=\infty$ on $C'$. For some $\sigma\in S_7$ fixing the sets $\{1,2,3\},\{4,5\},\{6,7\}$, we have (see Figure \ref{fig:curve-drawing-a-proper-dual-cell-product-decmop-n7a3} right) 
    \begin{itemize}
        \item on $C^*$: $d_1=x_{\sigma(2)}-x_{\sigma(1)}\ge 0$, $d_2=x_{\sigma(3)}-x_{\sigma(2)}\ge 0$, $d_3=y-x_{\sigma(3)}>0$ with $\max\{d_1,d_2\}\le d_3$;
        \item on $C'$: $d_i=x_{\sigma(i+1)}-x_{\sigma(i)}\ge 0$ for $i=4,5,6$ with $\max\{d_4, d_6\}\le d_5$ and $d_5>0$.
    \end{itemize}
    
    As in the smooth case, it follows that for $i=1,2,3$, $C_i=\R\Pj^1$ with 
    
    \begin{alignat*}{2}
        x^{(1)}_{\bullet} & = (x^{(1)}_{\sigma(1)}, x^{(1)}_{\sigma(2)}, x^{(1)}_{\sigma(3)}, x^{(1)}_{\whitecircle}, x^{(1)}_{*}) && = (0, d_1, d_1+d_2, d_1+d_2+d_3, \infty)\\
        x^{(2)}_{\bullet} & = (x^{(2)}_{\sigma(4)}, x^{(2)}_{\sigma(5)}, x^{(2)}_{\whitecircle}, x^{(2)}_{*}) && = (0, d_4, d_4+d_5, \infty)\\
        x^{(3)}_{\bullet} & = (x^{(3)}_{\whitecircle}, x^{(3)}_{\sigma(6)}, x^{(3)}_{\sigma(7)}, x^{(3)}_{*}) && = (-d_5, 0, d_6, \infty).
    \end{alignat*}
    However, $C_0$ consists of two components, $C_0^*$ containing $x^{(0)}_1,x^{(0)}_*$; and $C_0'$ containing $x^{(0)}_2, x^{(0)}_3$. Let $y^{(0)}$ denote the node of $C_0$, which is set to $\infty$ on $C_0'$. Then up to coordinate change, we have
    \begin{alignat*}{2}
        \text{ on } C_0^* :\ && (x^{(0)}_1,\ y^{(0)},\ x^{(0)}_*) & = (0, d_3, \infty); \\
        \text{ on } C_0' :\ && (x^{(0)}_2, x^{(0)}_3, y^{(0)}) & = (0, d_5, \infty).
    \end{alignat*}
    Note that $(C^{(0)}, x^{(0)}_{\bullet})$ is the unique element in $\overline{W_{\tau_0}}$ that corresponds to a singular curve.  
\end{enumerate}
\end{exmp}

        \begin{figure}[h]
            \centering

\begin{tikzpicture}

    \coordinate (p) at (0, 0);
    \def \arcradius {1.75};
    \def \sep {1.5};
    \def \textradius{1.4};

    \begin{scope}[shift={(p)}]
        \def \radius {2}
        \draw (0, 0) circle[radius=\radius];
        \foreach \a in {270, 170, 150, 130, 75, 60, 30, 15}
            	\path (0, 0) ++(\a:\radius) node {$\bullet$};
        \path (0, 0) ++ (270:\radius-0.1) node[above]	{$\infty$}  +(0, -0.2) node[below] {$*$};
        \path (0,0) ++  
        (130:{\radius*1.1}) 
         node[above left, rotate=60]	{$\{1, 2, 3\}$};
        \path (0, 0) ++ (80:\radius) node[above right]	{$\{4, 5\}$};
        \path (0, 0) ++ (20:\radius) node[above right]	{$\{6, 7\}$};
    
        \foreach \a/\s [remember=\a as \lasta (initially 170)]
        		in {150/$d_1$, 130/$d_2$, 75/$d_3$, 60/$d_4$, 30/$d_5$, 15/$d_6$} {
            \draw[<->] (0, 0) ++ (\lasta-\sep:\arcradius) arc (\lasta:\a+\sep:\arcradius);
            \path (0, 0) ++ (\lasta:\textradius) arc (\lasta: 0.5*(\lasta+\a) : \textradius) node {\s};
        };
    \end{scope}

    \coordinate (s) at (7, -1);
    \def \smallradius {1.7};
    \def \arcradius {1.5};
    \def \sep {1.5};
    \def \textradius{1.2};
        
    \begin{scope}[shift={(s)}]
        \def \radius {\smallradius};
        \draw (0, 0) circle[radius=\radius];
        \foreach \a in {270, 170, 150, 130, 45}
            	\path (0, 0) ++(\a:\radius) node {$\bullet$};
        \path (0, 0) ++ (270:\radius-0.1) node[above]	{$\infty$} +(0, -0.15) node[below] {$*$};
        \path (0,0) ++  
        (130:{\radius*1.1}) 
         node[above left, rotate=60]	{$\{1, 2, 3\}$};
        \path (0, 0) ++ (45:\radius) node[above right]	{$\infty$};
    
        \foreach \a/\s [remember=\a as \lasta (initially 170)]
        		in {150/$d_1$, 130/$d_2$, 45/$d_3$} {
            \draw[<->] (0, 0) ++ (\lasta-\sep:\arcradius) arc (\lasta:\a+\sep:\arcradius);
            \path (0, 0) ++ (\lasta:\textradius) arc (\lasta: 0.5*(\lasta+\a) : \textradius) node {\s};
        };

    \end{scope}
    
    \coordinate (r) at ($(s) + \smallradius*sqrt(2)*(1, 1)$);
    
    \begin{scope}[shift={(r)}]
        \def \radius {\smallradius};
        \draw (0, 0) circle[radius=\radius];
        \foreach \a in {225, 120, 90, 30, 0}
            	\path (0, 0) ++(\a:\radius) node {$\bullet$};
        \path (0, 0) ++ (110:\radius+0.1) node[above]	{$\{4, 5\}$};
        \path (0, 0) ++ (15:\radius) node[right]	{$\{6, 7\}$};
    
        \foreach \a/\s [remember=\a as \lasta (initially 120)]
        		in {90/$d_4$, 30/$d_5$, 0/$d_6$} {
            \draw[<->] (0, 0) ++ (\lasta-\sep:\arcradius) arc (\lasta:\a+\sep:\arcradius);
            \path (0, 0) ++ (\lasta:\textradius) arc (\lasta: 0.5*(\lasta+\a) : \textradius) node {\s};
        };
    \end{scope}
\end{tikzpicture}
\caption{Illustrations of the stable curves described by Example~\ref{ex:a-proper-dual-cell-product-decmop-n7a3}. Left: Smooth case. Right: Singular case with two components.}
\label{fig:curve-drawing-a-proper-dual-cell-product-decmop-n7a3}
\end{figure}
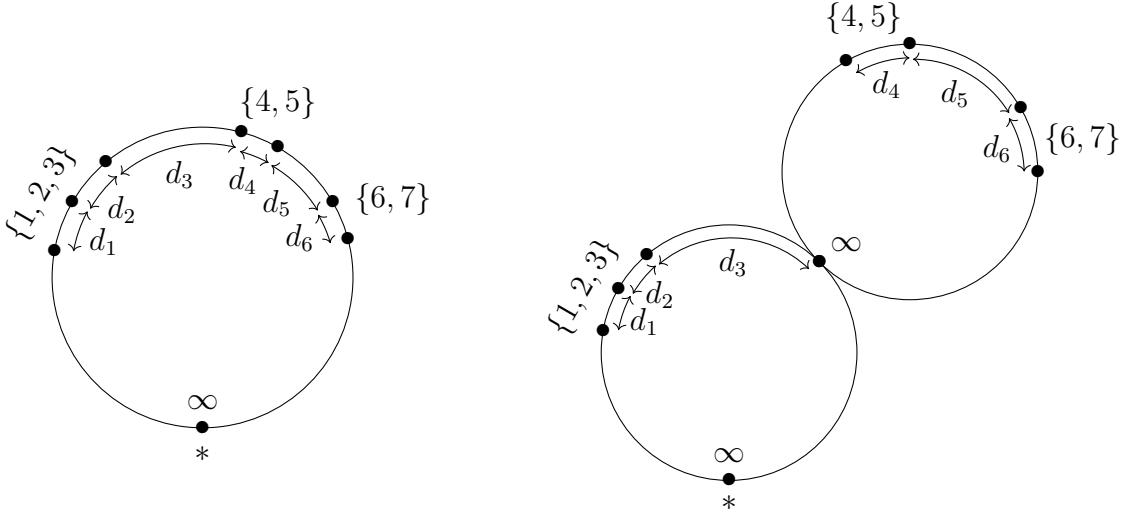

\begin{rmk}
    With notation as in Proposition \ref{prop:a-proper-dual-cell-product-decmop}, we have $\overline{W_{\tau_0}} \cong [-1, 1]^{E_{\Int}(\tau_0)}$ by Theorem \ref{thm:cl-dual-cell-homeo-to-pm1-to-E} from the unweighted case. Thus, in order to show that our weighted dual cells $\overline{W_{\tau}}$ are indeed cells, it remains and is sufficient to consider the case $a = n-1$, with $\tau$ the unique $(n-1)$-stable tree on $[n]$ with one leaf labeled $[n-1]$ and the other labeled $\{n\}$, both attached to the root. We will see that in this case, $\overline{W_{\tau}}$ is naturally identified with the permutahedron.
\end{rmk}

    \subsubsection{Maps between products of permutahedra and Hassett spaces}
    
        The $n$-th order \emph{permutahedron} $\Pi_n\subseteq \R^n$ is the convex hull
        \begin{equation}\label{eqn:def-permutahedron}
        \Pi_n := \mathrm{conv} \{\vec{\sigma}:=(\sigma(1), \sigma(2), \ldots, \sigma(n))\in \R^n: \sigma\in S_n\}.
        \end{equation}
        Note that $\Pi_n$ is an $(n-1)$-dimensional polytope contained in the hyperplane 
        $$\left\{\sum x_i = \frac{n(n+1)}{2}\right\}\subseteq \R^n.$$
        Faces of $\Pi_n$ correspond to compositions of $[n]$. Let $A_{\bullet}=(A_1, \ldots, A_r)$ be a composition of $[n]$, let $\Pi_{A_{\bullet}}$ be the corresponding face and $\vec{o}_{A_{\bullet}}$ its centroid. Then $\Pi_{A_{\bullet}}$ has dimension $n-r$ and is the convex hull of the points $\vec{x}=(x_1, x_2, \ldots, x_n)$ satisfying
        \begin{align*}
            \{x_i:i\in A_1\}&=\{1,2,\ldots, |A_1|\};\\
            \{x_i:i\in A_1\cup A_2\}&=\{1,2,\ldots, |A_1|+|A_2|\};\\
            &\vdots \\
            \{x_i:i\in A_1\cup \ldots\cup  A_{r-1}\}&=\{1,2,\ldots, |A_1|+\ldots+|A_{r-1}|\};\\
            \{x_i:i\in [n]\}&=\{1,2,\ldots, n\}.
        \end{align*}
        That is, $\vec{x}=\vec{\sigma}$ for some $\sigma\in S_n$ where $\inv{\sigma}$ is compatible with $A_{\bullet}$. It is straightforward to see that $\Pi_{A_{\bullet}}$ is combinatorially equivalent to the product of permutahedra $\Pi_{|A_1|}\times \Pi_{|A_2|}\times \ldots \times \Pi_{|A_r|}$. We have $\Pi_{B_{\bullet}}\subseteq \Pi_{A_{\bullet}}$ if and only if $B_{\bullet}$ is a refinement of $A_{\bullet}$; in particular
        $$\partial \Pi_{A_{\bullet}}= \bigcup_{B_{\bullet}} \Pi_{B_{\bullet}}=\coprod_{B_{\bullet}}\Int(\Pi_{B_{\bullet}}),$$
        where $B_{\bullet}$ ranges over all proper refinements of $A_{\bullet}$.

            Let $a=n-1$, and fix $\tau$ to be the tree
            $$\tau =
            \vcenter{
                \hbox{
                    \begin{tikzpicture}[xscale=0.5, yscale=0.8]

                        \trivialTree{$\{1,2,\ldots, n-1\}$}{$\{n\}$}
                        
                    \end{tikzpicture}
                }
            }.
            $$
            Note that, by the $(n-1)$-stability condition, every $(C; x_\bullet) \in \hassett{n-1}$ is a smooth curve, in particular irreducible. For the remainder of this section, we assume that the $(n+1)$-st marked point is fixed at $\infty$ and not explicitly mention this fact in our discussion. The closed cell $\overline{W_{\tau}} \subseteq \hassett{n-1}(\R)$ consists of the all the curves of the form $(\R\Pj^1; x_{\bullet})$ such that, for some $\sigma\in S_n$ with $\sigma(n) = n$,
            $$x_{\sigma(1)}\le x_{\sigma(2)}\le \ldots \le x_{\sigma(n)}=x_n$$
            with 
            $$\max_{i\in [n-2]} (x_{\sigma(i+1)}-x_{\sigma(i)}) \le x_{\sigma(n)}-x_{\sigma(n-1)}=1.$$
            Intuitively, the marked points indexed at $1,\ldots, n-1$ are close together in the sense that they have ``small consecutive gaps"; whereas the $n$-th marked point is ``far away" from the first $n-1$ marked points.

            We describe a particular map $\varphi: \Pi_{n-1}\to \overline{W_{\tau}}$ and show that it is a homeomorphism. We identify $S_{n-1} \subseteq S_n$ as the permutations for which $\sigma(n) = n$,
            
            and we define $\varphi$ inductively by dimension on the faces of $\Pi_{n-1}$, such that the face $\Pi_{A_\bullet}$ corresponds to the curves for which the following holds: rescaling so that $d_{n-1} = 1$,
            \begin{align}\label{eqn:permutahedron-face-condition-for-distances}
            d_i = 1 \text{ if and only if } i = n-1 \text{ or } \sigma(i), \sigma(i+1) \text{ are in different parts of } A_\bullet.
            \end{align}
            For $\sigma\in S_{n-1} \subseteq S_n$, define 
            $$\varphi(\vec{\sigma}):=(\R\Pj^1; \sigma(1), \sigma(2), \ldots, \sigma(n-1), \sigma(n) = n, \infty)$$
            to be the permutation point corresponding to $\inv{\sigma}$. This defines $\varphi$ on the vertices (0-dimensional faces) of $\Pi_{n-1}$ and satisfies \eqref{eqn:permutahedron-face-condition-for-distances}. Next, fix a composition $A_{\bullet}=(A_1, \ldots, A_r)$ of $[n-1]$ where $r<n$. Assume $\varphi$ is defined compatibly with \eqref{eqn:permutahedron-face-condition-for-distances} and continuously on the union of the faces of $\Pi_{n-1}$ of dimension less than $\dim \Pi_{A_{\bullet}}=n-r$. In particular, $\varphi|_{\partial \Pi_{A_{\bullet}}}$ is completely determined. Every $\vec{z} \in \Pi_{A_\bullet}$ is then uniquely expressible as 
            $$\vec{z}=t\vec{y}+(1-t)\vec{o}_{A_{\bullet}}$$
            for some $\vec{y}\in \partial \Pi_{A_{\bullet}}$ and some $t$, and where where $\vec{o}_{A_\bullet}$ denotes the centroid of the face. We have $\varphi(\vec{y})=(\R\Pj^1; x_{\bullet})$, where for some $\sigma\in S_{n-1} \subseteq S_n$ consistent with $A_{\bullet}$, we have
            \begin{equation}\label{eqn:sigma-marked-point-concec-ordering}
                x_{\sigma(1)}\le x_{\sigma(2)}\le \ldots \le x_{\sigma(n-1)}< x_{\sigma(n)} = x_n.
            \end{equation}
            We set $d_i:= x_{\sigma(i+1)} - x_{\sigma(i)}$ for all $i$. Up to translation and scaling, we may assume $x_{\sigma(1)} = 0$ and that $\max_i(d_i) = 1$. In particular, $d_i \leq 1$ for all $i$, and by \eqref{eqn:permutahedron-face-condition-for-distances}, $d_i = 1$ if $i = n-1$ or $\sigma(i), \sigma(i+1)$ lie in different parts of $A_\bullet$.
            
            We define $\varphi(\vec{z})$ by scaling by $t$ the successive distances internal to each part of $A_\bullet$, putting
            \begin{equation}\label{eqn:scaling-the-di-by-param-t}
                d'_i:=\begin{cases}
                    td_i, & \sigma(i), \sigma(i+1)\text{ lie in the same part of } A_{\bullet}\\
                    d_i=1, & \text{otherwise}
                \end{cases}.
            \end{equation}
            We define $\varphi(\vec{z})$ to be the stable curve $(\R\Pj^1; x_{\bullet})$ where
            $$x_{\sigma(i)} = \sum_{j=1}^{i-1} d'_j \text{ for } 1\le i\le n,$$
            considering the empty sum to be $0$.

            \begin{rmk} The map $\varphi$ is well defined.
                \begin{enumerate}
                    \item Our construction of $\varphi(\vec{z})$ from $\varphi(\vec{y})$ is independent of the choice of $\sigma$.
                    \item\label{enumitem:perm-map-to-cell-t=0case} ($t=0$ case) If $A_{\bullet}=(A_1\ldots, A_r)$ is a composition of $[n-1]$, then $\varphi(\vec{o}_{A_{\bullet}})$ is the stable curve that sets the $i$-th marked point at $j-1$ if $i\in A_j$ and the $n$-th marked point at $r$. In this case each $d'_i\in \{0,1\}$. For example, if $A_{\bullet}=(\{1,3\}, \{2,5,6\}, \{4\}, \{7,8\})$, then 
                $$\varphi(\vec{o}_{A_{\bullet}})=
                \vcenter{
                    \hbox{
                        \begin{tikzpicture}
                            \node [circle, draw, minimum size=3cm] (c1) at (0,0) {};
                            \def\labelpos{180};
                            \markedpoint{c1}{\labelpos}{0}{1,3};
                            
                            \def\labelpos{130};
                            \markedpoint{c1}{\labelpos}{1}{2,5,6};

                            \def\labelpos{70};
                            \markedpoint{c1}{\labelpos}{2}{4};

                            \def\labelpos{30};
                            \markedpoint{c1}{\labelpos}{3}{7,8};

                            \def\labelpos{-20};
                            \markedpoint{c1}{\labelpos}{4}{9};
                            \def\labelpos{270};
                            \markedpoint{c1}{\labelpos}{$\infty$}{10};
                        \end{tikzpicture}
                    }
                }
                \in \hassett{n-1}(\R)=\R\Pj^7.$$
                    \item ($t=1$ case) If $\vec{z}=\vec{y}$, then indeed, our definitions of $\varphi(\vec{z})$ and $\varphi(\vec{y})$ coincide since $d'_i=d_i$.
                \end{enumerate}
            \end{rmk}
            By \eqref{enumitem:perm-map-to-cell-t=0case}, $\varphi$ is continuous at $\vec{o}_{A_\bullet}$ and hence on all of $\Pi_{A_\bullet}$, and compatible with \eqref{eqn:permutahedron-face-condition-for-distances}. Taking $A_{\bullet}=([n-1])$ to be the trivial composition implies that $\varphi$ is continuous.

            We briefly describe the inverse map $\theta : \overline{W_\tau} \to \Pi_{n-1}$. Given a stable curve $C=(\R\Pj^1;x_{\bullet})$ from $\overline{W_\tau}$, choose $\sigma$ to satisfy \eqref{eqn:sigma-marked-point-concec-ordering}, define $d_i := x_{\sigma(i+1)} - x_{\sigma(i)}$ for $1 \leq i \leq n-1$, and normalize so that $\max d_i=1$. The vector $\vec{d}=(d_i)$ of distances of consecutively placed marked points induces a composition $A_{\bullet}$ of $[n-1]$ compatible with $\sigma$, satisfying \eqref{eqn:permutahedron-face-condition-for-distances}.

            As above, we define $\theta$ inductively according to $\dim \Pi_{A_{\bullet}}$. If every $d_i=1$, then $C=\varphi(\inv{\vec{\sigma}})$. Suppose at least one of the $d_i$ is strictly less than $1$ and define 
            $$t:=\max_{d_i<1} d_i.$$
            If $t=0$, then $C=\varphi(\vec{o}_{A_{\bullet}})$. Assume $t\in (0,1)$. Consider the stable curve $C'=(\R\Pj^1; x'_{\bullet})$, where $x'_1=0$, $x'_{n}-x'_{\sigma(n-1)}=1$, and for $2\le i\le n-1$,
            $$x'_{\sigma(i)}-x'_{\sigma(i-1)}=\min\left\{1, \frac{d_{i-1}}{t}\right\}.$$

            By the inductive hypothesis, $C'=\varphi(\vec{y})$ for some $\vec{y}\in \Pi_{n-1}$ contained in a face of dimension smaller than $\dim\Pi_{A_{\bullet}}$. By construction, $\vec{y}\in \partial \Pi_{A_{\bullet}}$. It readily follows that $C=\varphi(t\vec{o}_{A_{\bullet}}+(1-t)\vec{y})$, showing that $\varphi$ is a bijection. By compactness of the Hausdorff spaces $\Pi_{n-1}$ and $\overline{W_{\tau}}$, $\varphi$ is a homeomorphism. Our discussion proves the following result. 

            \begin{thm}\label{thm:homeo-permu-dual-cell}
                Let $a = n-1$ and let $$\tau =
            \vcenter{
                \hbox{
                    \begin{tikzpicture}[xscale=0.5, yscale=0.8]

                        \trivialTree{$\{1,2,\ldots, n-1\}$}{$\{n\}$}
                        
                    \end{tikzpicture}
                }
            }.$$
            There is a homeomorphism $\Pi_{n-1}\cong \overline{W_{\tau}}$, taking the face $\Pi_{A_\bullet}$, for each composition $A_\bullet$ of $[n]$, to the set of curves satisfying \eqref{eqn:permutahedron-face-condition-for-distances}.
            In particular, $W_{\tau}=\Int(\overline{W_{\tau}})$ is homeomorphic to the ball of dimension $n-2$.
            \end{thm}

\begin{rmk}[Star-shaped polytopes]
    With $n=4$, the cell $W_\tau$ of Theorem \ref{thm:homeo-permu-dual-cell} uses the conditions 
    \[
    \max_{i = 1, 2}(x_{\sigma(i+1)} - x_{\sigma(i)}) \leq x_{\sigma(4)} - x_{\sigma(3)},
    \]
    where $\sigma \in S_4$ is the permutation such that
    \[x_{\sigma(1)} \leq x_{\sigma(2)} \leq x_{\sigma(3)} \leq x_{\sigma(4)},\]
    and it is assumed that $\sigma(4) = 4$. This condition is in fact only star-convex around the origin, not convex. If we rescale so that $x_4 - x_{\sigma(3)} = 1$ and translate so that $x_{\sigma(1)} = 0$, then $(x_1, x_2, x_3)$ is a permutation of $(0, d_1, d_1+d_2)$ where $\max(d_1, d_2) \leq 1$. The six resulting regions form a star-shaped two-dimensional subset of the coordinate planes of the first octant of $\R^3$:
    \begin{center}
        \includegraphics[height=5cm]{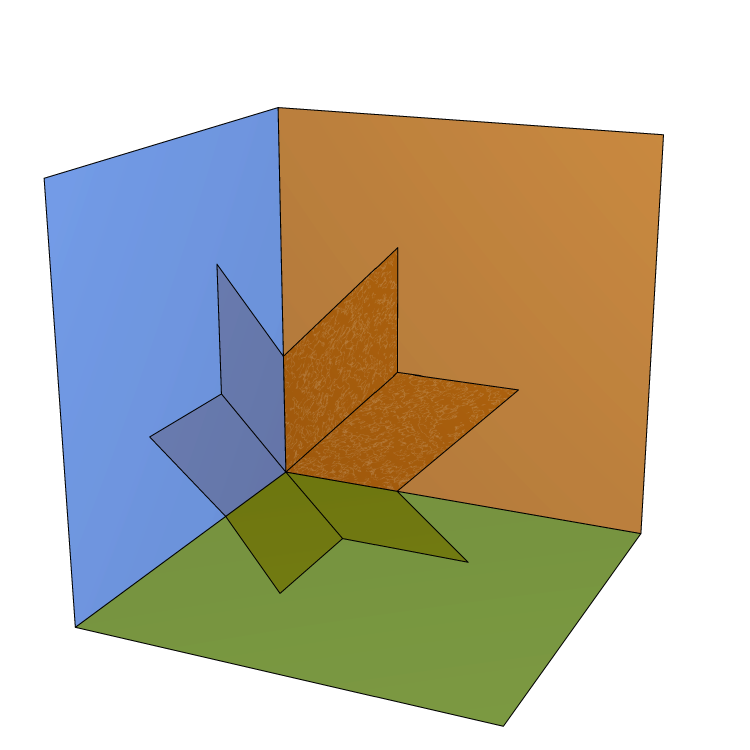}
    \end{center}
    This is why the homeomorphism $\varphi$ to the permutahedron is likewise defined in a star-shaped manner. The star polytope, and certain quotients of it, also arise in \cite{ilinKLPR2023}, e.g. see \cite[Theorem 8.2]{ilinKLPR2023}.
\end{rmk}

\begin{rmk}
    Consider a height-$1$ tree of the form
    \[\tau=
\vcenter{
    \hbox{
        \begin{tikzpicture}[xscale=1.5,yscale=1]
            \node [circle, draw, fill, inner sep =1pt, label = { 90: $A_1$ }] (node_0_0) at (0,0) {};
            \node [circle, draw, fill, inner sep =1pt, label = { 90: $A_2$ }] (node_1_0) at (1,0) {};
            \node [circle, draw, fill, inner sep =1pt, label = { 90: $A_r$ }] (node_3_0) at (3,0) {};
            \node [circle, draw, fill, inner sep =1pt] (node_1_-2) at (1.5,-2) {};
            \node [label = { 90: $\cdots$ }] at (2,0) {};
            \draw[] (node_1_-2) -- (node_0_0);
            \draw[] (node_1_-2) -- (node_1_0);
            \draw[] (node_1_-2) -- (node_3_0);
        \end{tikzpicture}
    }
}\in \Rtree([n]; a).\]
The same argument as in Theorem \ref{thm:homeo-permu-dual-cell} directly yields a homeomorphism
\[
\overline{W_\tau} \cong \prod_{i=1}^r \Pi_{A_i} \cong \Pi_{A_\bullet},
\]
which recovers the product decomposition of Proposition \ref{prop:a-proper-dual-cell-product-decmop} for this $\tau$.
\end{rmk}

            \begin{cor}
            \label{cor:wted-decomp-of-Wtau-bar}
                Let $A_{\bullet}=(A_1, A_2, \ldots, A_r)$ be a composition of $[n]$. Let $\tau\in \Rtree(A_{\bullet})$ be $a$-stable. Then
                $$\overline{W_{\tau}}\cong [-1,1]^{|E_{\Int}(\tau)|}\times \Pi_{|A_1|}\times \ldots \times \Pi_{|A_r|}.$$
                In particular, $W_{\tau}$ is a cell of dimension $|E_{\Int}(\tau)|+n-r$.
            \end{cor}

            \begin{proof}
                The result follows immediately from Proposition~\ref{prop:a-proper-dual-cell-product-decmop} and Theorem~\ref{thm:homeo-permu-dual-cell}.
            \end{proof}

            \begin{rmk}\ 
                \begin{enumerate}
                    \item If $\tau$ is a binary tree, then it has exactly $r-2$ internal edges, so $\dim W_{\tau}=\dim \overline{W_{\tau}} = n-2$, which is equal to the dimension of the Hassett space $\overline{M_{0,\mcA(a)}}(\R)$.
                    \item In general, $\codim W_{\tau} = \sum_{v\in V_{\Int}(\tau)} (c(v)-2) $, where $c(v)$ is the number of children of the vertex $v$. This coincides with the formula given in Equation~\eqref{eqn:big-cube-codim-formula}.
                \end{enumerate}
            \end{rmk}

            \begin{cor}
                The cells $W_{\tau}$ form a cell decomposition of $\overline{M_{0, \mcA(a)}}(\R)$:
                $$\overline{M_{0, \mcA(a)}}(\R)=\coprod_{\tau\in \StRtree([n]; a)}W_{\tau}.$$
            \end{cor}
As discussed in Remark \ref{rmk:properties-of-weighted-algorithm}, for each $\tau \in \StRtree([n]; a)$, the cell $W_\tau$ is the image of several cells under the reduction map from $\DMumford{n+1}(\R)$ or from $\hassett{a-1}(\R)$. This decomposition, highlighting the structure of $W_\tau$ in terms of permutahedra, can be seen in Figure \ref{fig:hexagon-merger} for $\Pi_3$ and Figure \ref{fig:illust-of-3d-permutahedron} for $\Pi_4$.

\begin{figure}
    \centering
    \includegraphics[width=0.3\linewidth]{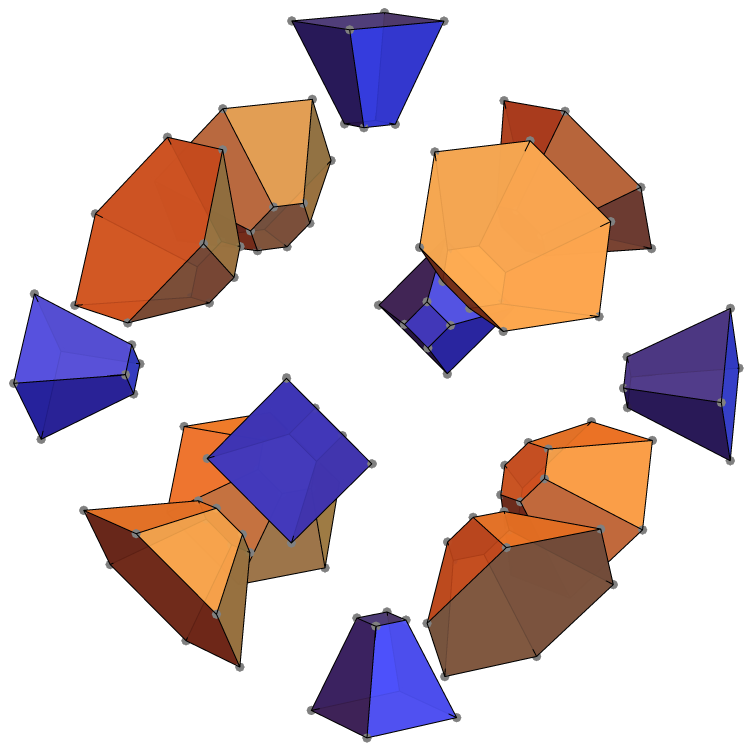} \quad
    \includegraphics[width=0.3\linewidth]{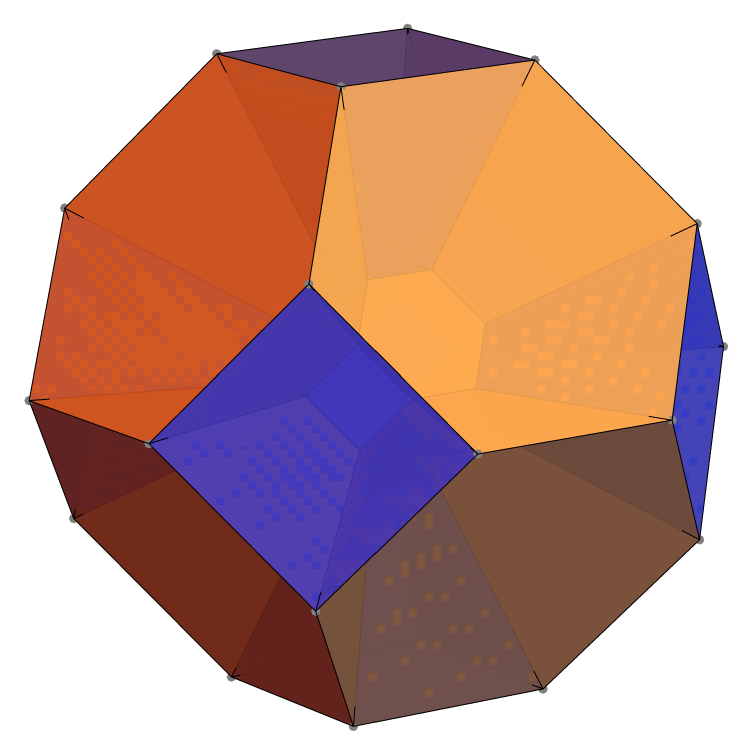}

    \caption{With $n=5$, there are three-dimensional permutahedra $\overline{W_\tau} \cong \Pi_4$ in $\DMumford{\mcA(4)}(\R) \cong \R\Pj^3$. Such a cell is the image under the reduction map of seven polytopal cells $\overline{W_{\tau'}} \subseteq \DMumford{\mcA(3)}(\R)$, shown above. Each antipodal pair of polytopes represents a single cell $\overline{W_{\tau'}}$, by gluing the inner two faces via the antipodal map. Thus, there are four copies of $[-1, 1] \times \Pi_3$ (hexagonal prisms, in orange) and three copies of $[-1, 1] \times \Pi_2 \times \Pi_2$ (cubes, in blue). The inner boundary is homeomorphic to $\mathbb{RP}^2$ (the exceptional divisor) and the reduction map contracts it to a point to yield $\overline{W_\tau} \cong \Pi_4$. Note that each hexagonal prism is itself the image of three cubes from $\DMumford{n+1}(\R)$; that decomposition resembles Figure \ref{fig:hexagon-merger} (crossed with $\Pi_2$).}
    \label{fig:illust-of-3d-permutahedron}
\end{figure}

    \begin{rmk}
        For an arbitrary vector of weights $\mcA = (a_1, \ldots, a_n, a_{n+1} = 1)$, there is an analogous procedure for compressing rooted trees from $\Rtree([n])$ to render them $\mcA$-stable. Likewise, there is an analogous $\mcA$-weighted distance algorithm, given by running the standard distance algorithm and compressing the result to an $\mcA$-stable tree. The results of this section then carry through to show that the $\mcA$-stable trees index a dual cell decomposition of $\overline{M_{0, \mcA}}(\R)$ by products of permutahedra and intervals. Indeed, similar remarks apply to the simplicially-stable moduli spaces of \cite{blankersB2022} containing an isolated vertex. However, the resulting decomposition is typically not $S_n$-symmetric, so we do not know a useful way to extract a description of $\pi_1(\overline{M_{0, \mcA}}(\R))$ in general.
    \end{rmk}

\subsubsection{The $2$-skeleton and presentation of the weighted cactus group} \label{sec:2-skeleton-and-presentation}
    We now classify the dual cells $W_{\tau}$ up to dimension two such that $\overline{W_{\tau}}$ contains the permutation point $(\R\Pj^1; 1,2,\ldots, n, \infty)$. If $\tau$ has leaf labels $A_1, \ldots, A_r$, this condition is equivalent to requiring that $A_\bullet$ be compatible with the identity permutation.
        \begin{itemize}
            \item(Dimension 0) The $0$-cells are precisely the permutation points. We note that the base point is itself a $0$-cell.
            \item(Dimension 1) There are two types of 1-cell adjacent to the base point.
            \begin{enumerate}
                \item\label{enumitem:1cell-type-no-Eint}
                 $A_\bullet$ has one part $\{i,i+1\}$ of size $2$ and all other parts singletons, and $\tau$ has no internal edges. The other endpoint of $\overline{W_{\tau}}$ is the permutation point corresponding to $w_{i,i+1}\in S_n$.
                \item\label{enumitem:1cell-type-one-Eint} $A_\bullet = (\{1\}, \ldots, \{n\})$ and $\tau$ has one internal edge, whose subtree contains the leaves $\{p\},\ldots,\{q\}$ for some $p < q$. By $a$-stability, we must have $q+1-p>a$.  The other endpoint of $\overline{W_{\tau}}$ is the permutation point corresponding to $w_{p,q}\in S_n$.
            \end{enumerate}
            \item(Dimension 2) There are four kinds of $2$-cell adjacent to the base point. We list them according to their combinatorial types. See Figure \ref{fig:2-cell-pictures} for representative examples.
            
            \begin{enumerate}
                \item \label{enumitem:2cell-1part-size3} (Hexagon, $\Pi_3$): $A_\bullet$ has one part of size $3$, $\{i, i+1, i+2\}$ for some $i$, and $\tau$ has no internal edges.
                \item \label{enumitem:2cell-2part-size2}
                (Square, $\Pi_2 \times \Pi_2$): $A_\bullet$ has two parts of size 2, $\{i,i+1\}$ and $\{j,j+1\}$ with $i+1<j$, and $\tau$ has no internal edges.
                \item \label{enumitem:2cell-1part-size2} 
                (Square, $[-1, 1] \times \Pi_2$): $A_\bullet$ has one part of size $2$, $\{i, i+1\}$, and $\tau$ has one internal edge, whose subtree contains the labels $p, \ldots, q$, which may or may not include the leaf labeled $\{i, i+1\}$. By $a$-stability, we must have $q+1-p > a$.
                \item \label{enumitem:2cell-all-part-singleton}
                (Square, $[-1, 1]^2$): $A_{\bullet} = (\{1\}, \ldots, \{n\})$ and $\tau$ has two internal edges, which correspond to either nested or disjoint subintervals of leaves (similar to the unweighted case); see Figure \ref{fig:2-cell-trees-of-moduli}. By $a$-stability, the subintervals must have length $> a$.
            \end{enumerate}
        
        \end{itemize}

\begin{figure}
\centering
\begin{tabular}{cc}
\begin{tikzpicture}
    \def \radius {3.3}
    \fill[cyan!40!lime!20] (0:\radius) -- (60:\radius) -- (120:\radius) -- (180:\radius) -- (240:\radius) -- (300:\radius) -- cycle;
    
    \draw (\radius, 0) node {$\bullet$}
    	-- (60:\radius) node {$\bullet$}
	-- (120:\radius) node {$\bullet$}
	-- (180:\radius) node {$\bullet$}
	-- (240:\radius) node {$\bullet$}
	-- (300:\radius) node {$\bullet$}
	-- (360:\radius) node {$\bullet$};

    \draw (0, -0.5) -- (-0.5, 0) node[above] {$123$};
    \draw (0, -0.5) -- (0.5, 0) node[above] {$4$};

   \begin{scope}[shift={(120:\radius*1.05)}, scale=0.6]
   \foreach \x/\i in {-0.6/1, -0.2/2, 0.2/3, 0.6/4}
    	\draw (0, 0) -- (\x, 0.5) node[above] {$\i$};
   \end{scope}

   \begin{scope}[shift={(150:\radius*0.9)}]
   \foreach \x/\i in {-0.6/12, 0/3, 0.6/4}
    	\draw (0, -0.5) -- (\x, 0) node[above] {$\i$};
   \end{scope}
   \begin{scope}[shift={(210:\radius*0.9)}]
   \foreach \x/\i in {-0.6/2, 0/13, 0.6/4}
    	\draw (0, -0.5) -- (\x, 0) node[above] {$\i$};
   \end{scope}
   \begin{scope}[shift={(270:\radius*0.8)}]
   \foreach \x/\i in {-0.6/23, 0/1, 0.6/4}
    	\draw (0, -0.5) -- (\x, 0) node[above] {$\i$};
   \end{scope}
   \begin{scope}[shift={(330:\radius*0.9)}]
   \foreach \x/\i in {-0.6/3, 0/12, 0.6/4}
    	\draw (0, -0.5) -- (\x, 0) node[above] {$\i$};
   \end{scope}
   \begin{scope}[shift={(30:\radius*0.9)}]
   \foreach \x/\i in {-0.6/13, 0/2, 0.6/4}
    	\draw (0, -0.5) -- (\x, 0) node[above] {$\i$};
   \end{scope}
   \begin{scope}[shift={(90:\radius*0.95)}]
   \foreach \x/\i in {-0.6/1, 0/23, 0.6/4}
    	\draw (0, -0.5) -- (\x, 0) node[above] {$\i$};
   \end{scope}
\end{tikzpicture}

&

\begin{tikzpicture}
    \def \radius {2.7}
    \fill[cyan!20] (\radius, \radius) -- (\radius, -\radius) -- (-\radius, -\radius) -- (-\radius, \radius) -- cycle;

    \draw  (\radius, \radius)   node {$\bullet$}
        -- (\radius, -\radius)  node {$\bullet$}
        -- (-\radius, -\radius) node {$\bullet$}
        -- (-\radius, \radius)  node {$\bullet$}
        -- cycle;

    \draw (0, -0.5) -- (-0.5, 0) node[above] {$12$};
    \draw (0, -0.5) -- (0.5, 0) node[above] {$34$};

   \begin{scope}[shift={(-\radius*1.05, \radius*1.05)}, scale=0.6]
   \foreach \x/\i in {-0.6/1, -0.2/2, 0.2/3, 0.6/4}
    	\draw (0, 0) -- (\x, 0.5) node[above] {$\i$};
   \end{scope}

   \begin{scope}[shift={(-\radius+0.25, 0)}]
   \foreach \x/\i in {-0.6/12, 0/3, 0.6/4}
    	\draw (0, -0.5) -- (\x, 0) node[above] {$\i$};
   \end{scope}
   \begin{scope}[shift={(0, -\radius+0.25)}]
   \foreach \x/\i in {-0.6/2, 0/1, 0.6/34}
    	\draw (0, -0.5) -- (\x, 0) node[above] {$\i$};
   \end{scope}
   \begin{scope}[shift={(\radius-0.25, 0)}]
   \foreach \x/\i in {-0.6/12, 0/4, 0.6/3}
    	\draw (0, -0.5) -- (\x, 0) node[above] {$\i$};
   \end{scope}
   \begin{scope}[shift={(0, \radius+0.25)}]
   \foreach \x/\i in {-0.6/1, 0/2, 0.6/34}
    	\draw (0, -0.5) -- (\x, 0) node[above] {$\i$};
   \end{scope}
\end{tikzpicture}

\\

\begin{tikzpicture}
    \def \radius {2.7}
    \fill[yellow!20] (\radius, \radius) -- (\radius, -\radius) -- (-\radius, -\radius) -- (-\radius, \radius) -- cycle;

    \draw  (\radius, \radius)   node {$\bullet$}
        -- (\radius, -\radius)  node {$\bullet$}
        -- (-\radius, -\radius) node {$\bullet$}
        -- (-\radius, \radius)  node {$\bullet$}
        -- cycle;

   \draw (0, -0.5) -- (-0.3, -0.3);
   \draw (-0.3, -0.3) -- (-0.6, 0) node[above] {$12$};
   \draw (-0.3, -0.3) -- (0, 0) node[above] {$3$};
   \draw (0, -0.5) -- (0.6, 0) node[above] {$4$};

   \begin{scope}[shift={(-\radius*1.05, \radius*1.05)}, scale=0.6]
   \foreach \x/\i in {-0.6/1, -0.2/2, 0.2/3, 0.6/4}
    	\draw (0, 0) -- (\x, 0.5) node[above] {$\i$};
   \end{scope}

   \begin{scope}[shift={(-\radius+0.15, 0)}]
   \draw (0, -0.5) -- (-0.3, -0.3);
   \draw (-0.3, -0.3) -- (-0.6, 0) node[above] {$1$};
   \draw (-0.3, -0.3) -- (-0.3, 0) node[above] {$2$};
   \draw (-0.3, -0.3) -- (0, 0) node[above] {$3$};
   \draw (0, -0.5) -- (0.6, 0) node[above] {$4$};
   \end{scope}
   \begin{scope}[shift={(0, -\radius+0.25)}]
   \foreach \x/\i in {-0.6/3, 0/12, 0.6/4}
    	\draw (0, -0.5) -- (\x, 0) node[above] {$\i$};
   \end{scope}
   \begin{scope}[shift={(\radius-0.25, 0)}]
   \draw (0, -0.5) -- (-0.3, -0.3);
   \draw (-0.3, -0.3) -- (-0.6, 0) node[above] {$2$};
   \draw (-0.3, -0.3) -- (-0.3, 0) node[above] {$1$};
   \draw (-0.3, -0.3) -- (0, 0) node[above] {$3$};
   \draw (0, -0.5) -- (0.6, 0) node[above] {$4$};
   \end{scope}
   \begin{scope}[shift={(0, \radius+0.25)}]
   \foreach \x/\i in {-0.6/12, 0/3, 0.6/4}
    	\draw (0, -0.5) -- (\x, 0) node[above] {$\i$};
   \end{scope}
\end{tikzpicture}

&

\begin{tikzpicture}
    \def \radius {2.7}
    \fill[yellow!20] (\radius, \radius) -- (\radius, -\radius) -- (-\radius, -\radius) -- (-\radius, \radius) -- cycle;

    \draw  (\radius, \radius)   node {$\bullet$}
        -- (\radius, -\radius)  node {$\bullet$}
        -- (-\radius, -\radius) node {$\bullet$}
        -- (-\radius, \radius)  node {$\bullet$}
        -- cycle;

   \draw (0, -0.5) -- (-0.3, -0.3);
   \draw (-0.3, -0.3) -- (-0.6, 0) node[above] {$1$};
   \draw (-0.3, -0.3) -- (-0.3, 0) node[above] {$2$};
   \draw (-0.3, -0.3) -- (0, 0) node[above] {$3$};
   \draw (0, -0.5) -- (0.6, 0) node[above] {$45$};

   \begin{scope}[shift={(-\radius*1.05, \radius*1.05)}, scale=0.6]
   \foreach \x/\i in {-0.8/1, -0.4/2, 0/3, 0.4/4, 0.8/5}
    	\draw (0, 0) -- (\x, 0.5) node[above] {$\i$};
   \end{scope}

   \begin{scope}[shift={(-\radius+0.15, 0)}]
   \draw (0, -0.5) -- (-0.3, -0.3);
   \draw (-0.3, -0.3) -- (-0.6, 0) node[above] {$1$};
   \draw (-0.3, -0.3) -- (-0.3, 0) node[above] {$2$};
   \draw (-0.3, -0.3) -- (0, 0) node[above] {$3$};
   \draw (0, -0.5) -- (0.35, 0) node[above] {$4$};
   \draw (0, -0.5) -- (0.65, 0) node[above] {$5$};
   \end{scope}
   \begin{scope}[shift={(0, -\radius+0.25)}]
   \foreach \x/\i in {-0.6/3, -0.2/2, 0.2/1, 0.6/45}
    	\draw (0, -0.5) -- (\x, 0) node[above] {$\i$};
   \end{scope}
   \begin{scope}[shift={(\radius-0.15, 0)}]
   \draw (0, -0.5) -- (-0.3, -0.3);
   \draw (-0.3, -0.3) -- (-0.6, 0) node[above] {$1$};
   \draw (-0.3, -0.3) -- (-0.3, 0) node[above] {$2$};
   \draw (-0.3, -0.3) -- (0, 0) node[above] {$3$};
   \draw (0, -0.5) -- (0.35, 0) node[above] {$5$};
   \draw (0, -0.5) -- (0.65, 0) node[above] {$4$};
   \end{scope}
   \begin{scope}[shift={(0, \radius+0.25)}]
   \foreach \x/\i in {-0.6/1, -0.2/2, 0.2/3, 0.6/45}
    	\draw (0, -0.5) -- (\x, 0) node[above] {$\i$};
   \end{scope}
\end{tikzpicture}

\\

\begin{tikzpicture}
    \def \radius {2.7}
    \fill[magenta!20] (\radius, \radius) -- (\radius, -\radius) -- (-\radius, -\radius) -- (-\radius, \radius) -- cycle;

    \draw  (\radius, \radius)   node {$\bullet$}
        -- (\radius, -\radius)  node {$\bullet$}
        -- (-\radius, -\radius) node {$\bullet$}
        -- (-\radius, \radius)  node {$\bullet$}
        -- cycle;

   \draw (0, -0.6) -- (-0.2, -0.4);
   \draw (-0.2, -0.4) -- (-0.3, -0.2);
   \draw (-0.3, -0.2) -- (-0.6, 0) node[above] {$1$};
   \draw (-0.3, -0.2) -- (-0.3, 0) node[above] {$2$};
   \draw (-0.3, -0.2) -- (0, 0) node[above] {$3$};
   \draw (-0.2, -0.4) -- (0.4, 0) node[above] {$4$};
   \draw (0, -0.6) -- (0.8, 0) node[above] {$5$};

   \begin{scope}[shift={(-\radius*1.05, \radius*1.05)}, scale=0.6]
   \foreach \x/\i in {-0.8/1, -0.4/2, 0/3, 0.4/4, 0.8/5}
    	\draw (0, 0) -- (\x, 0.5) node[above] {$\i$};
   \end{scope}

   \begin{scope}[shift={(-\radius+0.15, 0)}]
   \draw (0, -0.5) -- (-0.3, -0.3);
   \draw (-0.3, -0.3) -- (-0.6, 0) node[above] {$1$};
   \draw (-0.3, -0.3) -- (-0.3, 0) node[above] {$2$};
   \draw (-0.3, -0.3) -- (0, 0) node[above] {$3$};
   \draw (0, -0.5) -- (0.35, 0) node[above] {$4$};
   \draw (0, -0.5) -- (0.65, 0) node[above] {$5$};
   \end{scope}
   \begin{scope}[shift={(0, -\radius+0.25)}]
   \draw (0, -0.6) -- (-0.2, -0.4);
   \draw (-0.2, -0.4) -- (-0.6, 0) node[above] {$3$};
   \draw (-0.2, -0.4) -- (-0.3, 0) node[above] {$2$};
   \draw (-0.2, -0.4) -- (0, 0) node[above] {$1$};
   \draw (-0.2, -0.4) -- (0.3, 0) node[above] {$4$};
   \draw (0, -0.6) -- (0.8, 0) node[above] {$5$};
   \end{scope}
   \begin{scope}[shift={(\radius-0.15, 0)}]
   \draw (0, -0.5) -- (-0.3, -0.3);
   \draw (-0.3, -0.3) -- (-0.6, 0) node[above] {$1$};
   \draw (-0.3, -0.3) -- (-0.3, 0) node[above] {$2$};
   \draw (-0.3, -0.3) -- (0, 0) node[above] {$3$};
   \draw (0, -0.5) -- (0.35, 0) node[above] {$5$};
   \draw (0, -0.5) -- (0.65, 0) node[above] {$4$};
   \end{scope}
   \begin{scope}[shift={(0, \radius+0.25)}]
   \draw (0, -0.6) -- (-0.2, -0.4);
   \draw (-0.2, -0.4) -- (-0.6, 0) node[above] {$1$};
   \draw (-0.2, -0.4) -- (-0.3, 0) node[above] {$2$};
   \draw (-0.2, -0.4) -- (0, 0) node[above] {$3$};
   \draw (-0.2, -0.4) -- (0.3, 0) node[above] {$4$};
   \draw (0, -0.6) -- (0.8, 0) node[above] {$5$};
   \end{scope}
\end{tikzpicture}

&

\begin{tikzpicture}
    \def \radius {2.7}
    \fill[magenta!20] (\radius, \radius) -- (\radius, -\radius) -- (-\radius, -\radius) -- (-\radius, \radius) -- cycle;

    \draw  (\radius, \radius)   node {$\bullet$}
        -- (\radius, -\radius)  node {$\bullet$}
        -- (-\radius, -\radius) node {$\bullet$}
        -- (-\radius, \radius)  node {$\bullet$}
        -- cycle;

   \draw (0, -0.5) -- (-0.4, -0.3);
   \draw (-0.4, -0.3) -- (-0.65, 0) node[above] {$1$};
   \draw (-0.4, -0.3) -- (-0.4, 0) node[above] {$2$};
   \draw (-0.4, -0.3) -- (-0.15, 0) node[above] {$3$};
   \draw (0, -0.5) -- (0.4, -0.3);
   \draw (0.4, -0.3) -- (0.15, 0) node[above] {$4$};
   \draw (0.4, -0.3) -- (0.4, 0) node[above] {$5$};
   \draw (0.4, -0.3) -- (0.65, 0) node[above] {$6$};

   \begin{scope}[shift={(-\radius*1.05, \radius*1.05)}, scale=0.6]
   \foreach \x/\i in {-1/1, -0.6/2, -0.2/3, 0.2/4, 0.6/5, 1/6}
    	\draw (0, 0) -- (\x, 0.5) node[above] {$\i$};
   \end{scope}

   \begin{scope}[shift={(-\radius+0.28, 0)}]
   \draw (0, -0.5) -- (-0.4, -0.3);
   \draw (-0.4, -0.3) -- (-0.65, 0) node[above] {$1$};
   \draw (-0.4, -0.3) -- (-0.4, 0) node[above] {$2$};
   \draw (-0.4, -0.3) -- (-0.15, 0) node[above] {$3$};
   \draw (0, -0.5) -- (0.15, 0) node[above] {$4$};
   \draw (0, -0.5) -- (0.4, 0) node[above] {$5$};
   \draw (0, -0.5) -- (0.65, 0) node[above] {$6$}; \end{scope}
   \begin{scope}[shift={(0, -\radius+0.2)}]
   \draw (0, -0.5) -- (-0.65, 0) node[above] {$3$};
   \draw (0, -0.5) -- (-0.4, 0) node[above] {$2$};
   \draw (0, -0.5) -- (-0.15, 0) node[above] {$1$};
   \draw (0, -0.5) -- (0.4, -0.3);
   \draw (0.4, -0.3) -- (0.15, 0) node[above] {$4$};
   \draw (0.4, -0.3) -- (0.4, 0) node[above] {$5$};
   \draw (0.4, -0.3) -- (0.65, 0) node[above] {$6$};
   \end{scope}
   \begin{scope}[shift={(\radius-0.28, 0)}]
   \draw (0, -0.5) -- (-0.4, -0.3);
   \draw (-0.4, -0.3) -- (-0.65, 0) node[above] {$1$};
   \draw (-0.4, -0.3) -- (-0.4, 0) node[above] {$2$};
   \draw (-0.4, -0.3) -- (-0.15, 0) node[above] {$3$};
   \draw (0, -0.5) -- (0.15, 0) node[above] {$6$};
   \draw (0, -0.5) -- (0.4, 0) node[above] {$5$};
   \draw (0, -0.5) -- (0.65, 0) node[above] {$4$};
   \end{scope}
   \begin{scope}[shift={(0, \radius+0.2)}]
   \draw (0, -0.5) -- (-0.65, 0) node[above] {$1$};
   \draw (0, -0.5) -- (-0.4, 0) node[above] {$2$};
   \draw (0, -0.5) -- (-0.15, 0) node[above] {$3$};
   \draw (0, -0.5) -- (0.4, -0.3);
   \draw (0.4, -0.3) -- (0.15, 0) node[above] {$4$};
   \draw (0.4, -0.3) -- (0.4, 0) node[above] {$5$};
   \draw (0.4, -0.3) -- (0.65, 0) node[above] {$6$};
   \end{scope}
\end{tikzpicture}             
\end{tabular}
        \caption{Dual $a$-stable $2$-cells.
        Top left: $\Pi_3$ ($a=3)$.
        Top right: $\Pi_2 \times \Pi_2$ ($a=2$).
        Second row: $[-1, 1] \times \Pi_2$ ($a=2$).
        Third row: $[-1, 1]^2$ ($a=2$).
        }
        \label{fig:2-cell-pictures}
    \end{figure}
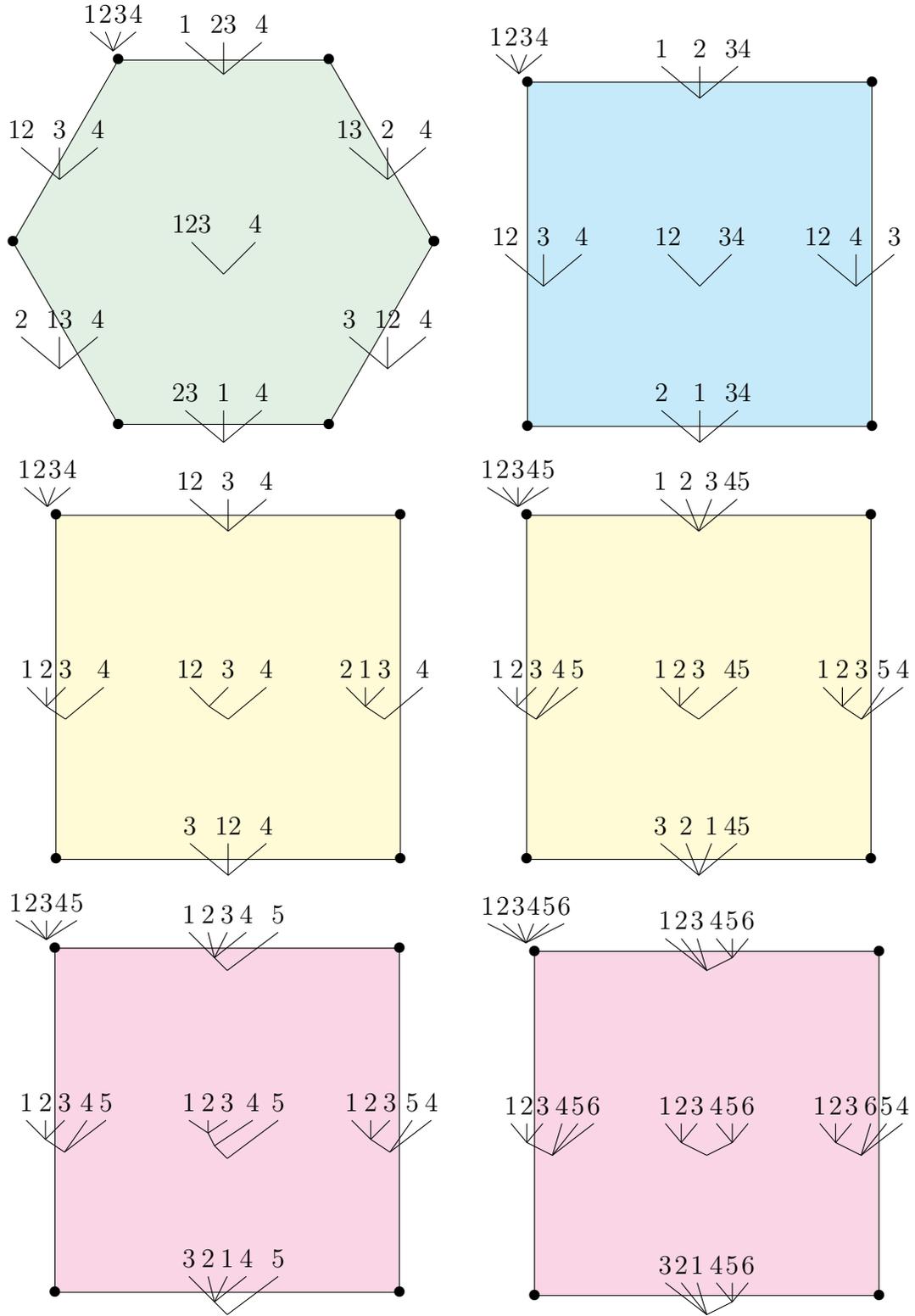

    \begin{rmk}
        As in the unweighted case, the cell decomposition of $\hassett{a}(\R)$ lifts to the double cover $\hassettDC{a}(\R)$, with dual cells indexed by $a$-stable trees, but where the root vertex is never considered flippable. The preimage of the dual cell $W_{\tau}\subseteq \hassettDC{a}(\R)$ is thus the disjoint union of two cells that differ by flipping at the root. The remaining details are essentially the same as in the unweighted case discussed in \ref{subsec:cubes-in-2:1-cover-unweighted}, and will be omitted here.
    \end{rmk}        

        \begin{proof}[Proof of Theorem~\ref{thm:main-thm-pres-of-Jna}]     
        Recall that the $S_n$-action on $\overline{\DC{M}_{0,n+1}}(\R)$ is simply transitive on the dual $0$-cells. This action descends to an $S_n$-action on $\overline{\DC{M}_{0,\mcA(a)}}(\R)$ simply transitive on its dual $0$-cells. Hence, an argument similar to the one given in \ref{subsec:cubes-in-2:1-cover-unweighted} and Theorem~\ref{thm:pres-pi1GX-simply-transitive} gives a presentation for $\tilde{J}_n^a$. Proposition~\ref{prop:group-pres-extend-by-Zmod2} then yields the desired presentation of $J_n^a$.

        It remains to obtain the relations given in Theorem~\ref{thm:main-thm-pres-of-Jna} (after omitting the $s_{1,n}$ generator) from the 2-skeleton of the Hassett space. The $1$-cells in the Hassett space contain pairs of endpoints of the form $\{x_0, w_{p,q}x_0\}$ for some 0-cell $x_0$, where either $p=q+1$ or $a<q-p+1<n$. Since $w_{p,q}$ is an involution, we may unambiguously label the $1$-cell with a formal symbol $s_{p,q}$, and we have $s_{p, q}^2 = 1$ in $\tilde{J}_n^a$.
        
        The remaining relations come from the boundaries of the $2$-cells adjacent to the base point. It is straightforward to verify that the 2-cells of types \eqref{enumitem:2cell-2part-size2}, \eqref{enumitem:2cell-1part-size2}, and \eqref{enumitem:2cell-all-part-singleton} yield the commuting relations $s_{p,q}s_{m,r}=s_{m,r}s_{p,q}$ and the cactus relations $s_{p,q}s_{m,r}=s_{p+q-r,p+q-m}s_{p,q}$. Finally, those of type \eqref{enumitem:2cell-1part-size3} yield the braid relations $(s_{i,i+1}s_{i+1,i+2})^3=1$. (See Figure \ref{fig:hexagon-merger-2}.) Theorem~\ref{thm:main-thm-pres-of-Jna} now follows from Proposition~\ref{prop:group-pres-extend-by-Zmod2}.
        \end{proof}
        
        \begin{rmk}
            If $a=1$, then all the $2$-cells are of type \eqref{enumitem:2cell-all-part-singleton}; and if $a=2$, then the cells are all of types \eqref{enumitem:2cell-2part-size2}, \eqref{enumitem:2cell-1part-size2}, and \eqref{enumitem:2cell-all-part-singleton}. In either case, we recover the presentations of $J_n$ and $\tilde{J}_n$.
        \end{rmk}

    \section{
    Generalized braid relations
    }\label{sec:group-weighted-cactus}

    As discussed in the introduction, the reduction map 
    $$\DMumford{n+1}(\R)\onto \DMumford{\mcA}(\R)$$
    induces a surjection $J_n\onto J_n^a$ of $S_n$-equivariant fundamental groups; see Equation \eqref{eqn:surjection-of-cactus-groups}. We conclude with a brief discussion of the kernel of this map. In particular, when $a \geq 3$, Theorem \ref{thm:main-thm-pres-of-Jna} gives a presentation of $J_n^a$ using a subset of the generators of $J_n$, namely the $s_{p,q}$ for which $q-p = 1$ or $q-p \in \{a, a+1, \ldots, n-1\}$. The remaining generators yield certain words in the kernel of the quotient map, which we refer to as the \emph{generalized braid relations}. It is interesting to observe that these relations are, themselves, straightforward to establish by a geometric argument, essentially because the $p$-th through $q$-th marked points can collide without bubbling.

            \begin{lem}\label{lem:gen-braid-rels}
               Assume $a\ge 3$. Write $\sigma_i = s_{i,i+1}$ and $s_{i,i} = 1$. For $2 \le q-p \le a-1$, we have the \textbf{generalized braid relations}
               \[
               s_{p,q}s_{p, q-1}\sigma_{q-1} \hdots \sigma_p, \quad
               s_{p,q}\sigma_{q-1} \hdots \sigma_p s_{p+1,q}
               \in \ker(J_n \twoheadrightarrow J_n^a).\]
            \end{lem}
            \noindent When $q=p+2$, the generalized braid relations reduce to the classical braid relations,
            \[\sigma_p\sigma_{p+1}\sigma_p = \sigma_{p+1}\sigma_p\sigma_{p+1},\]
            upon left cancellation of $s_{p,p+2}$.
            
            \begin{proof}
            
            Let $\alpha = s_{p,q}s_{p, q-1}\sigma_{q-1}\hdots\sigma_p$ and $\beta=s_{p,q}\sigma_{q-1}\hdots\sigma_ps_{p+1,q}$. It is elementary to check that $\alpha, \beta$ both map to $(1)\in S_n$, and so correspond to loops in $\overline{M_{0,\mcA}}(\R)$. It remains to show that the loops $\alpha$ and $\beta$ are contractible.

            Fix the marked points $x_1,\ldots, x_p$, $x_{q+1}, \ldots, x_{n+1}$ at distinct locations on $\R\Pj^1$. By $a$-stability, the map $\iota: \A^{q-p}(\R)\into \overline{M_{0, \mcA}}(\R)$ given by $(x_{p+1}, \ldots,  x_{q})\mapsto (\R\Pj^1; x_{\bullet})$ is well-defined and an embedding. Up to homotopy, for $p\le m<r\le q$, $s_{m,r}$ represents a path in the Hassett space that lies completely in the image of $\iota$, regardless of the fixed starting point in this image. Thus, $\alpha$ and $\beta$ represent loops which lie completely in the image of $\iota$. Since $\A^{q-p}(\R)$ is simply connected, the loops $\alpha$ and $\beta$ are contractible.
        \end{proof}
        In other words, it is straightforward to show that the weighted cactus group $J_n^a$, defined abstractly as in Theorem \ref{thm:main-thm-pres-of-Jna}, surjects onto $\pi_1^{S_n}(\hassett{a}(\R))$. We do not, however, know of a simple proof that this map is injective, i.e. that these relations generate the full kernel as a normal subgroup, without the construction of a cell decomposition.

        We illustrate the effect of the reduction map in dimension $2$ in Figure \ref{fig:hexagon-merger-2}, with the aid of our cell decomposition. In general, $s_{p,q}$ reduces to a path across the center of a permutahedron of dimension $q-p$.

        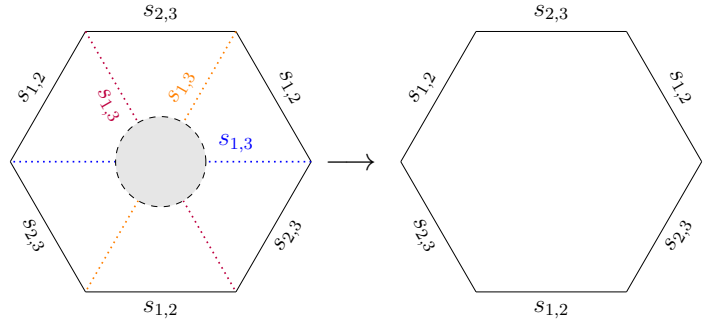
\begin{figure}
            \centering
            $$
            \vcenter{
            \hbox{
            \scalebox{0.8}{
            \begin{tikzpicture}[scale=2.5] 
                \def\r{0.3}
                \def\Xfactor{1.25}
                
                \coordinate[] (c0) at (1,0) {};
                \coordinate[] (c1) at (1/2,{sqrt(3)/2}) {};
                \coordinate[] (c2) at (-1/2,{sqrt(3)/2}) {};
                \coordinate[] (c3) at (-1,0) {};
                \coordinate[] (c4) at (-1/2,-{sqrt(3)/2}) {};
                \coordinate[] (c5) at (1/2,-{sqrt(3)/2}) {};
            
                \draw[] (c0) -- node[midway, sloped, above]{$s_{1,2}$} (c1);
                \draw[] (c1) -- node[midway, sloped, above]{$s_{2,3}$} (c2);
                \draw[] (c2) -- node[midway, sloped, above]{$s_{1,2}$} (c3);
                \draw[] (c3) -- node[midway, sloped, below]{$s_{2,3}$} (c4);
                \draw[] (c4) -- node[midway, sloped, below]{$s_{1,2}$} (c5);
                \draw[] (c5) -- node[midway, sloped, below]{$s_{2,3}$} (c0);
            
                \draw[dotted,blue,thick] (c0) -- node[sloped, above]{$s_{1,3}$} ++(c3) -- node[sloped, above]{} (c3);
                \draw[dotted, orange, thick] (c1) -- node[sloped, above]{$s_{1,3}$} ++(c4) -- node[sloped, above]{} (c4);
                \draw[dotted,purple,thick] (c2) -- node[sloped, below]{$s_{1,3}$} ++(c5) -- node[sloped, above]{} (c5);
            
                \draw[color=black, dashed, fill=gray!20] (0,0) circle (\r);
            \end{tikzpicture}
            }
            }}
            {\longrightarrow}
            \vcenter{
            \hbox{
            \scalebox{0.8}{
            \begin{tikzpicture}[scale=2.5] 
                \def\r{0.3}
                \def\Xfactor{1.25}
                
                \coordinate[] (c0) at (1,0) {};
                \coordinate[] (c1) at (1/2,{sqrt(3)/2}) {};
                \coordinate[] (c2) at (-1/2,{sqrt(3)/2}) {};
                \coordinate[] (c3) at (-1,0) {};
                \coordinate[] (c4) at (-1/2,-{sqrt(3)/2}) {};
                \coordinate[] (c5) at (1/2,-{sqrt(3)/2}) {};
            
                \draw[] (c0) -- node[midway, sloped, above]{$s_{1,2}$} (c1);
                \draw[] (c1) -- node[midway, sloped, above]{$s_{2,3}$} (c2);
                \draw[] (c2) -- node[midway, sloped, above]{$s_{1,2}$} (c3);
                \draw[] (c3) -- node[midway, sloped, below]{$s_{2,3}$} (c4);
                \draw[] (c4) -- node[midway, sloped, below]{$s_{1,2}$} (c5);
                \draw[] (c5) -- node[midway, sloped, below]{$s_{2,3}$} (c0);
            \end{tikzpicture}
            } 
            }}
            $$
            \caption{A local picture of the blowdown map $\DMumford{5}(\R)\onto \hassett{3}(\R) \cong \R\Pj^2$. Left: The central circle is an $\R\Pj^1$ with antipodal points identified, surrounded by three regions that are topological squares. Right: Merging regions and contracting the exceptional $\R\Pj^1$ gives a hexagon $\Pi_3$, in which the path $s_{1,3}$ becomes homotopic to $s_{1,2}s_{2,3}s_{1,2}$.}
            \label{fig:hexagon-merger-2}
        \end{figure}

    \section*{References}
    \printbibliography[title={\ }, heading=none]

\end{document}